\documentclass[12pt]{article}

\usepackage{latexsym}
\usepackage{a4wide}
\usepackage{amscd}
\usepackage{graphics}
\usepackage{amsmath}
\usepackage{amssymb}
\usepackage{mathrsfs}
\usepackage{amsthm}
\usepackage{bbm}
\usepackage{varioref}
\usepackage{color}
\usepackage{accents}
\usepackage{enumerate}
\usepackage{url}
\input xy
\xyoption{all}

\newcommand{\N}{\mathbb{N}}                     
\newcommand{\Z}{\mathbb{Z}}                     
\newcommand{\R}{\mathbb{R}}                     
\newcommand{\C}{\mathbb{C}}                     
\newcommand{\T}{\mathbb{T}}                     
\newcommand{\set}[2]{\left\{{#1}\mid{#2}\right\}}       
\newcommand{\then}{\Rightarrow}                 
\newcommand{\im}{\mathrm{Im\,}}                 
\newcommand{\codim}{\mathrm{codim}}           
\newcommand{\crit}{\mathrm{crit}\,}               
\newcommand{\A}{\mathbb{A}}

\newtheorem{mainthm}{\sc Theorem}           
\newtheorem{thm}{\sc Theorem}[section]               
\newtheorem*{thm*}{\sc Theorem}               
\newtheorem*{cor*}{\sc Corollary}        
\newtheorem{lem}[thm]{\sc Lemma}            
\newtheorem{prop}[thm]{\sc Proposition}     
\newtheorem{defn}[thm]{\sc Definition}      
\newtheorem{rem}[thm]{\sc Remark}           
\newtheorem*{rem*}{\sc Remark}           

\title{Floer homology on the time-energy extended phase space}

\author{Alberto Abbondandolo and Will J. Merry} 

\date{}	

\begin{document}

\maketitle

\begin{abstract}
We show how Rabinowitz Floer homology can be seen as a standard Floer homology for fixed period Hamiltonian orbits on an extended phase space. 

\tableofcontents
\end{abstract}

\section*{Introduction}
\addcontentsline{toc}{section}{\numberline{}Introduction}

Rabinowitz Floer homology is a version of Floer homology for Hamiltonian periodic orbits having prescribed energy.  It was introduced by Cieliebak and Frauenfelder in \cite{cf09}. The setting is the following. One considers a smooth autonomous Hamiltonian $H$ on an exact symplectic manifold $(M,\omega)$ having a compact regular level $\Sigma:= H^{-1}(0)$, which is of restricted contact type with respect to a global primitive $\lambda$ of $\omega$. The manifold $M$ is assumed to be symplectically convex at infinity, and the Hamiltonian $H$ is assumed to be constant and positive outside from a compact set. This situation arises, for instance, when $M$ is the completion of a Liouville domain $W$ and $\Sigma$ is the boundary of $W$ (see e.g.\ \cite{sei06b}). In this introduction we also assume that the first Chern class of $TM$ vanishes on tori, but this assumption can be eliminated if one does not want to develop a $\Z$--graded theory. The periodic orbits of the Hamiltonian vector field $X_H$ induced by $H$ having energy zero are critical points of the free period action functional, or Rabinowitz action functional, which is defined as
\[
\mathcal{A}_H(x, \tau) := \int_{\T} x^*\lambda - \tau \int_{\T} H(x(t))\, dt, \qquad (x,\eta) \in C^{\infty}(\T,M) \times \R,
\]
where $\T:= \R/\Z$. More precisely, the critical points of $\mathcal{A}_H$ are either elements $(x_0,0)$, where $x_0$ is a constant loop in $\Sigma$, or elements $(x,\tau)$, where $\tau \neq 0$ and $z(t):= x(t/\tau)$ is a closed orbit of $X_H$ of period $|\tau|$ and energy $H(z)=0$. The negative $L^2$ gradient equation for this functional is the following Floer equation for $u:\R \times \T \rightarrow M$ coupled with an ODE for $\eta: \R \rightarrow \R$:
\begin{equation}
\label{rfeq}
\begin{split}
\partial_s u + J_t (u,\eta)&  \bigl(  \partial_t u - \eta X_H(u) \bigr) = 0, \\
\eta'(s) &= \int_{\T} H(u(s,t))\, dt.
\end{split}
\end{equation}
Here $J = \{ J_t(\cdot, \tau)\}_{(t,\tau) \in \T \times \R}$ is a family of almost complex structures on $M$  
compatible with the symplectic form $-\omega$. 
In most of the literature about Rabinowitz Floer homology, the almost complex structure does not depend on $\tau$. However, an explicit dependence on $\tau$ seems to be necessary in order to achieve transversality. See Section \ref{rfh_trans} below, where we explicitly discuss the transversality issue in Rabinowitz Floer homology.

The standard counting of zero-dimensional spaces of finite energy solutions of (\ref{rfeq}) defines a boundary operator
\[
\partial : RF_*(H,f) \rightarrow RF_{*-1}(H,f)
\]
on a graded $\Z_2$-vector space $RF_*(H,f)$ generated by critical points of an auxiliary Morse function $f$ on the critical set of $\mathcal{A}_H$, which generically is an at most countable union of finite dimensional compact manifolds. The $\Z$-grading of $RF_*(H,f)$ is defined by the transverse Conley-Zehnder index of periodic orbits plus the Morse index of critical points of $f$. The homology of the chain complex $(RF_*(H,f),\partial)$ is called the Rabinowitz Floer homology of the pair $(\Sigma,M)$ and is independent on the choice of the auxiliary data used in its definition. It is denoted by
\[
\{RFH_k(\Sigma,M)\}_{k\in \Z}.
\]

Rabinowitz Floer homology has several applications: it can be used to prove various versions of the Weinstein conjecture on the existence of periodic orbits on contact type hypersurfaces \cite{cf09,afm13}, gives obstructions to the existence of exact open symplectic embeddings \cite{cf09}, can be used to study Moser's problem of leafwise intersections , both in codimension one \cite{af10,af10b,af12,mmp11,kan12,kan14,mena14} and in higher codimensions \cite{kan13}. In addition Rabinowitz Floer homology is a productive tool in the study of the  symplectic properties of the energy levels of magnetic flows \cite{cfo10,bf11,me11} and orderability questions in contact geometry \cite{alfr12c,alme13,we13}. See \cite{alfr12} for a recent survey on the role of Rabinowitz Floer homology in some of these topics. This homology has been computed in the case of a displaceable $\Sigma$, for which $RFH(\Sigma,M)$ is actually zero \cite{cf09}, and in the case of the boundary $\Sigma$ 
of a domain in the cotangent bundle of a closed manifold $Q$ which is star-shaped with respect to the zero-section, for which $RFH_*(\Sigma,M)$ recovers the singular homology of the free loop space of $Q$ in positive degree and the singular cohomology of this space in negative degree \cite{cfo10,as09}. More generally, Rabinowitz Floer homology fits into an exact sequence involving symplectic homology and cohomology \cite{cfo10}.

The aim of this paper is to explain how Rabinowitz Floer homology can be seen as a standard Floer homology for fixed period Hamiltonian orbits on the extended phase space 
\[
\tilde{M} := M \times T^* \R.
\]
The starting observation is that the critical points of $\mathcal{A}_H$ are in one-to-one correspondence with the 1-periodic orbits $\tilde{x} : \T \rightarrow \tilde{M}$ of the Hamiltonian vector field induced by the smooth Hamiltonian
\[
\tilde{H}: \tilde{M} \rightarrow \R , \qquad \tilde{H}(x,\tau,\sigma) = \tau H(x), \qquad \forall \, (x,\tau,\sigma) \in M \times T^*\R = M \times \R \times \R^*.
\]
More precisely, the orbits of this system come in $\R^*$-families, because $\tilde{H}$ is invariant with respect to translations of the variable $\sigma\in \R^*$, and the set of 1-periodic orbits of $X_{\tilde{H}}$ modulo this action of $\R^*$ is in one-to-one correspondence with the critical points of $\mathcal{A}_H$.
This $\R^*$-symmetry could be broken by adding to the Hamiltonian $\tilde{H}$ a function of $\sigma\in \R^*$ having a unique critical point, but we prefer not to do this in order to keep the number of non-canonical choices to a minimum. 
In the above picture, $\sigma$ plays the role of a ``time dilation''. Since the dual variable to time is energy, we refer to $\tilde{M}$ as to the ``time-energy extended phase space''.  

The above observation suggests to look at the standard fixed period action functional for loops in $\tilde{M}$,
\[
\mathbb{A}_{\tilde{H}} (\tilde{x}) := \int_{\T} \tilde{x}^* \tilde{\lambda} - \int_{\T} \tilde{H}(\tilde{x}(t))\, dt,
\]
where $\tilde{\lambda} := \lambda \times (-\sigma\, d\tau)$. We choose a loop of compatible almost complex structures $\tilde{J}_t$ on $\tilde{M}$ of the form 
\[
\tilde{J}_t( x ,\tau ,\sigma) = J_t(x,\tau) \times \widehat{J},
\]
where $\widehat{J}$ is the standard complex structure on $T^* \R \cong \C$, $(\tau,\sigma) \mapsto \sigma + i \tau$, and as above, $J = \{ J_t(\cdot, \tau)\}_{(t,\tau) \in \T \times \R}$ is a family of almost complex structures on $M$  compatible with the symplectic form $-\omega$. Then the Floer negative gradient equation of $\mathbb{A}_{\tilde{H}}$ for 
\[
\tilde{u} := (u,\eta,\zeta) : \R \times \T \rightarrow M \times \R \times \R^* = \tilde{M}
\]
takes the form
\begin{equation}
\label{floerintro}
\begin{split}
\partial_s u + J_t(u , \eta) ( \partial_t u - \eta X_H(u)) &= 0, \\
\partial_s \eta + \partial_t \zeta - H(u) &= 0, \\
\partial_s \zeta - \partial_t \eta &= 0.
\end{split}
\end{equation}
The first result of this paper is that the above equation produces a well-defined Floer homology:

\begin{mainthm}
The Floer complex $(F_*(\tilde{H},f),\partial)$ over $\Z_2$ produced by the critical points of an auxiliary Morse function $f$ on the manifold of 1-periodic orbits of $X_{\tilde{H}}$ and by the zero-dimensional spaces of finite energy solutions of (\ref{floerintro}) modulo the $\R^*$-action is well-defined. \end{mainthm}

The main difficulty that we have to address in the construction of this Floer complex is the fact that the Hamiltonian $\tilde{H}$ is not coercive - it is unbounded from below and from above. Thus the standard arguments for showing that the Floer cylinders - that is, solutions of (\ref{floerintro}) - with prescribed asymptotics take values in a compact subset of $\tilde{M}$ cannot be applied. 
We overcome this difficulty in Section \ref{estisec}, where we prove that uniform bounds for the energy of solutions of the Floer equation associated to $\tilde{H}$  imply that these solutions take values in a compact subset of $\tilde{M}$, modulo translations of the variable $\sigma\in \R^*$.

The other non-standard feature of the Hamiltonian $\tilde{H}$ is that its periodic orbits come in non-compact families, due to the presence of the $\R^*$-action. Since also the Floer equation (\ref{floerintro}) is invariant with respect to this action, we deal with this issue by simply modding out the $\R^*$-action in our counting process. This involves revisiting  the proof of transversality in Hamiltonian Floer theory, in order to check that we can actually work with $\R^*$-invariant almost complex structures $\tilde{J}$. This is done in Section \ref{indcomptra}, where we also compute the Fredholm index of the operators which arise when linearising the equations (\ref{floerintro}). Also after modding out the $\R^*$-action, the periodic orbits of $X_{\tilde{H}}$ come in continuous families, because the Hamiltonian $\tilde{H}$ is autonomous. This fact is dealt with by the now standard way of counting Floer trajectories with cascades (see \cite{fra03,fra04,bh13} for the finite dimensional case). 

In our second main result we show that the Floer homology of $\tilde{H}$ is isomorphic to the Rabinowitz Floer homology of $(\Sigma,M)$. More precisely, we prove the following:

\begin{mainthm}
\label{maindue}
There is a chain complex isomorphism 
\[
\Phi: RF_*(H,f) \rightarrow F_*(\tilde{H},f).
\]
\end{mainthm}

As already observed, there is a one-to-one correspondence between the critical set of $\mathcal{A}_H$ and the quotient  of the critical set of $\mathbb{A}_{\tilde{H}}$ by the $\R^*$-action. However, there is no obvious correspondence between the spaces of negative gradient flow lines of these two functionals. There are however some relationships between these gradient flow lines, the most notable of which being the following: if $(u,\eta,\zeta): \R \times \T \rightarrow \tilde{M}$ is a solution of (\ref{floerintro}), then the average of $\eta$ over $\T$, that is the function
\[
\hat{\eta}(s) := \int_{\T} \eta(s,t)\, dt,
\]
solves the second equation in (\ref{rfeq}). 

Despite this relationship between solutions of (\ref{rfeq}) and (\ref{floerintro}), there is no reason to believe that the natural identification between $\crit \mathcal{A}_H$ and $\crit \mathbb{A}_{\tilde{H}}/\R^*$ produces a chain map. This makes the construction of the isomorphism $\Phi$ of Theorem \ref{maindue} non-trivial. Its definition is based on counting solutions of the following hybrid problem: we consider tuples $(u^-,\eta^-,u^+,\eta^+,\zeta^+)$ where 
\[
u^- : (-\infty,0] \times \T \rightarrow M, \qquad \eta^- : (- \infty, 0] \to \R
\]
is a solution of (\ref{rfeq}) with a prescribed asymptotics at $-\infty$,
\[
(u^+,\eta^+,\zeta^+) : [0,+\infty) \times \T \rightarrow M \times \R \times \R^* = \tilde{M}
\]
is a solution of (\ref{floerintro}) with a prescribed asymptotic at $+\infty$, and the following coupling condition at $s=0$ hold:
\[
u^-(0,t) = u^+(0,t), \qquad \eta^+(0,t) = \eta^-(0) \qquad \forall t\in \T.
\]
The above coupling condition can be seen as a Lagrangian boundary condition for $M \times \tilde{M}$-valued maps on a half-cylinder. Together with the compactness results which are proved in Section \ref{estisec}, this fact allows us to construct the chain map $\Phi$. In order to prove that $\Phi$ is an isomorphism, we use a standard argument (see \cite{as06}): we show that automatic transversality holds at stationary solutions of the above hybrid problem and that the difference of action at the asymptotics of a non-stationary solution is strictly positive. These facts imply that $\Phi$ is an isomorphism, since it can be represented by a (possibly infinite) upper triangular matrix whose diagonal entries equal 1.

\paragraph{Outlook.} One advantage of having an interpretation of Rabinowitz Floer homology as a Floer homology for fixed period Hamiltonian orbits is that its $S^1$-equivariant version can be constructed by standard arguments as in \cite{bo10}. A direct construction of equivariant Rabinowitz Floer homology has been recently presented by Frauenfelder and Schlenk in \cite{fs13}.

Another the advantages of being able to see Rabinowitz Floer homology as a Floer homology for fixed period Hamiltonian orbits is that the latter can be expected to have a ring structure. Indeed, Floer homology for fixed period Hamiltonian orbits on closed symplectic manifolds or on Liouville domains carries the pair-of-pants product, which is constructed by counting solutions of the Floer equation on the pair-of-pants Riemann surface, that is, the sphere minus three points (see \cite{sch95,sei06b,as10}). 

It is possible to show that the Floer homology for the unbounded Hamiltonian $\tilde{H}$ also carries a standard pair-of-pants product. However, some preliminary investigations suggest that this product is not very rich, and might be even zero. This product has degree $-(n+1)$, and this fact excludes that it restricts to the intersection product on $\Sigma$, when seen on the filtered Floer homology corresponding to a small interval containing $0$.

Nevertheless, it is seems to be possible to define a ``better'' product on the Floer homology of $\tilde{H}$, which is of degree $-n$. This product should be unital (and in particular non-zero), and under the isomorphism with Rabinowitz Floer homology it should restrict to the intersection product on the hypersurface $\Sigma$.  The definition of this product involves counting pair-of-pants with varying conformal structure. We hope to study it a sequel to the present paper.

\emph{Acknowledgement.} The possibility to see Rabinowitz Floer homology as a standard Floer homology on an extended phase space was suggested to us by Alexandru Oancea, whom we would like to thank also for the many discussions we had during the preparation of this article. 
Our gratitude goes also to Peter Albers and Matthias Schwarz, with whom we had countless conversations about products and Rabinowitz Floer homology. This work is part of the first author's activity within CAST, a Research Network Program of the European Science Foundation. The second author is supported by an ETH Postdoctoral Fellowship. 

\numberwithin{equation}{section}

\section{Floer homology on the extended phase space}
\label{compsec}

\subsection{The extended phase space} 

Let $(M,\omega)$ be an exact $2n$-dimensional symplectic manifold and 
let $H\in C^{\infty}(M)$ be an autonomous Hamiltonian. The corresponding Hamiltonian vector field $X_H$ is defined by $\imath_{X_H} \omega = - dH$. We assume that $H$ is constant and positive outside from a compact subset of $M$. We also assume that 0 is a regular value of $H$ and that the (necessarily compact) energy level
\[
\Sigma := H^{-1}(0)
\]
is non-empty and of restricted contact type, meaning that $\omega$ has a global primitive $\lambda$ such that $\alpha:=\lambda|_{\Sigma}$ is a positive contact form on $\Sigma$. 
The latter assumption means that the $(2n-1)$-form $\alpha\wedge d\alpha^{n-1}$ is everywhere positive on the oriented hypersurface $\Sigma$, where the orientation of $\Sigma$ is induced by that of $M$ by seeing $\Sigma$ as the boundary of the compact set $\{H\leq 0\}$. It follows that the restriction of $X_H$ to $\Sigma$ has the same direction as the Reeb vector field $R$ which is associated to $\alpha$, and hence
\begin{equation}
\label{alpha0}
\alpha(X_H|_{\Sigma}) \geq \alpha_0 
\end{equation}
for some positive constant $\alpha_0$. Of course, one could normalize $H$ in such a way that $X_H$ coincides with the Reeb vector field on $(\Sigma,\alpha)$, and hence $\alpha(X_H|_{\Sigma})=1$.

We are interested in the closed Reeb orbits on $\Sigma$, or equivalently in the closed orbits of $X_H$ with arbitrary period and energy $H=0$. This fixed energy problem can be transformed into a fixed period problem on an enlarged phase space. Indeed, consider the manifold
\[
\tilde{M} := M \times T^* \R,
\]
equipped with the one-form 
\[
\tilde{\lambda} = \lambda \times (-\lambda_0),
\]
$\lambda_0$ being the standard Liouville form on $T^*\R$, that is, \[
\lambda_0 = \sigma d\tau,
\]
where we are using coordinates $(\tau,\sigma) \in \R \times \R^* = T^* \R$.  The corresponding symplectic form on $\tilde{M}$ is denoted by $\tilde{\omega}=\omega \times (d\tau \wedge d\sigma)$. On $\tilde{M}$ we consider the Hamiltonian $\tilde{H}$ which is defined as
\[
\tilde{H}(x,\tau,\sigma) = \tau H(x), \qquad \forall (x,\tau,\sigma) \in M \times T^* \R,
\]
whose associated Hamiltonian vector field on $(\tilde{M},\tilde{\omega})$ is
\[
X_{\tilde{H}}(x,\tau,\sigma) = \tau X_H(x) +  H(x) \partial_{\sigma}.
\]
Therefore, $(x,\tau,\sigma):\R \rightarrow \tilde{M}$ is an orbit of $X_{\tilde{H}}$ if and only if it solves the equations
\begin{equation}
\label{hamsys}
\begin{split}
x' &= \tau X_H(x), \\
\tau' &= 0, \\
\sigma'&=H(x).
\end{split}
\end{equation}
The second equation says that $\tau$ is constant. If $\tau\neq 0$, then the first equation says that the reparametrized curve $t\mapsto x(t/\tau)$ is an orbit of $X_H$. If $\tau=0$, then the first equation says that $x$ is constant. In both cases, $H(x)$ is constant, so the third equation says that $\sigma$ is an affine function with slope $H(x)$. Therefore, the flow $\phi^t_{X_{\tilde{H}}}$ of the vector field $X_{\tilde{H}}$ on $\tilde{M}$ has the form
\begin{equation}
\label{theflow}
\phi_{X_{\tilde{H}}}^t(x_0,\tau_0,\sigma_0) = (\phi_{\tau_0 X_H}^t(x_0),\tau_0,\sigma_0 + t H(x_0)) = (\phi_{X_H}^{\tau_0 t}(x_0),\tau_0,\sigma_0 + t H(x_0)),
\end{equation}
for all $t\in \R$ and $(x_0,\tau_0,\sigma_0)\in \tilde{M}$.

We are interested in particular in 1-periodic orbits of $X_{\tilde{H}}$, that is in
closed curves $(x,\tau,\sigma): \T \rightarrow \tilde{M}$ which solve the above equations, where $\T:= \R/\Z$. In this case, the periodicity of $\sigma$ forces $H(x)$ to be zero, so that $\sigma$ is actually constant. If $\tau\neq 0$, then $t\mapsto x(t/\tau)$ must be a $|\tau|$-periodic orbit on $\Sigma=H^{-1}(0)$ ($|\tau|$ needs not be the minimal period). If $\tau=0$, then $x$ is an arbitrary constant path on $\Sigma=H^{-1}(0)$. Therefore, the one-periodic orbits of $X_{\tilde{H}}$ form the set
\[
\begin{split}
\mathcal{P}_1(X_{\tilde{H}}) = &\set{(x,\tau,\sigma)}{\tau\in \R\setminus \{0\}, \; \sigma\in \R^*, \; t\mapsto x(t/\tau) \mbox{ is a $|\tau|$-periodic orbit of $X_H$ on $\Sigma$}} \\ &\cup \bigl( \Sigma \times \{0\} \times \R^* \bigr).
\end{split}
\]
Notice that this set is invariant with respect to the free symplectic action of $\R^*$:
\begin{equation}
\label{Raction}
\R^* \times \tilde{M} \rightarrow \tilde{M}, \qquad (\xi,(x,\tau,\sigma)) \mapsto (x,\tau,\sigma+\xi).
\end{equation}
Indeed, this is a consequence of the fact that the Hamiltonian $\tilde{H}$, and hence also the system (\ref{hamsys}), is invariant with respect to this action. 

The elements of $\mathcal{P}_1(X_{\tilde{H}})$ are critical points of the Hamiltonian action functional
\[
\mathbb{A}_{\tilde{H}} (\tilde{x}) = \int_{\T} \tilde{x}^* \tilde{\lambda} - \int_{\T} \tilde{H}(\tilde{x}) \, dt = \int_{\T} x^* \lambda - \int_{\T} \sigma \tau' dt - \int_{\T} \tau H(x)\, dt , 
\]
where $\tilde{x}=(x,\tau,\sigma)$ is an element of $C^{\infty}(\T,\tilde{M})$.

\subsection{The Floer equation on the extended phase space} 
\label{floeq}
In order to write the Floer negative gradient equation for $\mathbb{A}_{\tilde{H}}$, we start by fixing a smooth family 
\[
\mathrm{J} =  \{ J_t(\cdot, \tau) \}_{ (t, \tau) \in \T \times \R}
\]
  of almost complex structures on $M$, which are compatible with $-\omega$, meaning that for each $(t,x,\tau) \in \T \times M \times \R$,
\[
\langle \cdot,\cdot \rangle_{J_t(x,\tau)} := \omega_x(J_t(x, \tau)\cdot,\cdot), 
\]
defines a Riemannian metric on $T_x M$, whose associated norm is denoted by $|\cdot|_{J_t(x,\tau)}$. For compactness purposes we will require that 
\begin{equation}
\label{eq:acs_bounded}
	\sup_{\tau \in \R } \| J_t(\cdot , \tau) \|_{C^{\ell}} < + \infty, \qquad \forall \ell\in \N. 
\end{equation}	
where $\| \cdot \|_{C^{\ell}} $  is the $C^{\ell}$-norm taken with respect to some background metric on $M$.  Let us denote by $\mathcal{J}$ the set of all such families $ \mathrm{J}$ of $-\omega$-compatible almost complex structures for which \eqref{eq:acs_bounded} is satisfied. 

Given $ \mathrm{J }\in \mathcal{J}$ we then consider the loop $\tilde{J}_t$ of almost complex structures on $\tilde{M}$ which is defined for $\tilde{x} = (x, \tau, \sigma) \in \tilde{M}$ by
\begin{equation}
\label{special form of acs}
\tilde{J}_t(\tilde{x}) = J_t (x,\tau)\times \widehat{J} , \qquad \mbox{where } \widehat{J} := \left( \begin{array}{cc} 0 & 1 \\ -1 & 0 \end{array} \right): T^* \R \rightarrow T^* \R.
\end{equation}
Thus $\tilde{J}_t$, $t \in \T$, is a loop of almost complex structures compatible with $-\tilde{\omega}$. The corresponding metric 
\[
\langle \cdot ,\cdot \rangle_{\tilde{J}_t(\tilde{x})} := \tilde{\omega}_{\tilde{x}} ( \tilde{J}_t (\tilde{x})\cdot, \cdot)
\]
is the product metric of $\langle \cdot,\cdot\rangle_{J_t(x, \tau)}$ with the Euclidean metric of $T^* \R \cong \R^2$. The norms $|\cdot|_{J_t(x, \tau)}$ and $|\cdot|_{\tilde{J}_t(\tilde{x})}$ are the norms which are used whenever we write the $L^p$ norm of sections of pullbacks of $TM$ or $T\tilde{M}$. 

The $L^2$-gradient of $\mathbb{A}_{\tilde{H}}$ has the form
\begin{equation}
\label{nabla}
\nabla \mathbb{A}_{\tilde{H}} (\tilde{x}) = \nabla \mathbb{A}_{\tilde{H}} (x,\tau,\sigma) = \tilde{J}_{\cdot}(\tilde{x}) ( \tilde{x}' - X_{\tilde{H}}(\tilde{x}) ) = \left( \begin{array}{c} J_{\cdot}(x,\tau) (x' - \tau X_H(x)) \\ \sigma' - H(x) \\ -\tau'  \end{array} \right),
\end{equation}
so the Floer negative gradient equation for $\mathbb{A}_{\tilde{H}}$, that is,
\[
\frac{d\tilde{u}}{ds}  + \nabla \mathbb{A}_{\tilde{H}} (\tilde{u})=0, \qquad \mbox{for } \tilde{u} : \R \rightarrow C^{\infty}(\T,\tilde{M}),
\]
or
\begin{equation}
\label{gradA}
\partial_s \tilde{u} + \tilde{J}_t(\tilde{u}) \bigl( \partial_t \tilde{u} -  X_{\tilde{H}}(\tilde{u})\bigr) = 0,
\end{equation}
is the following system of PDEs 
\begin{equation}
\label{floer}
\begin{split}
\partial_s u + J_t(u,\eta) \bigl( \partial_t u - \eta X_H(u)\bigr) &= 0, \\
\partial_s \eta + \partial_t \zeta - H(u) &= 0, \\
\partial_s \zeta - \partial_t \eta &= 0.
\end{split}
\end{equation}
for
\[
\tilde{u} = (u,\eta,\zeta) : \R \times \T \rightarrow M \times T^* \R = \tilde{M}.
\]
This system is invariant with respect to the $\R^*$-action (\ref{Raction}), so if $(u,\eta,\zeta)$ is a solution, so is $(u,\eta,\zeta+\xi)$ for every $\xi\in \R^*$.
As usual, we are interested in finite-energy solutions of the above systems, that is in solutions $\tilde{u} = (u,\eta,\zeta)$ for which the quantity
\begin{equation}
\label{ene1}
\mathbb{E}(\tilde{u}) := \int_{\R \times \T} |\partial_s \tilde{u}|_{\tilde{J}_t}^2 \, ds \, dt = \frac{1}{2} \int_{\R \times \T} \| d\tilde{u} - X_H(u) \otimes dt \|_{\tilde{J}_t}^2\, ds \, dt
\end{equation}
is finite. The norm $\|\cdot\|_{\tilde{J}_t}$ on $T\tilde{M}$-valued differential forms is induced by the norm $|\cdot|_{\tilde{J}_t}$ and by the Euclidean norm on the tangent space of the cylinder. The proof of the last identity uses the fact that $\tilde{u}$ is a solution of the Floer equation. The gradient structure of the equation (\ref{gradA}) implies that the function $s\mapsto \mathbb{A}_{\tilde{H}}(\tilde{u}(s,\cdot))$ is decreasing and that
\[
\begin{split}
\mathbb{E}(\tilde{u}) = - \int_{\R} \frac{d}{ds} \mathbb{A}_{\tilde{H}}(\tilde{u}(s,\cdot))\, ds &= 
\lim_{s\rightarrow -\infty}  \mathbb{A}_{\tilde{H}}(\tilde{u}(s,\cdot)) - \lim_{s\rightarrow +\infty} \mathbb{A}_{\tilde{H}}(\tilde{u}(s,\cdot)) \\ &= \sup_{s\in \R} \mathbb{A}_{\tilde{H}}(\tilde{u}(s,\cdot)) - \inf_{s\in \R} \mathbb{A}_{\tilde{H}}(\tilde{u}(s,\cdot)).
\end{split}
\]

Notice that if we average over $\T$ the functions $\eta:\R\times \T \rightarrow \R$ and $\zeta: \R \times \T \rightarrow \R^*$, we obtain the functions
\[
\hat{\eta}(s) := \int_{\T} \eta(s,t)\, dt \qquad \mbox{and} \qquad \hat{\zeta}(s) := \int_{\T} \zeta(s,t)\, dt,
\]
which solve the ODEs
\[
\hat{\eta}'(s) - \int_{\T} H(s,t)\, dt = 0 \qquad \mbox{and} \qquad \hat{\zeta}'(s) = 0.
\]
The first equation is precisely the equation for the evolution of the Lagrange multiplier in the Rabinowitz Floer equations (see equation (\ref{rfeq}) in the Introduction and Section \ref{rabflodiff} below). The second equation implies that the average $\hat{\zeta}$ of $\zeta$ is constant.

We recall that the symplectic manifold $(M,\omega)$ is said to be convex at infinity if there exists a closed contact $(2n-1)$-dimensional manifold $(\Sigma_{\infty},\alpha_{\infty})$ and an open symplectic embedding
\[
\iota: \bigl( \Sigma_{\infty} \times (0,+\infty), d(r\alpha_{\infty}) \bigr) \hookrightarrow (M,\omega),
\]
such that 
\[
M_0:= M\setminus \iota(\Sigma_{\infty} \times (0,+\infty))
\] 
is compact. Here $r$ denotes the second variable in the product $\Sigma_{\infty} \times (0,+\infty)$. 
In this paper, we will always assume that $(M, \omega)$ is convex at infinity. Up to a shift in $r$ of the symplectic embedding $\imath$, we can also assume that $H$ is constant outside $M_0$ and, in particular, that $\Sigma$ belongs to the interior part of $M_0$.

Moreover, we will always work with families $\mathrm{J}\in \mathcal{J}$ of almost complex structures that are constant and of contact type outside $M_0$. This means that the pullback $\iota^*(J_t(\cdot ,\tau))$ of $J_t(\cdot, \tau)$ to $\Sigma_{\infty} \times (0,+\infty)$ is an almost complex structure $J_{\infty}$ on $\Sigma_{\infty} \times (0,+\infty)$ which is independent of both $t\in \T$ and $\tau \in \R$ and is of contact type, meaning that it satisfies
\begin{equation}
\label{ct}
dr \circ J_{\infty} = r \alpha_{\infty} \qquad \mbox{on } \Sigma_{\infty} \times (0,+\infty).
\end{equation}
We denote by $\mathcal{J}_{\mathrm{con}}$ the subset of $\mathcal{J}$ consisting of those $\mathrm{J}$'s which satisfy this condition. 

The possibility of associating a Floer complex to the Hamiltonian action functional $\mathbb{A}_{\tilde{H}}$ relies on the following a priori bounds: 

\begin{prop}
\label{estimate}
Assume that $M$ is convex at infinity, that the family  $\mathrm{J} $ belongs to $\mathcal{J}_{\mathrm{con}}$, and that $H$ is constant outside $M_0$. Then for any $A\in \R$ there is a number $C=C(A)$, such that for every solution $\tilde{u} = (u,\eta,\zeta)$ of the Floer equation (\ref{floer}) with
\[
|\mathbb{A}_{\tilde{H}}(\tilde{u}(s))| \leq A \qquad \forall s\in \R,
\]
there holds
\[
\|\eta\|_{L^{\infty}(\R\times \T)} \leq C, \qquad \|\zeta-\hat{\zeta}\|_{L^{\infty}(\R\times \T)} \leq C, \qquad u(\R \times \T) \subset M_0,
\]
where $\hat\zeta$ indicates the average of $\zeta$ over $\T$, which as we have seen does not depend on $s$.
\end{prop}

This result is proved in Section \ref{completing_first_uniform_result}. Basing on these a priori bounds, the construction of the Floer complex for the autonomous Hamiltonian $\tilde{H}$ is almost standard, the only novelty being the fact that the presence of the $\R^*$-action (\ref{Raction}) causes critical points of $\mathbb{A}_{\tilde{H}}$ and solutions of the Floer equation (\ref{floer}) to come in non-compact $\R^*$-families. These solutions will be counted by modding out this $\R^*$-action. Moreover, the critical set of $\mathbb{A}_{\tilde{H}}$ contains the $2n$-dimensional manifold $\Sigma\times \{0\} \times \R^*$ and its complement is invariant also under the non-trivial $\T$-action given by time translations, so also after modding out the $\R^*$-action critical points remain not isolated. We will deal with this fact by a standard method, namely by considering Floer trajectories with cascades.

\subsection{The Floer differential} 
\label{flodiff}

From now on we assume that the following conditions are fulfilled:
\label{assumptions}
\begin{enumerate}[(i)]
\item The exact symplectic manifold $(M,\omega)$ is convex at infinity. 
\item The Hamiltonian $H\in C^{\infty}(M)$ is constant and positive outside of the compact set $M_0$, and the compact energy level $\Sigma:=H^{-1}(0)$ is non-empty, regular, and of restricted contact type with respect to a global primitive $\lambda$ of $\omega$. 
\item The almost complex structure $\mathrm{J}$  belongs to $\mathcal{J}_{\mathrm{con}}$. 
\item The flow $\phi^t_R$ of the Reeb vector field $R$ of $(\Sigma,\alpha) = (\Sigma,\lambda|_{\Sigma})$ is Morse-Bott.
\end{enumerate}

The last assumption means the following: for each $T>0$ the set $\mathcal{P}_T(R)$ of $T$-periodic points of $R$ is a closed submanifold of $\Sigma$ with
\[
T_{p}\mathcal{P}_T(R)=\ker(d \phi^T_R(p) - I)\ \ \ \mbox{for all }p\in\Sigma,
\]
and the rank of $d\alpha$ is locally constant on $\mathcal{P}_T(R)$. This is equivalent to the fact that the action functional $\mathbb{A}_{\tilde{H}}$ is Morse-Bott, meaning that its critical set $\crit \mathbb{A}_{\tilde{H}}$ is a union of finite-dimensional manifolds, and that at each critical point the kernel of the second differential of $\mathbb{A}_{\tilde{H}}$ coincides with the tangent space of the critical manifold to which the point belongs.

Let $\Lambda^-$ and $\Lambda^+$ be two connected components of $\crit \A_{\tilde{H}}$. By standard arguments, Proposition \ref{estimate} implies that the space of solutions $\tilde{u}=(u,\eta,\zeta)$ of (\ref{floer}) such that 
\[
\tilde{u}(-\infty) := \lim_{s\rightarrow -\infty} \tilde{u}(s,\cdot) \in \Lambda^-, \qquad \mbox{and} \qquad \tilde{u}(+\infty) := \lim_{s\rightarrow +\infty} \tilde{u}(s,\cdot) \in \Lambda^+
\]
is relatively compact in the quotient $C^{\infty}_{\mathrm{loc}}(\R\times \T,\tilde{M})/\R^*$ defined by the action (\ref{Raction}): this means that for every sequence $\tilde{u}_h=(u_h,\eta_h,\zeta_h)$ in this space there is a subsequence $\tilde{u}_{k_h}$ and a sequence $(\sigma_h)\subset \R^*$ such that $(u_{k_h},\eta_{k_h},\zeta_{k_h}+\sigma_h)$ converges to some $(u,\eta,\zeta)$ in $C^{\infty}_{\mathrm{loc}}(\R\times \T,\tilde{M})$. Indeed, by  the action bounds
\[
\A_{\tilde{H}} (\Lambda^+) \leq \A_{\tilde{H}}(\tilde{u}(s)) \leq \A_{\tilde{H}} (\Lambda^-), \qquad \forall s\in \R,
\]
where $\mathbb{A}_{\tilde{H}}(\Lambda)$ denotes the common value of $\A_{\tilde{H}}$ on the connected component $\Lambda$ of $\crit \A_{\tilde{H}}$,
Proposition \ref{estimate} implies that the maps $(u_h,\eta_h,\zeta_h - \hat{\zeta}_h)$ take values in a compact set. Then
the $C^{\infty}_{\mathrm{loc}}$ compactness follows from a standard bubbling-off argument, because $(\tilde{M},\tilde{\omega})$ does not have holomorphic spheres, since $\tilde{\omega}$ is exact, together with an elliptic bootstrap. 

We set 
\[
\mathcal{K} := \crit \A_{\tilde{H}}/\R^*,
\]
where the quotient is given by the free action (\ref{Raction}), and we denote by
\[
\pi_{\mathcal{K}} : \crit \A_{\tilde{H}} \rightarrow \mathcal{K}
\]
the corresponding projection. Then $\mathcal{K}$ is a finite dimensional manifold and is diffeomorphic to the union of $\Sigma$ (the constant loops on $H^{-1}(0)$) and of two copies of $\mathcal{P}_T(R)$ for each $T>0$ (corresponding to positive and negative reparametrization of each closed orbit). 
The Hamiltonian action $\A_{\tilde{H}}$ descends to a locally constant functional on $\mathcal{K}$, which we still denote by $\A_{\tilde{H}}$. We also fix a Morse function 
\[
f: \mathcal{K} \rightarrow \R
\] 
and a Riemannian metric $g$ on $\mathcal{K}$, such that the negative gradient flow $\phi^s_{-\nabla f}$ of $f$ with respect to $g$ is Morse-Smale. We lift $f$ to the Morse-Bott function 
\begin{equation}
\label{tildef}
\tilde{f} := f\circ \pi_{\mathcal{K}} : \crit \A_{\tilde{H}} \rightarrow \R,
\end{equation}
and $g$ to the Riemannian metric $\tilde{g}$ on $\crit \A_{\tilde{H}}$ making $\pi_{\mathcal{K}}$ a Riemannian submersion. The corresponding negative gradient flow on $\crit \A_{\tilde{H}}$ is denoted by $\phi^s_{-\nabla \tilde{f}}$. Its singular set $\crit \tilde{f}$ is the at most countable set of lines $\pi_{\mathcal{K}}^{-1}(\hat{x})$, for $\hat{x}\in \crit f$. The flow $\phi^s_{-\nabla \tilde{f}}$ preserves the $\R^*$-coordinate.

If $\tilde{u}$ is a solution of (\ref{gradA}), so is $\tilde{u}(s_0+\cdot,\cdot)$ for every $s_0\in \R$, and $[\tilde{u}]$ denotes the equivalence class of all such solutions.
Given two distinct connected components $\Lambda^-$ and $\Lambda^+$ of $\crit \A_{\tilde{H}}$, we can define the space of negative gradient flow lines with cascades 
\[
\mathcal{C}(\Lambda^-,\Lambda^+)
\] 
between them as  the set of all tuples $([\tilde{u}_1],\dots,[\tilde{u}_k])$, $k\geq 1$, where each $\tilde{u}_j=(u_j,\eta_j,\zeta_j)$ is a non-stationary finite-energy negative gradient flow line of $\A_{\tilde{H}}$ such that 
\[
\tilde{u}_1(-\infty) \in \Lambda^-, \qquad \tilde{u}_k(+\infty) \in \Lambda^+,
\]
and for each $j=1,\dots,k-1$ there holds
\[
\phi^{s_j}_{-\nabla \tilde{f}} (\tilde{u}_j(+\infty)) = \tilde{u}_{j+1}(-\infty),
\]
for some $s_j\geq 0$. Since the flow $\phi_{-\nabla \tilde{f}}^s$ preserves the $\R^*$-coordinate, the last condition implies that all the constant averages $\hat{\zeta}_j$ of the components $\zeta_j$ coincide. 
Moreover, $\tilde{u}_j(+\infty)$ and $\tilde{u}_{j+1}(-\infty)$ belong to the same connected component  $\Lambda_j$ of $\crit \A_{\tilde{H}}$ and
\[
\A_{\tilde{H}} (\Lambda^-) > \A(\Lambda_1) > \dots > \A_{\tilde{H}}(\Lambda_{k-1}) > \A_{\tilde{H}} (\Lambda^+).
\]
The action (\ref{Raction}) induces a free action of $\R^*$ on $\mathcal{C}(\Lambda^-,\Lambda^+)$. There are natural $\R^*$-equivariant mappings
\[
\begin{split}
\mathrm{ev}_-  &: \mathcal{C}(\Lambda^-,\Lambda^+) \rightarrow \Lambda^-, \; ([\tilde{u}_1],\dots,[\tilde{u}_k]) \mapsto \tilde{u}_1(-\infty), \\ \mathrm{ev}_+ &: \mathcal{C}(\Lambda^-,\Lambda^+) \rightarrow \Lambda^+, \; ([\tilde{u}_1],\dots,[\tilde{u}_k]) \mapsto \tilde{u}_k(+\infty).
\end{split}
\]
If $\hat{x}^-\in \mathcal{K}$ and $\hat{x}^+\in \mathcal{K}$ are critical points of $f$ on $\pi_{\mathcal{K}}(\Lambda^-)$ and $\pi_{\mathcal{K}}(\Lambda^+)$, we can consider the $\R^*$-invariant spaces
\[
\begin{split}
\mathcal{C}(\hat{x}^-,\Lambda^+) &:= \set{w\in \mathcal{C}(\Lambda^-,\Lambda^+)}{\mathrm{ev}_-(w) \in W^u_{-\nabla \tilde{f}}(\pi_{\mathcal{K}}^{-1}(\hat{x}^-))}, \\ 
\mathcal{C}(\Lambda^-,  \hat{x}^+) &:= \set{w\in \mathcal{C}(\Lambda^-,\Lambda^+)}{\mathrm{ev}_+(w) \in W^s_{-\nabla \tilde{f}}(\pi_{\mathcal{K}}^{-1}(\hat{x}^+))}, \\ 
\mathcal{C}(\hat{x}^-,\hat{x}^+) &:= \mathcal{C}(\hat{x}^-,\Lambda^+) \cap \mathcal{C}(\Lambda^-, \hat{x}^+),
\end{split}
\]
where $W^u_{-\nabla \tilde{f}}(\pi_{\mathcal{K}}^{-1}(\hat{x}))$ and $W^s_{-\nabla \tilde{f}}(\pi_{\mathcal{K}}^{-1}(\hat{x}))$ denote the unstable and the stable manifold of the critical manifold $\pi_{\mathcal{K}}^{-1}(\hat{x})$ with respect to the negative gradient flow $\phi^s_{-\nabla \tilde{f}}$ on $\crit \A_{\tilde{f}}$. It is convenient to extend the above definitions to the case in which $\Lambda^-=\Lambda^+=\Lambda$ and $\hat{x}^-\neq \hat{x}^+$ are critical points of $f$ in $\pi_{\mathcal{K}}(\Lambda)$, by setting
\begin{equation}
\label{extension of cascades}
\begin{split}
\mathcal{C}(\hat{x}^-,\Lambda) &:= W^u_{-\nabla \tilde{f}} (\pi_{\mathcal{K}}^{-1}(\hat{x}^-))/\R,   \qquad \mathcal{C}(\Lambda,\hat{x}^+) := W^s_{-\nabla \tilde{f}} (\pi_{\mathcal{K}}^{-1}(\hat{x}^+))/\R, \\ 
&\mathcal{C}( \hat{x}^-,\hat{x}^+) := \bigl( W^u_{-\nabla \tilde{f}} (\pi_{\mathcal{K}}^{-1}(\hat{x}^-)) \cap W^s_{-\nabla \tilde{f}} (\pi_{\mathcal{K}}^{-1}(\hat{x}^+)) \bigr)/\R, 
\end{split}
\end{equation}
where the $\R$-action is the one given by the flow $\phi^s_{-\nabla f}$.

\begin{prop}
\label{cascades_prop}
For a generic choice of the pair $(\mathrm{J},g)$ the sets $\mathcal{C}(\Lambda^-,\Lambda^+)$, $\mathcal{C}(\hat{x}^-,\Lambda^+)$, $\mathcal{C}(\Lambda^-,\hat{x}^+)$, and $ \mathcal{C}(\hat{x}^-,\hat{x}^+)$ are finite dimensional manifolds, for every $\Lambda^-,\Lambda^+$ connected components of $\mathrm{crit}\, \A_H$ and for every $\hat{x}^-, \hat{x}^+$ in $\mathrm{crit}\, f$.
\end{prop}

Here and in the whole article, ``for a generic choice of the pair $(\mathrm{J},g)$'' means for a residual set of $(\mathrm{J},g)$ in the product of $\mathcal{J}_{\mathrm{con}}$ with the space of Riemannian metrics on $\mathcal{K}$. 
Due to the special form we insist our almost complex structures take on the $T^*\R$-factor, the above proposition does not follow directly from the standard transversality statements in Floer theory. Its proof will be discussed in Section \ref{transversality for cylinders}.

We define the $\Z_2$-vector space
\[
F(\tilde{H},f) := \Bigl\{\epsilon\in \Z_2^{\mathrm{crit}\, f} \; \Big| \; \sup_{\substack{\hat{x}\in \mathrm{crit}\, f \\ \epsilon(\hat{x})=1}} \A_{\tilde{H}} (\hat{x}) < +\infty\Bigr\}.
\]
Let $\hat{x}^-\in \mathcal{K}$ and $\hat{x}^+\in \mathcal{K}$ be critical points of $f$. By standard compactness and transversality arguments, the zero-dimensional part of the manifold $\mathcal{C}(\hat{x}^-,\hat{x}^+)/\R^*$ is a finite set, and we denote by $n_{\partial}(\hat{x}^-,\hat{x}^+)\in \Z_2$ its parity. Then we define
\[
\partial:F (\tilde{H},f) \rightarrow F (\tilde{H},f),
\]
to be the homomorphism which maps $\epsilon\in F(\tilde{H},f)$ into the element $\partial \epsilon \in F(\tilde{H},f)$ which is defined by
\begin{equation}
\label{ilbordo}
(\partial \epsilon ) (\hat{x}^+) = \sum_{\hat{x}^- \in \mathrm{crit}\, f} n_{\partial} (\hat{x}^-,\hat{x}^+) \epsilon(\hat{x}^-), \qquad \forall \hat{x}^+\in \crit \, f.
\end{equation}
The fact that $n_{\partial}(\hat{x}^-,\hat{x}^+)$ vanishes when $\A_H(\hat{x}^+)> \A_H(\hat{x}^-)$ implies that the above sum is finite for every $\hat{x}^+$ and that $\partial \epsilon$ belongs to $F (\tilde{H},f)$.

By studying the one-dimensional components of the manifolds $\mathcal{C}(\hat{x}^-,\hat{x}^+)/\R^*$, one proves that $\partial$ is a boundary operator, that is, $\partial\circ\partial=0$. Therefore,
\[
\bigl\{ F(\tilde{H},f), \partial \bigr\}
\]
is a differential $\Z_2$-vector space. 

The homology of the Floer differential $\Z_2$-vector space $\{ F(\tilde{H},f), \partial \}
$ is denoted by
\[
\mathit{FH}(\tilde{H}) := \frac{\ker \partial}{\mathrm{im}\, \partial}.
\]

The Floer homology $FH(\tilde{H})$ is independent on the choice of $\mathrm{J}$, $f$, $g$. Moreover, it depends on $H$ only through its zero-level set $\Sigma$. These assertions follow from the fact that, as we shall prove in Section \ref{isomorphism}, this Floer homology is isomorphic to the Rabinowitz Floer homology of $(\Sigma,M)$. 

A direct proof by the standard continuation argument is also possible, but one would have to face the following difficulty when dealing with the independence from $H$: a homotopy $H_s$ between two Hamiltonians $H_0$ and $H_1$ on $M$ induces a homotopy $\tilde{H}_s$ of Hamiltonians on $\tilde{M}$ whose derivative with respect to $s$ is unbounded. This fact has the following consequence. Let $\chi:\R \rightarrow [0,1]$ be a smooth function constantly equal to 0 on $(-\infty,0]$ and to 1 on $[1,+\infty)$. The energy of a solution of the $s$-dependent Floer equation 
\[
\partial_s \tilde{u} + \tilde{J}_t(\tilde{u}) \bigl( \partial_t \tilde{u} - X_{\tilde{H}_{\chi(s)}} (\tilde{u}) \bigr)=0,
\]
which joins two given periodic orbits $\tilde{x}^-$ and $\tilde{x}^+$ is the quantity
\[
E(\tilde{u}) = \int_{\R \times \T} |\partial_s \tilde{u}|_{\tilde{J}_t}^2\, ds \, dt =  \mathbb{A}_{\tilde{H}_0}(\tilde{x}^-) -\mathbb{A}_{\tilde{H}_1}(\tilde{x}^+) - \int_{\R\times \T} \chi'(s) \tilde{H}_{\chi(s)} (\tilde{u})\, ds\, dt.
\]
Since the function $\chi'(s) \tilde{H}_{\chi(s)}$ is unbounded, it is a priori not clear that these solutions have uniformly bounded energy. This difficulty can be overcome by factorising the homotopy $\tilde{H}_s$ into several homotopies between Hamiltonians which are sufficiently close, as in \cite[Section 1.8]{as06}. 

\subsection{Grading and the Floer chain complex}
\label{grading-sec}

Under the additional assumption that the first Chern class $c_1(TM)$ vanishes on tori, the Floer differential vector space $F(\tilde{H},f)$ has a $\Z$-grading with respect to which the boundary operator $\partial$ has degree $-1$. The definition of this grading its standard, and here we just review it quickly. The precise form of the trivialisations we take will be useful in Section \ref{grading-rfh-sec}.

Since $\Sigma\subset M$ is by assumption a hypersurface
of restricted contract type, there exists a neighborhood $U$ of $\Sigma$
in $M$ and a symplectomorphism
\begin{equation}
\label{the_nbd_U_of_Sigma}
(U,\omega|_{U})\cong\left(\Sigma\times(1-\delta,1+\delta),d(r\alpha)\right),
\end{equation}
where $r$ denotes the second variable in the product $\Sigma\times(1-\delta,1+\delta)$. 

For each free homotopy class $e\in [\T,M]$ which contains loops in $\Sigma$ we choose an element $y_e \colon \T \to \Sigma$ belonging to $e$. 
If $e$ is the trivial free homotopy class then we take $y_e$ to be constant, and we require that $y_{e^{-1}}$ must agree with $y_e$ taken with the opposite orientation. We now choose symplectic trivialisations $ \Phi_e : \T \times \R^{2n} \to (y_e)^*(TM)$ such that
\begin{equation}
\label{triv_formula}
\begin{split}
\Phi_e( \T \times \R^{2n-2} \times \{0\} \times \{0\})  = (y_e)^* ( \ker \alpha),\\
\Phi_e( \T \times \{0\} \times \R \times \{0\})  = y_e^* (\mathrm{span}\, R),\\
\Phi_e( \T \times \{0\} \times \{0\} \times \R)  = y_e^* (\mathrm{span}\, r \partial_r),
\end{split}
\end{equation}
where we are using our identification \eqref{the_nbd_U_of_Sigma} to define the vector field $r \partial_r$. Next, by taking the product by the constant trivialisation of the trivial bundle $\T \times T^* \R\rightarrow \T$, we obtain a symplectic trivialisation  $\tilde{ \Phi}_e : \T \times \R^{2n+2} \to \tilde{y}_e^*(T\tilde M) = y_e^*(TM) \times T^* \R$, where $\tilde{y}_e=(y_e,0,0)$. 
Suppose now that $ \tilde{x} = ( x , \tau, \sigma)$ is a critical point of $ \A_{ \tilde{H}}$. Denote by $e\in [\T,M]$ the free homotopy class of $x$. We now extend the trivialisation $ \tilde{ \Phi}_e$ via parallel transport along a homotopy connecting $\tilde{y}_e$ to $\tilde{x}$, thus defining a symplectic trivialisation of $\tilde{x}^*(T\tilde{M})$. This trivialisation conjugates the differential of the Hamiltonian flow $\phi_{X_{\tilde{H}}}^t$ along $\tilde{x}$ with a path 
\begin{equation}
\label{path_gamma}
\Gamma_{\tilde{x}} : [0,1] \rightarrow \mathrm{Sp}(2n+2).
\end{equation}

We denote by 
\[
\mu_{\mathrm{rs}}(\Gamma_{\tilde{x}})\in \frac{1}{2} \Z
\]
the Robbin-Salamon index of $\Gamma_{\tilde{x}}$ (see \cite{rs95}). Here we use the same sign conventions of \cite{as10}, see in particular \cite[Section 5.1]{as10}. The fact that $c_1(TM)$, and therefore also $c_1(T\tilde M)$, vanish on tori implies that $\mu_{\mathrm{rs}}(\Gamma_{\tilde{x}})$ does not depend on the choice of the homotopy connecting $\tilde{y}_e$ to $\tilde{x}$. Furthermore, $\mu_{\mathrm{rs}}(\Gamma_{\tilde{x}})$ depends only on the connected component $\Lambda$ of $\crit \A_{ \tilde{H}}$ to which $\tilde{x}$ belongs. By
\[
\mu_{\mathrm{rs}}(\Lambda) \in \frac{1}{2} \Z
\]
we denote the common value of $\mu_{\mathrm{rs}}(\Gamma_{\tilde{x}})$ over $\tilde{x}\in \Lambda$. Due to the special form of $\tilde{H}$, this Robbin-Salamon index can be expressed in terms of the index of the reparametrized closed Reeb orbit  $x: \T \rightarrow \Sigma$, where $\tilde{x}=(x,\tau,\sigma)$ (in the case $\tau \neq 0$). This result is due to Bourgeois and Oancea \cite{bo13}, and is discussed in Section \ref{grading-rfh-sec} below.

It is also convenient to define the number
\[
\mu(\Lambda) := \mu_{\mathrm{rs}} (\Lambda) - \frac{1}{2} \dim \Lambda,
\]
which by the properties of the Robbin-Salamon index is an integer. The grading on $F(\tilde{H},f)$ is induced by the integer-valued function
\begin{equation}
\label{mu f x}
\mu_f : \crit f \rightarrow \Z, \qquad \mu_f(\hat{x}) := \mu(\Lambda) + \mathrm{ind}_f (\hat{x}) + 1,
\end{equation}
where $\Lambda$ is the connected component of $\crit \A_{\tilde{H}}$ containing $\pi_{\mathcal{K}}^{-1}(\hat{x})$, and $\mathrm{ind}_f (\hat{x})$ denotes the Morse index of $\hat{x}\in \mathcal{K}$ as a critical point of $f: \mathcal{K} \rightarrow \R$. The additional``$+1$'' added in the formula is added purely for convenenience; as we will see later this will imply that the isomorphism between the Floer homology of $ \tilde{H}$ and the Rabinowtiz Floer homology of $ \Sigma$ is grading preserving. Notice that $\mathrm{ind}_f (\hat{x}) + 1$ is the dimension of the unstable manifold of the line $\pi_{\mathcal{K}}^{-1}(\hat{x})$ with respect to the flow of $-\nabla \tilde{f}$ on $\Lambda$. 

\begin{prop}
\label{dimprop}
For a generic choice of $(\mathrm{J},g)$ the manifolds of Floer trajectories with cascades have dimension  
\begin{equation}
\label{dimensions}
\begin{split}
\dim \mathcal{C}(\Lambda^-,\Lambda^+) &= \mu(\Lambda^-) + \dim \Lambda^- - \mu(\Lambda^+)-1, \\ 
\dim \mathcal{C}(\hat{x}^-,\Lambda^+) &= \mu_f(\hat{x}^-) - \mu(\Lambda^+) - 1,  \\ 
\dim \mathcal{C}(\Lambda^-,\hat{x}^+) &= \mu (\Lambda^-) + \dim \Lambda^- - \mu_f(\hat{x}^+) , \\ 
\dim \mathcal{C}(\hat{x}^-,\hat{x}^+) &= \mu_f(\hat{x}^-) - \mu_f(\hat{x}^+),
\end{split}
\end{equation}
for every connected components $\Lambda^-\neq \Lambda^+$  of $\mathrm{crit}\, \A_{\tilde{H}}$, and for every $\hat{x}^-, \hat{x}^+$ in $\mathrm{crit}\, f$.
\end{prop}

\begin{proof}
Assume that $\Lambda^-\neq \Lambda^+$ and denote by $\mathcal{M}( \Lambda^- , \Lambda^+)$ the space of gradient flow lines (without cascades) running from $ \Lambda^-$ to $ \Lambda^+$, where we have \emph{not} divided through by either the free translation $\R$-action or the free $\R^*$-action \eqref{Raction}. Theorem \ref{virdim_cylinders} below tells us that the virtual dimension of  $\mathcal{M}( \Lambda^- , \Lambda^+)$ is given by 
\[
\mu ( \Lambda^-)  + \dim \Lambda - \mu( \Lambda^+).
\]
One now argues as in \cite[Appendix A]{cf09} or \cite{bh13} to see that the dimension of the space $ \mathcal{C}( \Lambda^- , \Lambda^+)$ of cascades is given by 
\[
\dim \mathcal{C}(\Lambda^-,\Lambda^+) = \dim \mathcal{M}( \Lambda^- , \Lambda^+) - 1 = \mu(\Lambda^-) + \dim \Lambda^- - \mu(\Lambda^+) -1,
\] 
where the $-1$ comes from the fact that in the space $\mathcal{C}( \Lambda^-, \Lambda^+)$ we have divided out by the translation $\R$-action (but not the $\R^*$-action \eqref{Raction}). Next, if $\hat{x}^- \in \pi_{\mathcal{K}} (\Lambda^-)$ then
\begin{align*}
\dim \mathcal{C}(\hat{x}^-,\Lambda^+) &= \dim \mathcal{C}(\Lambda^-, \Lambda^+) -\codim_{\Lambda^-} W^u_{-\nabla \tilde{f}} (\pi_{\mathcal{K}}^{-1}(\hat{x}^-)) \\ &=
\mu(\Lambda^-) + \dim \Lambda^- - \mu(\Lambda^+) -1 - \dim \Lambda^- + \mathrm{ind}_f (\hat{x})+1\\ &= \mu_f(\hat{x}^-) - \mu(\Lambda^+) - 1.
\end{align*}
Similarly, if $\hat{x}^+ \in \pi_{\mathcal{K}} ( \Lambda^+)$ then
\begin{align*}
\dim \mathcal{C}(\Lambda^-, \hat{x}^+) &= \dim \mathcal{C}(\Lambda^-, \Lambda^+) - \codim_{\Lambda^+} W^s_{-\nabla \tilde{f}} (\pi_{\mathcal{K}}^{-1}(\hat{x}^+)) \\ &= \mu(\Lambda^-) + \dim \Lambda^- - \mu(\Lambda^+) -1 - \dim \Lambda^+ + (\dim  \Lambda^+ - \mathrm{ind}_f (\hat{x}^+) )\\ &= \mu(\Lambda^-) + \dim \Lambda^- - \mu_f(\hat{x}^+).
\end{align*}
Furthermore, if $\hat{x}^{\pm} \in  \pi_{\mathcal{K}} ( \Lambda^{\pm})$ then 
\begin{align*}
\dim \mathcal{C}( \hat{x}^-, \hat{x}^+) &= \dim \mathcal{C}(\hat{x}^-, \Lambda^+) + \dim \mathcal{C}(\Lambda^-, \hat{x}^+) - \dim \mathcal{C}(\Lambda^-,\Lambda^+) \\
 &=  \mu_f(\hat{x}^-) - \mu_f(\hat{x}^+).
\end{align*}
Finally, the last three identities in (\ref{dimensions}) obviously hold when $\Lambda^-=\Lambda^+$ and the spaces which appear there are defined as in (\ref{extension of cascades}).
\end{proof}

The last of the formulas (\ref{dimensions}) implies that 
\[
\dim  \mathcal{C}(\hat{x}^-,\hat{x}^+) /\R^* = \mu_f(\hat{x}^-) - \mu_f(\hat{x}^+) - 1,
\]
so $n_{\partial}(\hat{x}^-,\hat{x}^+)=0$ whenever $\mu_f(\hat{x}^+) \neq  \mu_f(\hat{x}^-) - 1$. Therefore, if we set
\[
F_k(\tilde{H},f) := \set{\epsilon\in F(\tilde{H},f)}{\epsilon(\hat{x}) = 0 \mbox{ if } \mu_f(\hat{x}) \neq k}, \qquad \forall k\in \Z,
\]
we see that the boundary operator $\partial$ maps the subspace $F_k(\tilde{H},f)$ into the subspace $F_{k-1}(\tilde{H},f)$, for every $k\in \Z$. Therefore, 
\[
\{F_k(\tilde{H},f),\partial\}_{k\in \Z}
\]
is a chain complex of $\Z_2$-vector spaces. Its homology is denoted by
\[
FH_k(\tilde{H}) := \frac{\ker \partial|_{F_k(\tilde{H},f)}}{\mathrm{im}\, \partial|_{F_{k+1}(\tilde{H},f)}}, \qquad \forall k\in \Z.
\]
This homology is independent of the choice of $\mathrm{J}$, $f$, $g$, and depends on $H$ only through its zero-level set $\Sigma$.

\section{The isomorphism with Rabinowitz Floer homology}
\label{isomorphism}

In this section we briefly recall K.~Cieliebak's and U.~Frauenfelder's construction of the Rabinowitz Floer complex from \cite{cf09} and we prove that this chain complex is isomorphic to the Floer complex of $\tilde{H}$. The assumptions on the symplectic manifold $(M,\omega)$ and on the Hamiltonian $H\in C^{\infty}(M)$ are still the conditions (i), (ii), (iii) and (iv) in Section \ref{flodiff}.

\subsection{The Rabinowitz Floer differential}
\label{rabflodiff}

The free period action functional - or Rabinowitz action functional -  
\[
\mathcal{A}_{H}:C^{\infty}(\mathbb{T},M)\times\mathbb{R}\rightarrow\mathbb{R}
\]
is defined by 
\[
\mathcal{A}_{H}(x,\tau):=\int_{\mathbb{T}}x^{*}\lambda-\tau\int_{\mathbb{T}}H(x(t))dt.
\]
One readily checks that a pair $(x,\tau)$ is a critical point of
$\mathcal{A}_{H}$ if and only if:
\begin{equation}
\begin{split}
x'  =\tau X_{H}(x),\\ 
\int_{\mathbb{T}}H(x(t))\, dt=0.
\end{split}
\end{equation}
By the first equation, $x$ is either a reparametrized closed orbit of $X_H$ - if $\tau\neq 0$ - or is constant - if $\tau=0$. In both cases $H(x)$ is constant, and by the second equation $H(x)$ is identically zero. We conclude that the critical set of $\mathcal{A}_H$ is
\[
\begin{split} 
\crit \mathcal{A}_H &= \set{(x,\tau)}{\tau \in \R\setminus \{0\}, \; t\mapsto x(t/\tau) \mbox{ is a $|\tau|$-periodic orbit of $X_H$ on $\Sigma$}} \\ &\cup (\Sigma \times \{0\}).
\end{split}
\]
Therefore, $\crit  \mathcal{A}_H$ is naturally identified with the set
\[
\mathcal{K} = \crit \A_{\tilde{H}}/\R^*,
\]
which was introduced in Section \ref{flodiff}. Moreover, the functionals $\mathcal{A}_H$ and $\A_{\tilde{H}}$ coincide on $\mathcal{K}$:
\[
\mathcal{A}_H(\hat{x}) = \int_{\T} x^* \lambda = \A_{\tilde{H}}(\hat{x}) \qquad \forall \hat{x}=(x,\tau)\in \mathcal{K}.
\] 

Fix an element $\mathrm{J} \in \mathcal{J}_{\mathrm{con}}$. The corresponding $L^{2}$-gradient of $\mathcal{A}_{H}$ is given by
\begin{equation}
\label{rfh_gradient}
\nabla\mathcal{A}_{H}(x,\tau)=\left(J_{\cdot}(x,\tau)(x'-\tau X_{H}(x)),-\int_{\mathbb{T}}H(x)dt\right).
\end{equation}
Thus the negative gradient flow equation for $\mathcal{A}_{H}$,
that is,
\[
\frac{dv}{ds}+\nabla\mathcal{A}_{H}(v)=0,\qquad\mathrm{for}\ v=(u,\eta):\mathbb{R}\rightarrow C^{\infty}(\mathbb{T},M)\times\mathbb{R},
\]
is the following system, coupling a PDE with an ODE:
\begin{equation}
\label{eq:RFH equations}
\begin{split}
\partial_{s}u+J_t(u,\eta)(\partial_{t}u-\eta X_{H}(u)) & =0,\\
\eta'-\int_{\mathbb{T}}H(u)dt & =0.
\end{split}
\end{equation}
As noticed in Section \ref{floeq}, the second equation has the same form of the equation for the time-average of the $\R$-component of a solution of the Floer equation for $\tilde{H}$ on $\tilde{M}$.

\begin{rem}
\label{rem:acs_in_rfh}
We emphasise that in contrast to most of the current literature on Rabinowitz Floer homology, we work with almost complex structures $\mathrm{ J} = \{ J_t(\cdot, \tau) \}_{ (t, \tau) \in \T \times \R}$ that depend explicitly on $\tau$. The reason for this choice will become transparent in Section \ref{rfh_trans} below. We were unable to verify that transversality can be achieved in Rabinowitz Floer homology using a standard loop $J_t, t \in \T$ of almost complex structures.
\end{rem}

The energy $\mathbb{E}(v)$ of a flow line $v=(u,\eta)$ is now the quantity
\[
\mathbb{E}(v):=\int_{\mathbb{R}\times\mathbb{T}}\left|\partial_{s}u\right|_{J_{t}(u,\eta)}^{2}\, ds\,dt+\int_{\mathbb{R}}\left|\eta'\right|^{2}ds.
\]
As with the functional $\mathbb{A}_{\tilde{H}}$, the free period action
functional $\mathcal{A}_{H}$ is never Morse. However, it is Morse-Bott
exactly when $\mathbb{A}_{\tilde{H}}$ is. Indeed, both conditions are equivalent to the assumption (iv) of Section \ref{flodiff}. 

Consider a Morse function $f$ and a Riemannian metric $g$ on $\mathcal{K}$, having a Morse-Smale negative gradient flow $\phi^s_{-\nabla f}$.  As before, given two distinct components $K^-$ and $K^+$ of $\mathcal{K}\cong \mathrm{crit}\, \mathcal{A}_H$,
we form the space $\mathcal{C}_{RF}(K^{-},K^{+})$ of gradient flow lines
with cascades. This is defined as the set of all tuples $([v_{1}],\dots,[v_{k}]),\, k\geq1$,
where each pair $v_{j}=(u_{j},\eta_{j})$ is a non-stationary finite-energy
negative gradient flow line of $\mathcal{A}_{H}$, such that
\[
v_{1}(-\infty)\in K^{-},\qquad v_{k}(+\infty)\in K^{+},
\]
and such that for each $j=1,\dots,k-1$ there exists $s_{j}\geq0$
such that
\[
\phi^{s_{j}}_{-\nabla f}(v_{j}(+\infty))=v_{j+1}(-\infty).
\]
As before, there are natural maps $\mathrm{ev}_{\pm}:\mathcal{C}_{RF}(K^{-},K^{+})\rightarrow K^{\pm}$. 

If $\hat{x}^- = (x^-,\tau^-)$ and $\hat{x}^+ = (x^+,\tau^+)$ are elements of $\mathcal{K}$ belonging to $K^-$ and $K^+$ respectively, we consider the spaces
\[
\begin{split}
\mathcal{C}_{RF}(\hat{x}^-,K^+) &:= \set{w\in \mathcal{C}_{RF}(K^-,K^+)}{\mathrm{ev}_-(w) \in W^u_{-\nabla f}(\hat{x}^-)}, \\ \mathcal{C}_{RF}(K^-, \hat{x}^+) &:= \set{w\in \mathcal{C}_{RF}(K^-,K^+)}{\mathrm{ev}_+(w) \in W^s_{-\nabla f}(\hat{x}^+)}, \\ \mathcal{C}_{RF}(\hat{x}^-,\hat{x}^+) &:= \mathcal{C}_{RF}(\hat{x}^-,K^+) \cap \mathcal{C}_{RF}(K^-, \hat{x}^+).
\end{split}
\]
Exactly as in \eqref{extension of cascades}, we extend these spaces to cover the cases when one or both of $\hat{x}^\pm$ belongs to $K^\pm$. For a generic choice of $(\mathrm{J},g)$, these are finite dimensional manifolds. Moreover, they fullfil the standard standard compactness-up-to-breaking property. As remarked above, we will study give a detailed proof that transversality can indeed be achieved for a generic family $(\mathrm{J},g)$ in Section \ref{rfh_trans}.

We define the $\Z_2$-vector space
\[
RF(H,f) := \Bigl\{\epsilon\in \Z_2^{\mathrm{crit}\, f} \; \Big| \; \sup_{\substack{\hat{x}\in \mathrm{crit}\, f \\ \epsilon(\hat{x})=1}} \mathcal{A}_{H} (\hat{x}) < +\infty\Bigr\}.
\]
If $\hat{x}^-\in \mathcal{K}$ and $\hat{x}^+\in \mathcal{K}$ are critical points of $f$, we denote by $n_{\partial}^{RF}(\hat{x}^-,\hat{x}^+)$ the parity of the zero-dimensional part of $\mathcal{C}_{RF}(\hat{x}^-,\hat{x}^+)$, and we define the Rabinowitz Floer differential 
\[
\partial:RF (H,f) \rightarrow RF (H,f),
\]
to be the homomorphism which maps $\epsilon\in RF(H,f)$ into the element $\partial \epsilon \in RF(H,f)$ which is defined by
\begin{equation}
\label{ilbordoRF}
\partial \epsilon (\hat{x}^+) = \sum_{\hat{x}^- \in \mathrm{crit}\, f} n_{\partial}^{RF} (\hat{x}^-,\hat{x}^+)\epsilon(\hat{x}^-), \qquad \forall \hat{x}^+\in \crit \, f.
\end{equation}
Its homology  is independent on the choice of $\mathrm{J}$, $f$, $g$, and depends on $H$ only through its zero-level set $\Sigma$. Therefore, it is denoted by 
\[
RFH(\Sigma,M).
\]
See \cite{cf09} for more details.

\subsection{Grading and the Rabinowitz Floer complex}
\label{grading-rfh-sec}

In this section we assume that $c_1(TM)$ vanishes on tori and recall the definition of the grading of  the Rabinowitz Floer complex. We use the the conventions of \cite{cfo10}, which are more standard than those used initially in \cite{cf09}. The difference is a factor 1/2. Then we compare this grading with the grading of the Floer complex of $\tilde{H}$ which is defined in Section \ref{grading-sec}.

Recall from Section \ref{grading-sec} that for each free homotopy class $e \in [ \T , M]$ containing loops in $\Sigma$ we chose an element $y_e : \T \to  \Sigma$ such that $[y_e] = e$ and a symplectic trivialisation 
\[
\Phi_e : \T \times \R^{2n} \rightarrow y_e^*(TM)
\]
mapping $\T \times \R^{2n-2}\times \{(0,0)\}$ onto $y_e^*(\ker \alpha)$. Here we consider the induced trivialisation
\[
\bar{\Phi}_e : \T \times \R^{2n-2} \rightarrow y_e^*(\ker \alpha).
\]
If $\hat{x} = (x , \tau)$ is a critical point of $\mathcal{A}_H$ such that $[x]=e$ then we can extend the trivialisation $ \bar{ \Phi}_e$ using parallel transport along a homotopy connecting $y_e$ to $x$ to define a symplectic trivialisation of $x^*(\ker \alpha) \to \T$. We denote by 
\[
\mu_{\mathrm{rs}}(\bar{\Gamma}_{\hat{x}})\in \frac{1}{2} \Z
\]
the Robbin-Salamon index (see \cite{rs95}) of the path 
\begin{equation}
\label{path_gamma_bar}
\bar{ \Gamma}_{ \hat{x}} :[0,1] \to \mathrm{Sp}(2n-2)
\end{equation}
which is obtained by conjugating the restriction $d \phi_{X_H}^{ \tau t}(x(0))$ to the invariant symplectic subbundle given by the contact distribution $\ker \alpha$ in this symplectic trivialization.

The fact that $c_1(TM)$ vanishes on tori implies that $\mu_{\mathrm{rs}}(\bar{\Gamma}_{\hat{x}})$ does not depend on the choice of the homotopy connecting $y_e$ to $x$. Furthermore, $\mu_{\mathrm{rs}}(\bar{\Gamma}_{\hat{x}})$ depends only on the connected component $K$ of $\crit \mathcal{A}_{H}$ to which $\hat{x}$ belongs. By
\[
\bar{\mu}_{\mathrm{rs}}(K) \in \frac{1}{2} \Z
\]
we denote the common value of $\mu_{\mathrm{rs}}(\bar{\Gamma}_{\hat{x}})$ over $\hat{x}\in K$. 

We set also
\begin{equation}
\label{muK}
\mu(K) := \bar{\mu}_{\mathrm{rs}}(K) - \frac{1}{2} (\dim K-1),
\end{equation}
which by the properties of the Robbin-Salamon index is an integer. 

In the case $K=\Sigma \times \{0\}$ we set
\begin{equation}
\label{constants}
\mu(\Sigma \times \{0\}) := 1-n.
\end{equation}
The next result clarifies the relationship between this index and the one defined in Section \ref{grading-sec}. It is a particular case of more general index identities which are proved in \cite{bo13}. We include a sketch of the proof for sake of completeness.

\begin{prop}
Let $K$ be a connected component of $\mathcal{K}=\crit \mathbb{A}_{\tilde{H}}/\R^* \cong \crit \mathcal{A}_H$ different from $\Sigma\times \{0\}$ and let $\Lambda:= \pi_{\mathcal{K}}^{-1}(K)$ be the corresponding connected component of $\crit \A_{\tilde{H}}$. Then
\[
\mu_{\mathrm{rs}}(\Lambda) = \bar{\mu}_{\mathrm{rs}}(K).
\]
\end{prop}

\begin{proof}
First note that by differentiating \eqref{theflow}, we see that the linearized flow at the periodic orbit $\tilde{x}(t)=(x(t),\tau,\sigma)$ is given by
\begin{equation}
\label{linearised_flow}
d\phi_{X_{\tilde{H}}}^{t}(x(0),\tau,\sigma)=\left(\begin{array}{ccc}
d\phi_{X_H}^{\tau t}(x(0)) & tX_{H}(x(t)) & 0\\
0 & 1 & 0\\
tdH(x(0)) & 0 & 1
\end{array}\right).
\end{equation}
Recall from \eqref{the_nbd_U_of_Sigma} that there exists a neighborhood $U$ of $\Sigma$
in $M$ and a symplectomorphism
\[
(U,\omega|_{U})\cong\left(\Sigma\times(1-\delta,1+\delta),d(r\alpha)\right),
\]
where $r$ denotes the second variable in the product $\Sigma\times(1-\delta,1+\delta)$. Let $C \subset M$ denote the submanifold 
\[
C := \left\{ x(\T) \mid (x , \tau) \in K\right\}.
\]
There exists a homotopy $H^{s},s\in[0,1]$ of $H$ with the following two
properties:
\begin{enumerate}
\item $H^{0}=H$,  $X_{H^{s}}|_{\Sigma}=X_{H}|_{\Sigma}$ and $\mathcal{A}_{H^s}$ is Morse-Bott
for each $s$,
\item \label{function_h} In a neighborhood $V\subset U$ of $C$ we can write $H^{1}(p,r)=h(r)$
for $(p,r)\in V$, where $h$ is a smooth function satisfying $h(1)=0$,
and $h'(1) >0$ and $h''(1) \ne  0$. 
\end{enumerate}
See \cite[p287]{cf09}. It suffices to prove the result with $H$ replaced by $H^1$. In order to simplify the notation, for the remainder of the proof we write simply $H$ again for $H^1$. Thus from \eqref{theflow}, the flow $\phi_{X_{\tilde{H}}}^t$ of $X_{\tilde{H}}$ on $V$ is given by 
\[
\phi_{X_{\tilde{H}}}^t(p,r,\tau,\sigma)=\left(\phi_R^{\tau h'(r)t}(p),r,\tau,\sigma+th(r)\right),
\]
where $\phi_R^t : \Sigma \to \Sigma$ denotes the Reeb flow of $ \alpha$. We now compare the paths $ \Gamma_{ \tilde{x}} : [0,1] \to \mathrm{Sp}(2n+2)$ from \eqref{path_gamma} and $ \bar{ \Gamma}_{ \hat{x}} :[ 0,1] \to \mathrm{Sp}(2n-2)$ from \eqref{path_gamma_bar}. It follows from \eqref{linearised_flow} that 
\begin{equation}
\label{the symplectic matrices}
\Gamma_{ \tilde{x}}(t)=\left(\begin{array}{cc}
\bar{ \Gamma}_{ \hat{x}}(t) & 0\\
 0 & \Theta(t)
\end{array}\right),
\end{equation}
where $ \Theta :[0,1] \to \mathrm{Sp}(4)$ is given by
\[
\Theta(t)=\left(\begin{array}{cccc}
1 & \tau h''(1)t & h'(1)t & 0\\
0 & 1 & 0 & 0\\
0 & 0 & 1 & 0\\
0 & h'(1)t & 0 & 1
\end{array}\right).
\]
Now we make use of \cite[Proposition 6]{bo13}, which tells us that 
\begin{equation}
\label{index_is_zero}
\mu_{\mathrm{rs}}( \Theta) = - \frac{1}{2} \mathrm{sgn} \left(\begin{array}{cc}
\tau h''(1) & h'(1)\\
h'(1) & 0
\end{array}\right)=0,
\end{equation}
Since the Robbin-Salamon index of a block diagonal symplectic path is the sum of the Robbin-Salamon indices of the blocks, we conclude the proof with
\[
\mu_{\mathrm{rs}}( \Lambda) = \mu_{\mathrm{rs}}( \Gamma_{ \tilde{x}}) = \mu_{\mathrm{rs}}( \bar{\Gamma}_{ \hat{x}} ) + \mu_{\mathrm{rs}}( \Theta)=\mu_{\mathrm{rs}}( \bar{ \Gamma}_{ \hat{x}} ) = \bar{\mu}_{\mathrm{rs}}(K).
\]
\end{proof}

Let $K$ be a connected component of $\mathcal{K}=\crit \mathbb{A}_{\tilde{H}}/\R^* \cong \crit \mathcal{A}_H$ different from $\Sigma\times \{0\}$.
If $\hat{x}\in K$ is a critical point of the Morse function $f$, we define
\[
\mu_f^{RF}(\hat{x}) := \mu(K) + \mathrm{ind}_f(\hat{x}).
\]
Since $\dim \Lambda=\dim K +1$, we deduce that
\begin{equation}
\label{rel-indices}
\mu(\Lambda)= \mu(K)-1 \qquad \mbox{and} \qquad \mu_f(\hat{x}) = \mu_f^{RF}(\hat{x}). 
\end{equation}

Notice that the above identity holds also in the case $\Lambda = \Sigma \times \{0\} \times \R^*$. Indeed, in this case
\[
\mu(\Lambda) = \mu_{\mathrm{rs}}(\Lambda) - \frac{1}{2} \dim \Lambda = 0 - n = 1- n - 1 = \mu(K)-1,
\]
because of (\ref{constants}).

Here is the version of Proposition \ref{dimprop} in the Rabinowitz Floer setting. This result is proved in \cite[Appendix A]{cf09}) (see also \cite[Appendix A]{fra04} for the analogous result in finite-dimensional Morse-Bott theory).

\begin{prop}
\label{dimprop-rfh}
For a generic choice of $(\mathrm{J}, g)$ the manifolds of Rabinowitz Floer trajectories with cascades have dimension   
\begin{equation}
\label{dimensions-rfh}
\begin{split}
\dim \mathcal{C}_{RF}(K^-,K^+) &= \mu(K^-) + \dim K^- - \mu(K^+) -1, \\ 
\dim \mathcal{C}_{RF}(\hat{x}^-,K^+) &= \mu_f^{RF}(\hat{x}^-) - \mu(K^+)-1,  \\ 
\dim \mathcal{C}_{RF}(K^-,\hat{x}^+) &= \mu (K^-) + \dim K^- - \mu_f^{RF}(\hat{x}^+) -1, \\ 
\dim \mathcal{C}_{RF}(\hat{x}^-,\hat{x}^+) &= \mu_f^{RF}(\hat{x}^-) - \mu_f^{RF}(\hat{x}^+)-1,
\end{split}
\end{equation}
for every $K^-,K^+$ connected components of $\mathrm{crit}\, \mathcal{A}_H$ and for every $\hat{x}^-, \hat{x}^+$ in $\mathrm{crit}\, f$.
\end{prop}

The last of the formulas (\ref{dimensions-rfh}) implies that  $n_{\partial}^{RF}(\hat{x}^-,\hat{x}^+)=0$ whenever $\mu_f^{RF}(\hat{x}^+) \neq  \mu_f^{RF}(\hat{x}^-) - 1$. Therefore, if we set
\[
RF_k(\tilde{H},f) := \set{\epsilon\in RF(H,f)}{\epsilon(\hat{x}) = 0 \mbox{ if } \mu_f^{RF}(\hat{x}) \neq k}, \qquad \forall k\in \Z,
\]
we see that the boundary operator $\partial$ maps the subspace $RF_k(H,f)$ into the subspace $RF_{k-1}(H,f)$, for every $k\in \Z$. Therefore, 
\[
\{RF_k(H,f),\partial\}_{k\in \Z}
\]
is a chain complex of $\Z_2$-vector spaces. Its homology is independent on the choice of $\mathrm{J}$, $f$, $g$, and depends on $H$ only through its zero-level set $\Sigma$. It is denoted by
\[
\{RFH_k(\Sigma,M)\}_{k\in \Z}. 
\]

\subsection{The isomorphism}
\label{theisomorphismsection}
The aim of this section is to construct an isomorphism of differential $\Z_2$-vector spaces
\begin{equation}
\Phi : \mathit{RF}(H,f) \to F(\tilde{H},f). 
\end{equation}
As we have seen, the spaces $RF(H,f)$ and $F(\tilde{H},f)$ are canonically identified (and by (\ref{rel-indices}) this identification preserves the grading, when the grading is defined). However, the two differentials are defined by counting solutions of different problems, and there is no reason why the natural identification should commute with the differentials. As it now standard in these sort of questions, the correct definition of the isomorphism $\Phi$ involves counting solutions of a hybrid problems, which we now define.

Consider a pair $(v,\tilde{u})$, where 
\begin{equation}
\label{pair}
\begin{split}
v=(u^{-},\eta^{-}):(-\infty,0] &\rightarrow C^{\infty}(\mathbb{T},M)\times\mathbb{R},
\\
\tilde{u}=(u^{+},\eta^{+},\zeta^{+}):[0,+\infty) &\rightarrow C^{\infty}(\mathbb{T},\tilde{M}),
\end{split}
\end{equation}
are negative gradient flow lines of $\mathcal{A}_{H}$ and $\mathbb{A}_{\tilde{H}}$,
respectively, which satisfy the following coupling conditions:
\begin{eqnarray}
u^{-}(0,t)&=u^{+}(0,t),\qquad &\forall t\in\mathbb{T},\label{eq:loop coupling}
\\
\eta^{-}(0)&=\eta^{+}(0,t)\qquad &\forall t\in\mathbb{T}.\label{eq:eta coupling}
\end{eqnarray}
These coupling conditions imply that
\[
\begin{split}
\mathbb{A}_{\tilde{H}}(\tilde{u}(0,\cdot)) & =\int_{\mathbb{T}}u^{+}(0,\cdot)^{*}\lambda-\int_{\mathbb{T}}\zeta^{+}(0,t)\partial_{t}\eta^{+}(0,t)\, dt -\int_{\mathbb{T}}\eta^{+}(0,t)H(u^{+}(0,t))\, dt\\
 & =\int_{\mathbb{T}}u^{-}(0,\cdot)^{*}\lambda-\eta^{-}(0)\int_{\mathbb{T}}H(u^{-}(0,t))\, dt=\mathcal{A}_{H}(v(0,\cdot)),
\end{split}
\]
since $\partial_{t}\eta^{+}(0,\cdot)\equiv0$. Therefore such a pair $(v,\tilde{u})$ satisfies the action estimates
\begin{equation}
\label{ae}
\mathcal{A}_{H}(v(-s,\cdot))\geq\mathcal{A}_{H}(v(0,\cdot))=\mathbb{A}_{\tilde{H}}(\tilde{u}(0,\cdot))\geq\mathbb{A}_{\tilde{H}}(\tilde{u}(s,\cdot)),
\end{equation}
for every $s\geq 0$, and the sharp energy identity
\begin{equation}
\label{cin}
\begin{split}
\mathbb{E}(v) + \mathbb{E}(\tilde{u}) &= \lim_{s\rightarrow -\infty} \mathcal{A}_H(v(s)) - \lim_{s\rightarrow +\infty} \A_{\tilde{H}}(\tilde{u}(s)) \\
&= \sup_{s\in (-\infty,0]} \mathcal{A}_H(v(s)) - \inf_{s\in [0, +\infty)} \A_{\tilde{H}}(\tilde{u}(s)).
\end{split}
\end{equation}
These action estimates and energy identity are the starting point of the following $L^{\infty}$ bound, which is proved in Section \ref{puehp} below:

\begin{prop}
\label{bounds-hybrid}
For every $A\in \R$ there is a number $C=C(A)$ such that for every pair $(v,\tilde{u})$ as in \eqref{pair} of negative gradient flow lines
of $\mathcal{A}_{H}$ and $\mathbb{A}_{\tilde{H}}$, respectively, which satisfies the coupling conditions (\ref{eq:loop coupling}) and (\ref{eq:eta coupling}) and the action bounds
\[
\mathcal{A}_H(v(s)) \leq A \quad \forall s\leq 0, \qquad \A_{\tilde{H}}(\tilde{u}(s)) \geq -A \quad \forall s\geq 0,
\]
there holds
\[
\begin{split}
\|\eta^-\|_{L^{\infty}((-\infty,0])} &\leq C, \qquad \|\eta^+\|_{L^{\infty}([0,+\infty)\times \T)}\leq C, \qquad \|\zeta^+-\hat{\zeta}^+\|_{L^{\infty}([0,+\infty)\times \T)} \leq C, \\
&u^-((-\infty,0]\times \T) \subset M_0, \qquad u^+([0,+\infty)\times \T) \subset M_0,
\end{split}
\]
where $\hat{\zeta}^+$ indicates the average of $\zeta^+$ over $\T$, which as we know does not depend on $s\geq 0$.
\end{prop}

Fix a component $K\subset\mathrm{crit}\,\mathcal{A}_{H}$ and a component
$\Lambda\subset\mathrm{crit}\,\mathbb{A}_{\tilde{H}}$, and define
\[
\mathcal{M}_{\Phi}(K;\Lambda)
\]
to be the set of all pairs $(v,\tilde{u})$ of negative gradient flow lines of $\mathcal{A}_H$ and $\A_{\tilde{H}}$, respectively, which satisfy the coupling conditions (\ref{eq:loop coupling}), (\ref{eq:eta coupling}) and  the asymptotic conditions
\[
v(-\infty)\in K,\qquad\tilde{u}(+\infty)\in\Lambda.
\]
There is a free action of $\R^*$ on $\mathcal{M}_{\Phi}(K;\Lambda)$, which is obtained by letting the action (\ref{Raction}) act on the second component of the pair $(v,\tilde{u})$. The evaluation map
\[
\mathrm{ev}: \mathcal{M}_{\Phi}(K;\Lambda) \rightarrow K \times \Lambda, \qquad
(v,\tilde{u})\mapsto (v(-\infty),\tilde{u}(+\infty)),
\]
is equivariant with respect the $\R^*$-action on the second component of the product $K \times \Lambda$.

The sharp energy identity now reads
\begin{equation}
\label{sei}
\mathbb{E}(v) + \mathbb{E}(\tilde{u}) = \mathcal{A}_H(K) - \A_{\tilde{H}}(\Lambda),
\end{equation}
for every $(v,\tilde{u})$ in $\mathcal{M}_{\Phi}(K;\Lambda)$.
By the above proposition, together with standard arguments, the space $\mathcal{M}_{\Phi}(K;\Lambda)$ is pre-compact in the $C^{\infty}_{\mathrm{loc}}$ topology.

\begin{prop}
The set $\mathcal{M}_{\Phi}(K;\Lambda)$ is the set of zeroes of a Fredholm section.
\end{prop}

The proof of the above result is discussed in Section \ref{newindexcomp} below. Transversality follows from standard arguments, the only obstruction being the existence of stationary solutions.
Indeed, if $\Lambda=\pi_{\mathcal{K}}^{-1}(K)$, then the sharp energy identity (\ref{sei}) implies that $\mathcal{M}_{\Phi}(K;\Lambda)$ is the set of all stationary solutions
$(v_0,\tilde{u}_0)$ with 
\[
v_0(s):= (x(\cdot),\tau) \quad \forall s\leq 0, \qquad \tilde{u}_0(s,\cdot)=(x(\cdot),\tau,\sigma) \quad \forall s\geq 0,
\]
where $(x,\tau)$ is a critical point of $\mathcal{A}_H$ in $K$ and $\sigma\in \R^*$. However, automatic transversality holds in this case: see Lemma \ref{autotrans} below.

It will be convenient to enlarge the space of cascades $\mathcal{C}_{RF}(\hat{x},K)$ by allowing the final point on $K$ to flow under the action of the negative gradient flow of $f$. More precisely, when $\hat{x} \in \crit f$ does not belong to the component $K$ of $\mathcal{K}$ we define
\[
\tilde{\mathcal{C}} _{RF}(\hat{x},K) := \mathcal{C}_{RF}(\hat{x},K) \times [0,+\infty),
\]
and
\[
\tilde{\mathrm{ev}}_+ : \tilde{\mathcal{C}}_{RF}(\hat{x},K) \rightarrow K, \qquad (w,s) \mapsto \phi^s_{-\nabla f}(\mathrm{ev}_+(w)).
\]
When $\hat{x}$ belongs to $K$, we define
\[
\tilde{\mathcal{C}}_{RF}(\hat{x},K) := W^u_{-\nabla f}(\hat{x}),
\]
and define  $\tilde{\mathrm{ev}}_+ :  \tilde{\mathcal{C}}_{RF} (\hat{x},K) \rightarrow K$ to be the inclusion. 
Similarly we enlarge the set $\mathcal{C}(\Lambda, \hat{x})$ of cascades by defining
\[
\tilde{\mathcal{C}}(\Lambda,\hat{x}):=\mathcal{C}(\Lambda,\hat{x})\times(-\infty,0],\qquad\hat{x}\notin\pi_{\mathcal{K}}(\Lambda),
\]
and define 
\[
\mathrm{\tilde{ev}}_{-}:\tilde{\mathcal{C}}(\Lambda,\hat{x})\rightarrow\Lambda,\qquad(w,s)\mapsto \phi^s_{-\nabla f}(\mathrm{ev}_{-}(w)).
\]
Meanwhile if $\hat{x}\in\pi_{\mathcal{K}}(\Lambda)$ then $\tilde{\mathcal{C}}(\Lambda,\hat{x}):=W^{s}_{-\nabla \tilde{f}}(\pi_{\mathcal{K}}^{-1}(\hat{x}))$, with $\tilde{\mathrm{ev}}_{-}$ the inclusion. 

Let $\hat{x}^-$ and $\hat{x}^+$ be critical points of $f$ on $\mathcal{K}$ and let $K$ and $\Lambda$ be connected components of $\crit \mathcal{A}_H$ and $\A_{\tilde{H}}$, respectively. We now form the fibred product:
\begin{equation}
\label{fp2}
\xymatrix{\mathcal{M}_{\Phi}(\hat{x}^-;\hat{x}^+ \mid K ; \Lambda)\ar@{-->}[rr]\ar@{-->}[d] && \mathcal{M}_{\Phi}(K;\Lambda)\ar[d]^{\mathrm{ev}}\\
\tilde{\mathcal{C}}_{RF}(\hat{x}^-,K)\times\tilde{\mathcal{C}}(\Lambda,\hat{x}^+)\ar[rr]_{\hspace{0.8cm}  \tilde{\mathrm{ev}}_{+}\times\tilde{\mathrm{ev}}_{-}} && K\times\Lambda.
}
\end{equation}
Finally, we define
\[
\mathcal{M}_{\Phi}(\hat{x}^-;\hat{x}^+):=\bigcup_{K, \Lambda}\mathcal{M}_{\Phi}(\hat{x}^-,\hat{x}^+ \mid K,; \Lambda),
\]
where the union is taken over all tuples $(K, \Lambda)$ of components. By taking such a union, we obtain an object which is actually a smooth manifold, without corners and without boundary. Indeed, consider an element 
\[
\psi = \bigl(\psi^-, \psi^+ \bigr)
\] 
of 
\[
\mathcal{M}_1 := \mathcal{M}_{\Phi}(\hat{x}^-;\hat{x}^+ \mid K ;\Lambda),
\] 
and assume that exactly one of the $\psi^{\pm}$'s - say $\psi^-$ - belongs to the boundary of the corresponding manifold $\tilde{\mathcal{C}}_{RF}(\hat{x}^-, K)$. This means that $\psi^-$ is of the form $(w,0)$, where 
\[
w = \bigl([v_1],\dots,[v_k] \bigr), 
\]
is a negative gradient flow line with cascades, with $k\geq 1$, and the non-stationary negative gradient flow line $v_k = (u_k, \eta_k)$ satisfies
\[
v_k (+\infty) = \mathrm{ev}_+(w) = \tilde{\mathrm{ev}}_+(\psi^-) \qquad \mbox{and} \qquad v_k(-\infty) \in K_1,
\]
for a suitable component $K_1$ of $\crit \mathcal{A}_H$. Then a gluing argument analogous to the one which is used in the proof of \cite[Theorem A.12]{fra04} shows that $\psi$ is a limiting point of the space
\[
\mathcal{M}_2 :=
\mathcal{M}_{\Phi}(\hat{x}^-,;\hat{x}^+ \mid K_1 ; \Lambda),
\]
and that the union $\mathcal{M}_1 \cup \mathcal{M}_2$ is regular at the point $\psi$. Analogous results hold in general when both more of the $\psi^{\pm}$'s belong to boundary components of the corresponding manifolds of cascades. We conclude that $\mathcal{M}_{\Phi}(\hat{x}^-;\hat{x}^+)$ is a finite dimensional manifold without boundary and corners. 

The coefficient
\begin{equation}
\label{shows isomorphism works on filtered level}
n_{\Phi}(\hat{x}^-,\hat{x}^+) \in \Z_2  
\end{equation}
is defined to be the parity of the zero-dimensional part of the quotient $\mathcal{M}_{\Phi}(\hat{x}^-;\hat{x}^+)/\R^*$. The sharp energy identity (\ref{sei}) implies that
\[
n_{\Phi}(\hat{x}^-,\hat{x}^+) = 0 \qquad \mbox{if } \mathcal{A}_H(\hat{x}^-) < \A_{\tilde{H}}(\hat{x}^+).
\]
Moreover, if $\mathcal{A}_H(\hat{x}^-) = \A_{\tilde{H}}(\hat{x}^+)$ then $n_{\Phi}(\hat{x}^-,\hat{x}^+)=0$, unless $\hat{x}^-=\hat{x}^+$, in which case the automatic transversality Lemma \ref{autotrans} implies that
\[
n_{\Phi}(\hat{x}^-,\hat{x}^+)=1.
\]
The homomorphism
\[
\Phi:RF(H,f)\rightarrow F(\tilde{H},f)
\]
is defined by
\[
(\Phi \epsilon )(\hat{x}^+) = \sum_{\hat{x}^-\in \crit f} n_{\Phi}(\hat{x}^-,\hat{x}^+)\epsilon(\hat{x}^-), \qquad \forall \epsilon\in RF(H,f), \; \forall \hat{x}^+ \in \crit f.
\]
A standard gluing and compactness argument implies that $\Phi$ commutes with the two differentials. Moreover, the triangular property of the coeffcient matrix
\[
\bigl(n_{\Phi}(\hat{x}^-,\hat{x}^+)\bigr)_{(\hat{x}^-,\hat{x}^+) \in (\crit f)^2}
\]
implies that $\Phi$ is an isomorphism with inverse
\[
\Phi^{-1} \epsilon  (\hat{x}^+) = \sum_{\hat{x}^-\in \crit f} m(\hat{x}^-,\hat{x}^+)\epsilon(\hat{x}^-), \qquad \forall \epsilon\in F(\tilde{H},f), \; \forall \hat{x}^+ \in \crit f,
\]
where the coefficients $m(\hat{x}^-,\hat{x}^+)$ are defined recursively by
\[
m(\hat{x}^-,\hat{x}^+) = \left\{ \begin{array}{ll} 0 &\mbox{if } \mathcal{A}_H(\hat{x}^+) \geq \A_{\tilde{H}}(\hat{x}^-) \mbox{ and } \hat{x}^+\neq \hat{x}^-, \\ 1 &\mbox{if } \hat{x}^+=\hat{x}^-,\\ \displaystyle{\sum_{\hat{x}\in \crit f} n_{\Phi}(\hat{x},\hat{x}^+) m(\hat{x}^-,\hat{x})} & \mbox{otherwise.} \end{array} \right.
\]
\begin{rem}
It is also possible to construct an chain map $\Psi:F(\tilde{H},f)\rightarrow RF(H,f)$
starting from an analogous space $\mathcal{M}_{\Psi}(\Lambda;K)$. This $\Psi$ is an isomorphism by the same arguments, and in fact is a homotopy inverse to $\Phi$.
\end{rem}

\subsection{Degree of the isomorphism}

Here we assume that $c_1(TM)$ vanishes on tori, so that the graded complexes $RF_*(H,f)$ and $F_*(\tilde{H},f)$ are well-defined. In this case, by Theorem \ref{newindexcompthm}, we have 
\begin{equation}
\label{indiso}
\dim \mathcal{M}_{\Phi}(K;\Lambda) =  \mu(K) + \dim K- \mu(\Lambda),
\end{equation}
and the same is true for the manifold $\mathcal{M}_{\Phi}(\hat{x}^-;\hat{x}^+)$. Therefore its quotient by the $\R^*$-action has dimension
\[
\dim \mathcal{M}_{\Phi}(\hat{x}^-;\hat{x}^+)/\R^* = \mu_f^{RF}(\hat{x}^-)    - \mu_f(\hat{x}^+) .
\]
We conclude that $n_{\Phi}(\hat{x}^-,\hat{x}^+)$ can be non-zero only if $\mu_f(\hat{x}^+) = \mu_f^{RF}(\hat{x}^-) $, so the isomorphism $\Phi$ is grading preserving:
\[
\Phi : RF_k(H,f) \rightarrow F_{k}(\tilde{H},f).
\]
Notice that this is consistent with identity (\ref{rel-indices}) and the fact that $n_{\Phi}(\hat{x},\hat{x})=1$.

\section{Uniform estimates}
\label{estisec}

The next two parts of this paper are technical in nature. In this one we prove the uniform estimates needed to define the Floer complex for the Hamiltonian $\tilde{H}$ on $\tilde{M}$ and those  needed to  define the chain complex isomorphism between the Rabinowitz Floer homology of the pair $(H^{-1}(0),M)$ and the Floer homology of $\tilde{H}$. 

The proof of these uniform estimates combines ideas from the proof of \cite[Proposition 3.2]{cf09} with elliptic estimates and is split in several Lemmata.

\subsection{Loops where the action functional has a small gradient} 

We begin by proving two Lemmata about some properties of loops $\tilde{x}$ in $\tilde{M}$ where the gradient of $\mathbb{A}_{\tilde{H}}$ has a small $L^2$-norm.

\begin{lem}
\label{lem1}
Let $\tilde{x}=(x,\tau,\sigma)\in C^{\infty}(\T,\tilde{M})$ be such that
\begin{equation}
\label{hypo}
\|\nabla \mathbb{A}_{\tilde{H}} (\tilde{x}) \|_{L^2(\T)} < \frac{h}{2} \min \{ 1/\|X_H\|_{L^{\infty}(M)},1\},
\end{equation}
for some $h>0$. Then
\[
|H(x(t))| < h \qquad \forall t\in \T.
\]
\end{lem}

\noindent Here the $L^{\infty}$-norm of $X_H$ is defined by
\[
\|X_H\|_{L^{\infty}(M)} := \max_{\substack{x\in M \\ t\in \T}} |X_H(x)|_{J_t}.
\]

\begin{proof}
We start by proving the bound
\begin{equation}
\label{uno}
\min_{t\in \T} |H(x(t))| \leq \|\nabla \mathbb{A}_{\tilde{H}} (\tilde{x}) \|_{L^2(\T)}.
\end{equation}
If $t\mapsto H(x(t))$ changes sign, then the above minimum vanishes and (\ref{uno}) is trivially true. Assume w.l.o.g. that $H(x(t))>0$ for every $t\in \T$. Then
\[
\begin{split}
\min_{t\in \T} |H(x(t))| &= \min_{t\in \T} H(x(t)) \leq \int_{\T} H(x)\, dt = \int_{\T} ( H(x)-\sigma')\, dt \leq \|\sigma'-H(x)\|_{L^2(\T)} \\ &\leq  \|\nabla \mathbb{A}_{\tilde{H}} (\tilde{x}) \|_{L^2(\T)},
\end{split}
\]
where we have used the expression (\ref{nabla}) for the gradient of the action functional. This proves (\ref{uno}).
  
By (\ref{hypo}) and (\ref{uno}) we get
\[
\min_{t\in \T} |H(x(t))| < \frac{h}{2}.
\]
If it is not true that $|H(x)|<h$ on $\T$, then by the above fact we can find an interval $[t_0,t_1]\subset \T$ such that
\[
\frac{h}{2} \leq |H(x(t))| \leq h \quad \forall t\in [t_0,t_1], \qquad \mbox{and} \quad |H(t_1) - H(t_0)|  = \frac{h}{2}.
\]
Therefore, we can estimate
\[
\begin{split}
\frac{h}{2} &=  |H(t_1) - H(t_0)| = \left| \int_{t_0}^{t_1} \frac{d}{dt} H(x(t))\, dt \right| \leq \int_{t_0}^{t_1} | dH(x)[x']|\, dt \\
&= \int_{t_0}^{t_1} |\omega (X_H(x),x')|\, dt = \int_{t_0}^{t_1} |\omega(X_H(x), x' - \tau X_H(x))|\, dt \\ &\leq \int_{t_0}^{t_1} |X_H(x)|_{J_t(x,\tau)} |x' - \tau X_H(x)|_{J_t(x,\tau)} \, dt \leq \|X_H\|_{L^{\infty}(M)} \int_{\T} |x' - \tau X_H(x)|_{J_t(x,\tau)} \, dt \\ &\leq   \|X_H\|_{L^{\infty}(M)} \|J_{\cdot}(x)(x' -\tau  X_H(x))\|_{L^2(\T)} \leq \|X_H\|_{L^{\infty}(M)} 
\|\nabla \mathbb{A}_{\tilde{H}} (\tilde{x}) \|_{L^2(\T)},
\end{split}
\]
where we have used again (\ref{nabla}). This shows that 
\[
\|\nabla \mathbb{A}_{\tilde{H}} (\tilde{x}) \|_{L^2(\T)} \geq \frac{h}{2 \|X_H\|_{L^{\infty}(M)}}
\]
and contradicts the hypothesis.
\end{proof}

By (\ref{alpha0}) and since $H$ is bounded away from zero outside from a compact set, we can find $h>0$ so small that
\begin{equation}
\label{epsilon}
N:=H^{-1}([-h,h]) \mbox{ is compact and }
\lambda(X_H) \geq \frac{2}{3} \alpha_0 \mbox{ on $N$.} 
\end{equation}

\begin{lem}
\label{tau}
Let $h\leq \alpha_0/3$ be such that (\ref{epsilon}) holds. For any $A,S\in \R$ there exists a number $L =L(A,S)$ such that if $\tilde{x}=(x,\tau,\sigma)$ satisfies
\[
\|\nabla \mathbb{A}_{\tilde{H}}(\tilde{x}) \|_{L^2(\T)} < \frac{h}{2} \min \{ 1/\|X_H\|_{L^{\infty}(M)},1\}, \qquad 
|\mathbb{A}_{\tilde{H}}(\tilde{x})| \leq A, \qquad  \|\sigma - \hat{\sigma}\|_{L^2(\T)} \leq S,
\]
where $\hat\sigma:=\int_{\T} \sigma(t)\, dt$ denotes the average of $\sigma$,
then 
\[
\|\tau\|_{L^{\infty}(\T)} \leq L.
\]
\end{lem}

\begin{proof}
From  (\ref{nabla}) we deduce the following bound on $\tau'$:
\begin{equation}
\label{deri}
\|\tau'\|_{L^2(\T)} \leq \|\nabla \mathbb{A}_{\tilde{H}}(\tilde{x})\|_{L^2(\T)} \leq \frac{h}{2} \leq \alpha_0.
\end{equation}
Moreover, in the expression
\[
\mathbb{A}_{\tilde{H}} (\tilde{x}) = \int_{\T} \tilde{x}^* \tilde{\lambda} - \int_{\T} \tilde{H}(\tilde{x}) \, dt = \int_{\T} x^* \lambda - \int_{\T} \sigma \tau' dt - \int_{\T} \tau H(x)\, dt, 
\]
the middle term in the right-hand side has the bound
\[
\left| \int_{\T} \sigma \tau' dt \right| = \left| \int_{\T} (\sigma - \hat\sigma) \tau' dt \right| \leq \|\sigma-\hat{\sigma}\|_{L^2(\T)} \|\tau'\|_{L^2(\T)} \leq \alpha_0 S.
\]
Therefore, from the bound on $\mathbb{A}_{\tilde{H}}(\tilde{x})$ we deduce the estimate
\begin{equation}
\label{due}
\left| \int_{\T} x^* \lambda - \int_{\T} \tau H(x)\, dt \right| \leq A + \alpha_0 S.
\end{equation}
From the bound on $\nabla \mathbb{A}_{\tilde{H}}(\tilde{x})$ and from Lemma \ref{lem1} we deduce that 
\[
x(t) \in N = H^{-1}([-h,h]) \qquad \forall t\in \T.
\]
In the right-hand side of the identity
\begin{equation}
\label{tre}
\int_{\T} x^* \lambda - \int_{T} \tau H(x)\, dt = \int_{\T} \lambda (x'-\tau X_H(x))\, dt + \int_{\T} \tau (\lambda(X_H(x)) - H(x))\, dt
\end{equation}
the first integral has the bound
\[
\begin{split}
\left| \int_{\T} \lambda (x'-\tau X_H(x))\, dt \right| &\leq \|\lambda\|_{L^{\infty}(N)} \int_{\T} |x' - \tau X_H (x)|_{J_t(x,\tau)} \, dt \leq  \|\lambda\|_{L^{\infty}(N)} \|\nabla \mathbb{A}_{\tilde{H}}(\tilde{x}) \|_{L^2(\T)} \\ & \leq \frac{h \|\lambda\|_{L^{\infty}(N)}}{2 \|X_H\|_{L^{\infty}(M)}}=: B.
\end{split}
\]
By the above estimate and (\ref{due}), identity (\ref{tre}) implies that
\begin{equation}
\label{intau}
\left|  \int_{\T} \tau (\lambda(X_H(x)) - H(x))\, dt \right| \leq A + \alpha_0 S + B.
\end{equation}
Since $x$ takes values in $N$, (\ref{epsilon}) implies that
\begin{equation}
\label{disco}
\lambda (X_H(x)) - H(x) \geq \frac{2}{3} \alpha_0 - h \geq \frac{\alpha_0}{3},
\end{equation}
where we have used the upper bound on $h$.
We claim that 
\begin{equation}
\label{mini}
\min_{t\in \T} |\tau(t)| \leq \frac{3}{\alpha_0} (A + \alpha_0 S + B).
\end{equation}
Indeed, the above bound holds trivially if $\tau$ changes sign. W.l.o.g. $\tau$ is positive, and in this case we deduce from (\ref{intau}) and (\ref{disco})
\[
A + \alpha_0 S + B \geq \int_{\T} \tau (\lambda(X_H(x)) - H(x))\, dt \geq \frac{\alpha_0}{3} \int_{\T} \tau\, dt \geq \frac{\alpha_0}{3} \min_{t\in \T} \tau(t) = \frac{\alpha_0}{3} \min_{t\in \T} |\tau(t)|,
\]
which proves (\ref{mini}). The inequalities (\ref{deri}) and (\ref{mini}) imply the desired bound on the uniform norm of $\tau$:
\[
\|\tau\|_{L^{\infty}(\T)} \leq \frac{3}{\alpha_0} (A +  \alpha_0 S + B) + \alpha_0 := L.
\]
\end{proof} 

\subsection{Solutions of the Floer equation on cylinders} 

Now let $I\subset \R$ be an unbounded open interval and let 
\[
\tilde{u}=(u,\eta,\zeta) : I \times \T \rightarrow \tilde{M} = M \times T^* \R
\]
be a solution of the Floer equation (\ref{floer}) with uniformly bounded action:
\begin{equation}
\label{ba}
|\mathbb{A}_H(\tilde{u}(s))| \leq A \qquad \forall s\in I.
\end{equation}
In particular, $\tilde{u}$ has bounded energy:
\begin{equation}
\label{energy}
\begin{split}
\mathbb{E}(\tilde{u}) &= \int_{I} \| \partial_s \tilde{u} \|_{L^2(\T)}^2 \, ds =  \int_{I} \| \nabla \mathbb{A}_{\tilde{H}}( \tilde{u}) \|_{L^2(\T)}^2 \, ds \\ &= \lim_{s\rightarrow \inf I} \mathbb{A}_{\tilde{H}}(\tilde{u}(s)) - \lim_{s\rightarrow \sup I} \mathbb{A}_{\tilde{H}}(\tilde{u}(s)) \leq 2A =:E.
\end{split}
\end{equation}
As we have seen, the last equation in (\ref{floer}) implies that time average of $\zeta$ is constant:
\[
\int_{\T} \zeta(s,t)\, dt = \hat{\zeta} \qquad \forall s\in I.
\]
The bound on the energy allows us to get a first bound on $\zeta$:

\begin{lem}
\label{rho}
For every $s\in I$ there holds
\[
\| \zeta(s,\cdot)- \hat\zeta\|_{L^2(\T)} \leq 2 \sqrt{E} + \|H\|_{L^{\infty}(M)}.
\] 
\end{lem}

\begin{proof}
Consider the set
\[
\mathcal{S} = \mathcal{S}(\tilde{u}) := \set{s\in I}{\|\nabla \mathbb{A}_{\tilde{H}} (\tilde{u}(s)) \|_{L^2(\T)} \leq \sqrt{E}}.
\]
By Chebichev's inequality, the complement of $\mathcal{S}$ in $I$ has uniformly bounded measure:
\[
|I \setminus \mathcal{S}| \leq \frac{1}{E} \int_{I} \|\nabla \mathbb{A}_{\tilde{H}_{
+}} (\tilde{u}(s)) \|_{L^2(\T)}^2 \, ds =  \frac{1}{E} \mathbb{E}(\tilde{u}) \leq 1,
\]
where we have used (\ref{energy}). In particular, $I \setminus \mathcal{S}$ contains no intervals of length larger than 1. Therefore,
given $s\in I$, we can find $s_0\in \mathcal{S}$ such that  $|s-s_0|\leq 1$.
By (\ref{nabla}),
\[
\begin{split}
\|\partial_t \zeta(s_0,\cdot) \|_{L^2(\T)} &\leq \|\partial_t \zeta(s_0,\cdot) - H(u(s_0,\cdot)) \|_{L^2(\T)} + \|H(u(s_0,\cdot)) \|_{L^2(\T)} \\ &\leq \| \nabla \mathbb{A}_{\tilde{H}} (\tilde{u}(s_0,\cdot)) \|_{L^2(\T)} + \|H\|_{L^{\infty}(M)} \leq \sqrt{E} + \|H\|_{L^{\infty}(M)}.
\end{split}
\]
Since $\zeta(s_0,\cdot)$ has mean $\hat\zeta$, the Poincar\'e inequality implies that
\[
\|\zeta(s_0,\cdot)-\hat\zeta\|_{L^2(\T)} \leq \|\partial_t \zeta(s_0,\cdot) \|_{L^2(\T)} \leq  \sqrt{E} + \|H\|_{L^{\infty}(M)}.
\]
From the above estimate and from the bound
\[
\|\partial_s \zeta\|_{L^2(I \times \T)} \leq \|\partial_s \tilde{u}\|_{L^2(I \times \T)} = \sqrt{\mathbb{E}(\tilde{u})} \leq \sqrt{E},
\]
we find
\[
\begin{split}
\|\zeta(s,\cdot)-\hat\zeta\|_{L^2(\T)} &= \|\zeta(s_0,\cdot)-\hat\zeta\|_{L^2(\T)} + \int_{s_0}^s \frac{d}{d\sigma} \|\zeta(\sigma,\cdot)\|_{L^2(\T)} \, d\sigma \\ &\leq \sqrt{E} + \|H\|_{L^{\infty}(M)} + \left| \int_{s_0}^s  \Bigl\| \frac{d}{d\sigma} \zeta(\sigma,\cdot) \Bigr\|_{L^2(\T)} \, d\sigma \right| \\
&= \sqrt{E} + \|H\|_{L^{\infty}(M)} + \left| \int_{s_0}^s  \Bigl( \int_{\T} | \partial_s \zeta(\sigma,t)|^2\, dt \Bigr)^{1/2}\, d\sigma \right| \\ &\leq \sqrt{E} + \|H\|_{L^{\infty}(M)} + |s-s_0|^{1/2} \left| \int_{s_0}^s \int_{\T} |\partial_s \zeta|^2 \, dt\, d\sigma \right|^{1/2} \\
&\leq \sqrt{E} + \|H\|_{L^{\infty}(M)} + \|\partial_s \zeta\|_{L^2(I \times \T)} \\ &\leq \sqrt{E} + \|H\|_{L^{\infty}(M)} + \sqrt{E} = 2 \sqrt{E} + \|H\|_{L^{\infty}(M)},
\end{split}
\]
as claimed.
\end{proof}

A similar argument allows us to prove the following lemma:

\begin{lem}
\label{eta}
For every $s\in I$ there holds
\[
\| \eta(s,\cdot)\|_{L^2(\T)} \leq L + \frac{2E}{h},
\] 
where $L=L(A, 2\sqrt{E} + \|H\|_{L^{\infty}(M)})$ is given by Lemma \ref{tau}, and $h\leq \alpha_0/3$ is such that (\ref{epsilon}) holds.
\end{lem}

\begin{proof}
Consider the set
\begin{equation}
\label{the set mathcal S}
\mathcal{S} = \mathcal{S}(\tilde{u}) := \set{s\in I}{\|\nabla \mathbb{A}_{\tilde{H}} (\tilde{u}(s)) \|_{L^2(\T)} < \frac{h}{2} \min \{1/\|X_H\|_{L^{\infty}(M)},1\}}.
\end{equation}
By Chebichev's inequality, the complement of $\mathcal{S}$ in $I$ has uniformly bounded measure:
\[
|I \setminus \mathcal{S}| \leq \frac{4}{h^2} \int_{I} \|\nabla \mathbb{A}_{\tilde{H}} (\tilde{u}(s)) \|_{L^2(\T)}^2 \, ds =  \frac{4}{h^2} \mathbb{E}(\tilde{u}) \leq \frac{4E}{h^2},
\]
where we have used (\ref{energy}). In particular, $I \setminus \mathcal{S}$ contains no intervals of length larger than $4E/h^2$.
Therefore,
given $s\in I$, we can find $s_0\in \mathcal{S}$ such that  
\[
|s-s_0|\leq \frac{4E}{h^2}.
\]
Since $s_0\in \mathcal{S}$, Lemmata \ref{tau} and \ref{rho} imply the following bound on $\eta(s_0,0)$:
\[
\|\eta(s_0,\cdot)\|_{L^2(\T)} \leq L = L(A, 2\sqrt{E} + \|H\|_{L^{\infty}(M)}).
\]
From the above estimate and from the bound
\[
\|\partial_s \eta\|_{L^2(I \times \T)} \leq \|\partial_s \tilde{u}\|_{L^2(I \times \T)} = \sqrt{\mathbb{E}(\tilde{u})} \leq \sqrt{E},
\]
we find
\[
\begin{split}
\|\eta(s,\cdot)\|_{L^2(\T)} &= \|\eta(s_0,\cdot)\|_{L^2(\T)} + \int_{s_0}^s \frac{d}{d\sigma} \|\eta(\sigma,\cdot)\|_{L^2(\T)} \, d\sigma \\ &\leq L + \left| \int_{s_0}^s  \Bigl\| \frac{d}{d\sigma} \eta(\sigma,\cdot) \Bigr\|_{L^2(\T)} \, d\sigma \right| \\
&\leq L + |s-s_0|^{1/2} \|\partial_s \eta\|_{L^2(I \times \T)} \\ &\leq L + \frac{2\sqrt{E}}{h} \sqrt{E} = L + \frac{2E}{h},
\end{split}
\]
as claimed.
\end{proof}

We can now use elliptic estimates to prove local $W^{1,p}$ bounds on the components $\eta$ and $\zeta$ of $\tilde{u}$:

\begin{lem}
\label{ubetarho}
Let $p>1$ be a real number.
There is a number $C=C(A,p)$ such that for every interval $I_0\subset I$ of length 1 whose distance from the complement of $I$ is at least 1 there holds
\[
\|\eta\|_{W^{1,p}(I_0 \times \T)} \leq C \qquad \mbox{and} \qquad
\|\zeta-\hat\zeta\|_{W^{1,p}(I_0 \times \T)} \leq C
\]
\end{lem}

\begin{proof}
Consider the smooth function
\[
f : I \times \T \rightarrow \C, \qquad f := \zeta-\hat\zeta + i \eta.
\]
In order to prove the theorem, we have to show that $\|f\|_{W^{1,p}(I_0 \times \T)}$ has a uniform bound.
By Lemmata \ref{rho} and \ref{eta}, 
\begin{equation}
\label{boundf}
\|f(s,\cdot)\|_{L^2(\T)} \leq F, \qquad \forall s\in \R,
\end{equation}
where
\[
F = F(A) := 2 \sqrt{2A} + \|H\|_{L^{\infty}(M)} + L + \frac{4 A}{h}.
\]
By (\ref{floer}), $f$ satisfies the equation
\begin{equation}
\label{cr}
\overline{\partial} f = iH(u),
\end{equation}
where
\[
\overline{\partial} = \partial_s + i \partial_t
\]
is the Cauchy-Riemann operator on the cylinder $\R \times \T$. The desired uniform bound on $\|f\|_{W^{1,p}(I_0\times \T)}$ will follow from (\ref{boundf}) and (\ref{cr}), thanks to the Calderon-Zygumund estimate
\begin{equation}
\label{cz}
\|f\|_{W^{1,p}(J_0\times \T)} \leq c_p(J_0,J) \left( \|\overline{\partial} f\|_{L^p(J \times \T)} + \|f\|_{L^2(J \times \T)} \right),
\end{equation}
where $J_0$ is a bounded interval whose closure is contained in the open interval $J$. Indeed, given an interval $I_0\subset I$ of length 1 as in the hypothesis, consider the interval
\[
I_1 := I_0 + (-1,1)
\]
of length 3, which by the assumption on the distance of $I_0$ from the complement of $I$ is still contained in $I$. By (\ref{cz}), (\ref{cr}) and (\ref{boundf}) we have
\[
\begin{split}
 \|f\|_{W^{1,p}(I_0\times \T)} &\leq c_p(I_0,I_1) \left( \|\overline{\partial} f\|_{L^p(I_1 \times \T)} + \|f\|_{L^2(I_1 \times \T)} \right)\\ &\leq c_p(I_0,I_1) \left( \|i H\circ u\|_{L^p(I_1 \times \T)} + \sqrt{3} F \right) \\  &\leq c_p(I_0,I_1) \left( 3^{1/p} \|H\|_{L^{\infty}(M)} + \sqrt{3} F \right).
\end{split}
\]
Therefore the desired estimate holds with
\[
C = C(A,p) := c_p((0,1),(-1,2)) \left( 3^{1/p} \|H\|_{L^{\infty}(M)} + \sqrt{3} F(A) \right).
\]
\end{proof}

\subsection{Proof of the uniform estimate for solutions on cylinders}
\label{completing_first_uniform_result}

We can now prove the uniform estimate which is stated in Section \ref{floeq}.

\begin{proof}[Proof of Proposition \ref{estimate}] Let $\tilde{u} = (u,\eta,\zeta)$ be a solution of (\ref{floer}) on $\R\times \T$ with
\[
|\mathbb{A}_{\tilde{H}}(\tilde{u}(s))| \leq A \qquad \forall s\in \R.
\]
We shall make use of the above Lemmata with $I=\R$. 

Consider the open subset of $\R \times \T$:
\[
\Omega:=u^{-1}(M\setminus M_0) = u^{-1} \bigl(  \iota \bigl( \Sigma_{\infty} \times (0,+\infty) \bigr).
\]
If $h>0$ is small enough, then $H^{-1}([-h,h])$ is contained in $M_0$. Therefore, by Lemma \ref{lem1} together with the fact that the energy of $\tilde{u}$ is bounded, we deduce that every connected component of $\Omega$ is bounded. The fact that $X_H$ vanishes on $\iota(\Sigma_{\infty} \times (0,+\infty))$ implies that the map $u$ satisfies the pure Cauchy-Riemann equation
\[
\partial_s u + J(u) \partial_t u = 0
\]
on $\Omega$ (note by assumption on $\Omega$ the almost complex structure $J_t(x, \tau) = J(x)$ does not depend on $t$ or $\tau$). A standard computation involving \eqref{ct} implies that the scalar function $r := r\circ u: \Omega \rightarrow \R$ is subharmonic:
\[
\Delta r = |\partial_s u|^2_{J(u)} \geq 0 \qquad \mbox{on } \Omega.
\]
By the definition of $\Omega$, $r$ extends continuously to $\overline{\Omega}$ and takes the value $0$ on $\partial \Omega$. Since each component of $\Omega$ is bounded, the maximum principle implies that the subharmonic function $r\leq 0$ on $\Omega$. On the other hand, $r>0$ on $\Omega$, and we conclude that $\Omega$ must be empty. This shows that all the maps $u$ take value in the compact set $M_0$. 

Let $I_0\subset \R$ be an interval of length 1. Lemma  \ref{ubetarho} and the fact that $W^{1,p}(I_0\times \T)$ continuously embeds into $L^{\infty}(I_0 \times \T)$ when $p>2$  imply that $\eta$ and $\zeta-\hat\zeta$ are uniformly bounded on $I_0\times \T$. Since $I_0$ is arbitrary, we get a uniform bound on the $L^{\infty}$ norm of $\eta$ and $\zeta-\hat\zeta$ on the whole of $\R \times \T$.
This concludes the proof of Proposition \ref{estimate}.
\end{proof}

\subsection{Proof of the uniform estimate for solutions of the hybrid problem}
\label{puehp}

Fix a real number $A$ and consider a pair $(v,\tilde{u})$ of negative gradient flow lines
\[
\begin{split}
v=(u^-,\eta^-) : (-\infty,0] &\rightarrow C^{\infty}(\T,M) \times \R, \\
\tilde{u}:(u^+,\eta^+,\zeta^+) : [0,+\infty) &\rightarrow C^{\infty}(\T,\tilde{M}) = C^{\infty}(\T,M) \times C^{\infty}(\T,\R) \times C^{\infty}(\T,\R^*),
\end{split}
\]
of $\mathcal{A}_H$ and $\A_{\tilde{H}}$, respectively, which satisfy the boundary conditions (\ref{eq:loop coupling}), (\ref{eq:eta coupling}) and the action bounds
\[
\mathcal{A}_H(v(s)) \leq A \quad \forall s\leq 0, \qquad \A_{\tilde{H}}(\tilde{u}(s)) \geq -A \quad \forall s\geq 0.
\]
Denote the time average of $\zeta^+$, which we know to be independent from $s\geq 0$, by $\hat{\zeta}^+$:
\[
\hat{\zeta}^+ := \int_{\T} \zeta^+(s,t)\, dt, \qquad \forall s\geq 0.
\]
Our aim in this section is to prove Proposition \ref{bounds-hybrid}, that is to show that there exists a number $C=C(A)$ such that
\[
\begin{split}
\|\eta^-\|_{L^{\infty}(-\infty,0])} &\leq C, \qquad \|\eta^+\|_{L^{\infty}([0,+\infty)\times \T)}\leq C, \qquad \|\zeta^+-\hat{\zeta}^+\|_{L^{\infty}([0,+\infty)\times \T)} \leq C, \\
&u^-((-\infty,0]\times \T) \subset M_0, \qquad u^+([0,+\infty)\times \T) \subset M_0.
\end{split}
\]
We start by observing that the above action bounds and (\ref{ae}) imply
\begin{equation}
\label{nae}
-A \leq \A_{\tilde{H}}(\tilde{u}(s)) \leq \mathcal{A}_H(v(-s)) \leq A, \qquad \forall s\geq 0.
\end{equation}
Moreover, the sharp energy identity (\ref{cin}) implies that the energy of $v$ and that of $\tilde{u}$ are uniformly bounded:
\[
\mathbb{E}(v) + \mathbb{E}(\tilde{u}) \leq 2A.
\]
The next lemma is a direct consequence of the main $L^{\infty}$ estimate for the component $\eta$ of solutions of the Rabinowitz Floer equation which is proved in \cite{cf09}:

\begin{lem}
There is a number $T=T(A)$ such that $\|\eta^-\|_{L^{\infty}((-\infty,0])} \leq T(A)$.
\end{lem}

\begin{proof}
By \cite[Proposition 3.2]{cf09} there are positive numbers $\epsilon$ and $c$ such that for every pair $(x,\tau)\in C^{\infty}(\T,M) \times \R$ there holds
\begin{equation}
\label{cifra}
\|\nabla \mathcal{A}_H(x,\tau)\|_{L^2(\T)} < \epsilon \qquad \Rightarrow \qquad |\tau| \leq c\, ( | \mathcal{A}_H(x,\tau)| + 1).
\end{equation}
Let $\mathcal{S} \subset (-\infty,0]$ be the set of $s$'s such that 
\[
\|\nabla \mathcal{A}_H(v(s)) \|_{L^2(\T)} < \epsilon.
\]
Since
\[
\int_{-\infty}^0 \|\nabla \mathcal{A}_H(v(s)) \|_{L^2(\T)}^2 \, ds = \mathbb{E}(v) \leq 2A,
\]
Chebichev's inequality implies that the complement of $\mathcal{S}$ in $(-\infty,0]$ has uniformly bounded measure. Therefore, every $s\in (-\infty,0]$ has a distance less than some constant $S=S(A)$ from an element $s_0\in \mathcal{S}$. For such a $s_0$, (\ref{cifra}) and (\ref{nae}) imply
\[
|\eta^-(s_0)| \leq c (A+1),
\]
and the second equation in (\ref{eq:RFH equations}) gives us the bound
\[
|\eta^-(s)| \leq |\eta^-(s_0)| + |s-s_0| \Bigl\| \frac{d\eta^-}{ds} \Bigr\|_{L^{\infty}((-\infty,0])} \leq c (A+1) + S \|H\|_{L^{\infty}(M)}.
\]
The conclusion follows.
\end{proof}

We can now prove the $L^{\infty}$ bounds for the solutions of the hybrid problem:

\begin{proof}[Proof of Proposition \ref{bounds-hybrid}]
By the above lemma, $\|\eta^-\|_{L^{\infty}((-\infty,0])}$
has a uniform bound. Fix some $p>2$. By Lemma \ref{ubetarho}, for every interval $I_0\subset [1,+\infty)$ of length one there are uniform bounds
\begin{equation}
\label{sobbo}
\|\eta^+\|_{W^{1,p}(I_0 \times \T)} \leq C_1, \qquad \|\zeta^+-\hat{\zeta}^+\|_{W^{1,p}(I_0 \times \T)} \leq C_1,
\end{equation}
for some $C_1=C_1(A)$. By covering $[1,+\infty)$ with intervals of length one and by using the Sobolev embedding theorem, we deduce the uniform estimates
\begin{equation}
\label{hhh}
\|\eta^+\|_{L^{\infty}([1,+\infty) \times \T)} \leq C_2, \qquad \|\zeta^+-\hat{\zeta}^+\|_{L^{\infty}([1,+\infty) \times \T)} \leq C_2.
\end{equation}
In order to obtain the $L^{\infty}$ bounds on $\eta^+$ and $\zeta^+ - \hat{\zeta}^+$, there remains to bound
\[
\|\eta^+\|_{L^{\infty}([0,1]\times \T)} \qquad \mbox{and} \quad \|\zeta^+-\hat{\zeta}^+\|_{L^{\infty}([0,1]\times \T)}.
\]
Since $\eta^-(0)$ is uniformly bounded, it is enough to show that the complex function
\[
f(s,t) :=  \zeta^+(s,t) - \hat{\zeta}^+\ + i (\eta^+(s,t) -\eta^-(0))
\]
has a uniformly bounded $L^{\infty}$ norm on $[0,1]\times \T$. Notice that by (\ref{hhh})
\begin{equation}
\label{kkk}
\|f\|_{L^{\infty}([1,+\infty)\times \T)} \leq C_3,
\end{equation}
for a suitable $C_3=C_3(A)$. By (\ref{floer}), this function satisfies the equation
\[
\overline{\partial} f = iH(u).
\]
By the coupling condition (\ref{eq:eta coupling}), we also have
\[
\im f(0,t) = 0 \qquad \forall t\in \T.
\]
We shall use the Calderon-Zygmund estimate
\begin{equation}
\label{czb}
\|\nabla \varphi\|_{L^p([0,+\infty)\times \T)} \leq c_p \|\overline{\partial} \varphi\|_{L^p([0,+\infty)\times \T)} 
\end{equation}
for every $\varphi\in C^{\infty}_c([0,+\infty)\times \T)$ such that $\im \varphi = 0$ on  $\{0\} \times \T$. Let $\chi:[0,+\infty) \rightarrow \R$ be a non-negative function such that $\chi=1$ on $[0,1]$, $\chi=0$ on $[2,+\infty)$, and $-2\leq \chi' \leq 0$. By applying (\ref{czb}) to the function $(s,t) \mapsto \chi(s)f(s,t)$ we obtain
\[
\begin{split}
\|\nabla f\|_{L^p([0,1]\times \T)} &= \| \nabla (\chi f)\|_{L^p([0,1]\times \T)} \leq \| \nabla (\chi f)\|_{L^p([0,+\infty) \times \T)} \leq c_p \|\overline{\partial} (\chi f)\|_{L^p([0,+\infty)\times \T)} \\ &= c_p \|\chi' f + \chi \overline{\partial} f\|_{L^p([0,+\infty)\times \T)}  \leq c_p ( 2 \|f\|_{L^p([1,2]\times \T)} + \|H(u)\|_{L^p([0,2]\times \T)}) \\ &\leq c_p ( 2\|f\|_{L^{\infty}([1,+\infty)\times \T)} + 2^{1/p}\|H\|_{L^{\infty}(M)}) \leq c_p ( 2 C_3 + 2^{1/p}\|H\|_{L^{\infty}(M)}), 
\end{split}
\]
where we have used also (\ref{kkk}). Together with the uniform bound on $\|f\|_{L^{\infty}(\{1\}\times \T)}$, which follows from (\ref{kkk}), the above estimate implies that 
$\|f\|_{W^{1,p}([0,1]\times \T)}$, is uniformly bounded, and since $p>2$, so is
$\|f\|_{L^{\infty}([0,1]\times \T)}$. This concludes the proof of the uniform $L^{\infty}$ bound on $\eta^+$ and $\zeta^+-\hat{\zeta}^+$.

There remains to show that $u^-$ and $u^+$ take values into $M_0$, or equivalently that the relatively open subsets of $(-\infty,0]\times \T$ and $[0,+\infty) \times \T$ which are defined by
\[
\Omega^- := (u^-)^{-1}(M \setminus M_0) \qquad \mbox{and} \qquad \Omega^+ := (u^+)^{-1}(M \setminus M_0)
\]
are empty. Let $r$ be the radial coordinate on 
\[
M\setminus M_0 = \iota(\Sigma_{\infty} \times (0,+\infty)),
\]
and consider the real valued functions
\[
r^{\pm} : \Omega^{\pm} \rightarrow \R, \qquad r^{\pm} := r\circ u^{\pm}.
\]
Since $u^-$ and $u^+$ are holomorphic on $\Omega^-$ and $\Omega^+$, standard computations involving the fact that $J_t(x,\tau)$ is independent of $t$ and $\tau$ and of contact type outside $M_0$ imply
\begin{equation}
\label{sub}
\Delta r^{\pm} = |\partial_s u^{\pm}|^2_{J}\geq 0,
\end{equation}
and 
\begin{equation}
\label{prime}
\partial_s r^{\pm} = - \lambda(\partial_t u^{\pm}).
\end{equation}
By the coupling condition (\ref{eq:loop coupling}), $r^-$ and $r^+$ coincide on the set
\[
\Omega^- \cap (\{0\}\times \T) = \Omega^+ \cap (\{0\}\times \T).
\]
Therefore, they define a continuous function $r$ on the union $\Omega^-\cup\Omega^+$. The fact that $v$ and $\tilde{u}$ have bounded energy implies that the connected components of $\Omega^-\cup\Omega^+$ are bounded. Using also the fact that $r=0$ on the boundary of $\Omega^-\cup\Omega^+$, we deduce that $r$ has an interior maximizer on each connected component of $\Omega^-\cup\Omega^+$. Such a point must belong to $\{0\}\times \T$, because the functions $r^-$ and $r^+$ cannot have interior maximizers, being subharmonic by (\ref{sub}). If $(0,t_0)$ is such a maximizer of $r$, then it is a maximizer of both $r^-$ and $r^+$ on suitable components of $\Omega^-$ and $\Omega^+$. Since $r^-$ and $r^+$ are subharmonic, the Hopf lemma implies the strict inequalities
\[
\partial_s r^-(0,t_0) >0 \qquad \mbox{and} \qquad \partial_s r^+(0,t_0)<0.
\]
Together with (\ref{prime}) and the coupling condition (\ref{eq:loop coupling}), these inequalities lead to a contradiction:
\[
0 < \partial_s r^-(0,t_0) = - \lambda(\partial_t u^-(0,t_0)) =  - \lambda(\partial_t u^+(0,t_0)) = \partial_s r^+(0,t_0) <0.
\]
This contradiction proves that $\Omega^-$ and $\Omega^+$ must be empty, concluding the proof.
\end{proof}

\section{Index computations and transversality}
\label{indcomptra}

In this section we deal with the Fredholm theory, the index computations, and the transversality issues that arise in this paper. 
Throughout this section, we equip $\R^{2n}$ with the standard complex structure
\begin{equation}
\label{J_n}
J_n : = 
 \begin{pmatrix}
0 & - I_n \\ I_n & 0
\end{pmatrix}
\end{equation}
and with the standard symplectic form $ \omega_n (u,v) :=  u\cdot J_n v$ (so that $J_n$ is $- \omega_n$-compatible).  

\subsection{Index computations}
\label{index_comp_general}

Here we use the notation from Sections \ref{grading-sec}  and recall
that given a connected component $ \Lambda \subset \crit \A_{ \tilde{H}}$ carries the index
\begin{equation}
\label{rs_index}
\mu( \Lambda) := \mu_{\mathrm{rs}}( \Lambda) - \frac{1}{2} \dim \, \Lambda.
\end{equation}
The following result is a standard computation (compare \cite[Proposition 4]{bm04}). Nevertheless, due to the presence of the $\R^*$-action we work in a slightly different functional setting than usual, and so we give the proof in full below. In fact, the precise functional setting we work in is irrelevant for the index computation, but it will be very important when we discuss transversality in Section \ref{transversality for cylinders} below.

\begin{thm}
\label{virdim_cylinders}
The virtual dimension of  $\mathcal{M}_{ \R \times \T}( \Lambda^- , \Lambda^+)$ is given by 
\[
\mathrm{virdim\,}\mathcal{M}_{ \R \times \T}( \Lambda^- , \Lambda^+) = \mu( \Lambda^-)  -  \mu( \Lambda^+)  +   \dim \, \Lambda^-  .
\]
\end{thm}

We will need to work with suitably weighted Sobolev spaces throughout. Let us recall the definition. We choose a family of positive smooth functions $\vartheta_{\delta} : \R \to (0,+\infty)$, $\delta>0$, which satisfy
\begin{equation}
\label{vartheta}
\vartheta_{\delta}(s) =  e^{\delta |s|} \qquad \mbox{for } | s | \ge 1,
\end{equation}
and if $I\subset \R$ is an unbounded interval we define the spaces
\[
\begin{split}
L^p_{\delta}(I) &:= \left\{ \xi \in L^p_{\mathrm{loc}}(I) \mid \xi \vartheta_{\delta} \in L^p(I) \right\}, \\
W^{1,p}_{\delta}(I) &:= \left\{ \xi \in W^{1,p}_{\mathrm{loc}}(I) \mid \xi \vartheta_{\delta} \in W^{1,p}(I) \right\}, 
\end{split}
\]
with Banach norms
\[
\|\xi \|_{p,\delta} := \|\xi \vartheta_{\delta} \|_p, \qquad
\| \xi \|_{1,p,\delta}: = \| \xi \vartheta_{\delta} \|_{1,p},
\]
respectively. 

Before getting started on the proof of the index formula from Theorem \ref{virdim_cylinders},  let us begin by explaining how to view $ \mathcal{M}( \Lambda^- , \Lambda^+) $ as the zero set of a Fredholm section $ \bar{\partial} : \mathcal{B} \to \mathcal{E}$ of a Banach bundle over a Banach manifold. Fix $p > 2$. Let us first consider a space $ \tilde{\mathcal{B}} (\Lambda^-, \Lambda^+) $ of maps $ \tilde{u}: \R \times \T \to \tilde{M}$ such that:

\begin{enumerate}
	\item $\tilde{u}$ is locally of class $W^{1,p}$,
	\item The following limits exist and are uniform in $t$: 
	\[
	\begin{split}
	\tilde{x}^-(t):= \lim_{s\rightarrow-\infty}\tilde{u}(s,t)\in\Lambda^-, \\
	\tilde{x}^+(t):= \lim_{s\rightarrow+\infty}\tilde{u}(s,t)\in\Lambda^+.
	\end{split}
	\]
	\item There exists $s_0 >0$ such that for $ | s | > s_0$ one can write 
	\begin{equation}
	\label{the_asymptotic_limits}
	\tilde{u}( s, t) = \exp_{ \tilde{x}^{ \pm}(t)} (\tilde{ v}^{ \pm}( s, t) ),
	\end{equation}
	for sections $ \tilde{v}^- \in W^{1,p}_{ \delta}( (- \infty , -s_0] \times \T, (\tilde{x}^-)^*(T \tilde{M}))$ and $\tilde{v}^+ \in W^{1,p}_{ \delta}( [s_0 , + \infty ) \times \T, (\tilde{x}^+)^*(T \tilde{M}))$ respectively. Here $ \exp$ denotes the exponential map with respect to some background Riemannian metric on $ \tilde{M}$.
\end{enumerate}

The space $ \tilde{\mathcal{B}}(\Lambda^-, \Lambda^+)$ admits the structure of a Banach manifold, and the tangent space to $ \tilde{\mathcal{B}}(\Lambda^-, \Lambda^+)$ at $ \tilde{u}$ can be identified as 
\begin{equation}
\label{tangent_space_to_B_tilde}
T_{ \tilde{u}} \tilde{ \mathcal{B}}(\Lambda^-, \Lambda^+) \cong \R^N \times W^{1,p}_{ \delta}(\R \times \T, \tilde{u}^*(T \tilde{M})),
\end{equation} 
where $N = \dim  \Lambda^- + \dim  \Lambda^+ $.   In fact we will be interested in the codimension one submanifold $\mathcal{B}(\Lambda^-, \Lambda^+) \subset \tilde{\mathcal{B}}(\Lambda^-, \Lambda^+)$, which consists of those elements $\tilde{u} \in \tilde{ \mathcal{B}}(\Lambda^-, \Lambda^+)$ with the property that  the asymptotic limits $\tilde{x}^\pm = (x^\pm ,\tau^\pm ,\sigma^\pm)$ of $\tilde{u}$ satisfy the additional requirement that
\begin{equation}
\label{the_space_B}
\sigma^- = \sigma^+
\end{equation}
Thus the tangent space to $\mathcal{B}(\Lambda^-, \Lambda^+)$ can be identified with
\begin{equation}
\label{tangent_space_to_B}
T_{ \tilde{u}}  \mathcal{B}(\Lambda^-, \Lambda^+)\cong \R^{N-1} \times W^{1,p}_{ \delta}(\R \times \T,  \tilde{u}^*(T \tilde{M})),
\end{equation} 
where $N$ is as in \eqref{tangent_space_to_B_tilde}. 

We now define a Banach bundle  $\tilde{\mathcal{E}} \to \tilde{\mathcal{B}}(\Lambda^-, \Lambda^+)$ by requiring that the  fibre over $\tilde{u}$ is given by
\[
 \tilde{\mathcal{E}}_{\tilde{u}} = L^p_{ \delta}(\R \times \T, \tilde{u}^*(T \tilde{M})).
 \] 
We then define a codimension one subbundle $\mathcal{E} \subset \tilde{\mathcal{E}}$ by setting 
\[
\mathcal{E}_{\tilde{u}} := \left\{ (w,\rho, \xi) \in L^p_{\delta}(\R \times \T , \tilde{u}^*(T\tilde{M})) \mid \int_{\R \times \T} \xi(s,t) \,ds\,dt =0 \right\}.
\] 
Fix $\mathrm{J} \in \mathcal{J}$ and let $\tilde{J}_t$ denote the corresponding loop of almost complex structures on $\tilde{M}$. Now define a section $ \bar{ \partial} = \bar{\partial}_{\tilde{H}, \mathrm{J}}$ of $ \tilde{\mathcal{E}} \to \tilde{\mathcal{B}}(\Lambda^- , \Lambda^+)$ by 
\begin{equation}
\label{eq:dbar}
\bar{ \partial}  \begin{pmatrix}
u \\ \eta \\ \zeta
\end{pmatrix} 	
=
\begin{pmatrix}
\partial_s u + J_t(u,\eta) \bigl( \partial_t u - \eta X_H(u)\bigr)  \\
\partial_s \eta + \partial_t \zeta - H(u)  \\
\partial_s \zeta - \partial_t \eta 
\end{pmatrix}.
\end{equation}

An argument similar to \cite[Appendix A]{bo09b} shows that provided $\delta$ is sufficiently small, the moduli space $\mathcal{M}(\Lambda^-, \Lambda^+)$ (which was defined as a set of \emph{smooth} maps) is included in the Banach manifold  $\tilde{\mathcal{B}}(\Lambda^-, \Lambda^+)$. From now on we assume that $\delta $ has this property. Standard elliptic regularity results then imply that $ \mathcal{M}(\Lambda^-, \Lambda^+)$ is exactly the set of zeros of the section $ \bar{\partial}$. In fact, we have the following simple lemma:
\begin{lem}
\label{B_in_B_tilde}
If $ \tilde{u} = (u ,\eta ,\zeta) \in \tilde{ \mathcal{B}}(\Lambda^-, \Lambda^+)$ belongs to the zero set of $ \bar{ \partial}$ then necessarily $ \tilde{u} \in \mathcal{B}(\Lambda^-, \Lambda^+)$. Moreover $\bar{\partial}( \mathcal{B}) (\Lambda^-, \Lambda^+)\subset \mathcal{E}$. 
\end{lem}

\begin{proof}
Assume that $\bar{\partial}(\tilde{u}) = 0$, and assume that  $\tilde{u} = (u ,\eta ,\zeta)$ has asymptotic limits 
\begin{equation}
\label{asy_lim}
\lim_{s \to \pm \infty}\tilde{u}(s, \cdot) =  (x^{\pm}, \tau^{\pm}, \sigma^{\pm}) .
\end{equation}
Since $s \mapsto \int_{\R}\zeta( s, t))\,dt$ is constant in $s$, a fact that has been used many times in this paper already, we immediately see that
\[
	\sigma^- =\sigma^+.
\]
Thus $\tilde{u} \in \mathcal{B}(\Lambda^-, \Lambda^+)$ as claimed. Now assume that $\tilde{u} = (u ,\eta ,\zeta)$ is an arbitrary element of $\mathcal{B}(\Lambda^-, \Lambda^+)$, with asymptotic limits as in \eqref{asy_lim}. Write $\bar{\partial}(u ,\eta , \zeta) = (w,  \rho , \xi)$. Then
\[
\begin{split}
	\int_{\R \times \T} \xi(s,t)\,ds\,dt & = \int_{ \R \times \T} (  \partial_s \zeta (s,t) - \partial_t \eta(s,t))\,ds\,dt \\
	& =  \int_{ \R \times \T}   \partial_s \zeta (s,t) \,ds\,dt \\
	& = \sigma^+ - \sigma^- =0,
\end{split}
\]	
since $\tilde{u } \in \mathcal{B}$.
\end{proof}

Lemma \ref{B_in_B_tilde} tells us that $\bar{ \partial}$ restricts to define a section $\bar{\partial} : \mathcal{B}(\Lambda^-, \Lambda^+) \to \mathcal{E}$ whose zero set is again all of $\mathcal{M}(\Lambda^-, \Lambda^+)$. It is this nonlinear operator whose index we will compute. More precisely, we will compute the Fredholm index of the vertical derivative $D_{\tilde{u}}$ of $\bar{\partial}$  at $\tilde{u}$:
\begin{equation}
\label{eq:the_operator_we_want_to_compute}
D_{ \tilde{u}} : T_{\tilde{u}} \mathcal{B}(\Lambda^-, \Lambda^+) \to \mathcal{E}_{\tilde{u}}.
\end{equation}
In fact, we will compute the index of the operator
\begin{equation}
\label{bar_D_u}
\bar{D}_{\tilde{u}} : = \iota \circ D_{\tilde{u}}|_{W^{1,p}_{\delta}(\R \times \T, \tilde{u}^*(T \tilde{M}))} : W^{1,p}_{\delta}(\R \times \T, \tilde{u}^*(T \tilde{M})) \to L^p_{ \delta}( \R \times \T, \tilde{u}^*(T \tilde{M})),
\end{equation}
where $\iota : \mathcal{E}_{\tilde{u}} \hookrightarrow \tilde{\mathcal{E}}_{\tilde{u}}$ is the inclusion. Since $\mathcal{E }_{\tilde{u}}$ has codimension one inside of $\tilde{\mathcal{E}}_{\tilde{u}}$, using \eqref{tangent_space_to_B} and additivity of the Fredholm index by composition, we see that that
\[
\mathrm{ind\,}D_{\tilde{u}} = \mathrm{ind\,}\bar{D}_{\tilde{u}}  -1  +1 = \mathrm{ind\,}D_{\tilde{u}}.
\]
In other words, the process of restricting to $\mathcal{B}(\Lambda^-, \Lambda^+) \subset \tilde{\mathcal{B}}(\Lambda^-, \Lambda^+)$ and then considering the image as lying inside $\mathcal{E}$ instead of $\tilde{\mathcal{E}}$ make no difference to the index (we just add and subtract one). Nevertheless, as far as transversality is concerned working in this functional setting is important, since we will see that $\bar{ \partial}$ is generically transverse as an operator $\mathcal{B}(\Lambda^-, \Lambda^+) \to \mathcal{E}$, but it is \emph{not} generically transverse as an operator $\tilde{\mathcal{B}}(\Lambda^-, \Lambda^+) \to \tilde{\mathcal{E}}$.

Let us now show that  $\bar{D}_ {\tilde{u}}$ is a Fredholm operator of index
\begin{equation}
\label{main_index_comp}
\mathrm{ind}\, \bar{D}_{ \tilde{u}} = 
 \mu( \Lambda^-) -  \mu( \Lambda^+)  - \dim  \Lambda^+ .
\end{equation}
This is by now a standard computation, compare \cite[Proposition 4]{bm04} or \cite{sch95}. Let us write $ \Gamma^{ \pm}$, $\bar{ \Gamma}^{\pm}$, and $ \Theta^{ \pm}$ for the corresponding symplectic paths associated to the orbits $\tilde{x}^{\pm}$ defined in \eqref{the symplectic matrices}. In the trivialisations specified above, the operator $\bar{D}_{\tilde{u}}$ is of the form 
\begin{equation}
\label{asymptotic_data}
\partial_s +J \partial_t + S,
\end{equation}
where $S: \R \times \T \to \mathrm{Mat}(\R^{2n+2})$ is a matrix valued path. Here $J $ is the $(2n+2) \times (2n+2)$ matrix defined by 
\[
J := J_{n-1} \times  J_1 \times -J_1
\]
under the splitting $\R^{2n+2} \cong \R^{2n-2} \oplus \R^2 \oplus \R^2$ (where $J_n$ and $J_1$ were defined in \eqref{J_n}), and the limits $S^\pm := \lim_{s \to \pm \infty} S(s, \cdot)$ exist and are given by
\[
S^{ \pm} := - J  \cdot \partial_t \Gamma_{j}^{ \pm} \cdot   \bigl(\Gamma_{j}^{ \pm}\bigr)^{-1}.
\]
Using the decomposition \eqref{the symplectic matrices}, we can write $S^{\pm}$ in block diagonal form:
\begin{equation}
\label{decomp_of_S}
S^{\pm}=\left(\begin{array}{cc}
B^{\pm} & 0\\
0 & C^{\pm}
\end{array}\right),
\end{equation}
where 
\[
B^{ \pm} := - J_{n-1}  \cdot \partial_t \bar{\Gamma}^{ \pm} \cdot   \bigl(\bar{\Gamma}^{ \pm}\bigr)^{-1}.
\]
\[
C^{ \pm} := - (J_1 \times -J_1) \cdot \partial_t \Theta^{ \pm} \cdot   \bigl( \Theta^{ \pm}\bigr)^{-1}.
\]
One checks directly that 
\[
C^{\pm} = \left( \begin{array}{cccc}
 0 & 0 & 0 & 0 \\
 0 & -\tau^{\pm}h''(1)  & -h'(1) & 0 \\
 0 & -h'(1) & 0 &0 \\
 0 & 0 & 0 & 0
 \end{array} \right).
 \] 

Going back to \eqref{asymptotic_data}, since the matrices $ S^{\pm}$ are not bijective, the operator $\bar{D}_{\tilde{u}}$ will \emph{not} be Fredholm when defined on the unweighted Sobolev spaces. However considering  $\bar{D}_{\tilde{u}}$ instead as being defined on the weighted Sobolev spaces is equivalent to considering the perturbed operator $ D^{ \mathrm{new}}_{ \tilde{u}}$ on the standard unweighted Sobolev spaces, where  $ D^{ \mathrm{new}}_{ \tilde{u}}$ looks like
\[
\partial_s + J \partial_t + (S^{\pm} \mp \delta I_{2n+2}).
\]
on the ends. Since (for $ \delta $ sufficiently small) the matrices $S^{\pm} \mp \delta I_{2n+2}
$ \emph{are} bijective, we can now apply the standard Fredholm theory from \cite{sch95} to see that $ D^{ \mathrm{new}}_{ \tilde{u}}$ (and hence also $ \bar{D}_{\tilde{u}}$) is Fredholm of index 
\[
\mu_{ \mathrm{rs}}( \Gamma_{ \delta}^- ) -  \mu_{ \mathrm{rs}}( \Gamma_{- \delta}^+ ),
\]
where 
\[
\Gamma_{\delta}^{\pm} : [0,1 ] \to \mathrm{Sp}(2n+2)
\]
denotes the path of symplectic matrices corresponding to $S^{\pm} \mp \delta  I_{2n+2}$, so that
\[
\partial_t{\Gamma}_{\delta}^{\pm} (t) = J \left(  S^{\pm} \mp \delta  I_{2n+2} \right) \Gamma_{  \delta}^{\pm}(t), \qquad \Gamma^{ \pm}_{  \delta}(0) = I_{2n+2}.
\]
To complete the proof we need only relate the Robbin-Salamon indices of the paths $ \Gamma^{ \pm }_{  \delta}$ with the unperturbed paths $ \Gamma^{ \pm}$. \\

Let us drop the $\pm$'s for clarity. So we consider a single path $ \Gamma : [ 0, 1] \to \mathrm{Sp}(2n+2)$ with corresponding matrix $S = - J \cdot \partial_t \Gamma \cdot\Gamma^{-1}$, and denote by $ \Gamma_{ \delta}$ the corresponding path with $S$ replaced by $ S - \delta I_{2n+2}$. We claim that for $ |\delta |$ sufficiently small,
\begin{equation}
\label{change_in_index}
\mu_{ \mathrm{rs}}(\Gamma_{\delta} ) = \mu_{ \mathrm{rs}}(\Gamma) -( \mathrm{sgn}\, \delta )  \frac{1}{2} \dim \Lambda.
\end{equation}
This is particularly transparent when the component $\Lambda$ has dimension exactly 2 (this corresponds to the case where the underlying Reeb orbit is transversally non-degenerate), and for simplicity here we consider only this special case. In this case the matrix $B$ in  \eqref{decomp_of_S} is bijective, and thus for $ \delta >0$ sufficiently small the only change in the Robbin-Salamon index will come from the $\Theta$ factor. 
We have already computed in \eqref{index_is_zero} that the Robbin-Salamon index of the unperturbed path (i.e. $ \delta = 0$) is given by
\[
\mu_{\mathrm{rs}}( \Theta) = - \frac{1}{2} \mathrm{sgn} \left(\begin{array}{cc}
\tau h''(1) & h'(1)\\
h'(1) & 0
\end{array}\right)=0.
\]
The new path $ \Theta_{ \delta}$ solves the equation
\[
\Theta_{\delta} : [0,1] \to \mathrm{Sp}(4), \qquad \partial_t \Theta_{ \delta} = ( J_1 \times - J_1) \bigl( C - \delta I_4 \bigr) \Theta_{\delta}, \qquad \Theta_{ \delta}(0) = I_4,
\]
and one has
\[
\mu_{ \mathrm{rs}}( \Theta_{ \delta}) = 
\mu_{\mathrm{rs}}( \Theta) - \frac{1}{2}\mathrm{sgn} \,\delta I_2 = - \mathrm{sgn}\, \delta.
\] 
Thus \eqref{change_in_index} follows, and hence so does \eqref{main_index_comp}. This completes the proof of Theorem \ref{virdim_cylinders}.

\subsection{Transversality on the extended phase space}
\label{transversality for cylinders}

In this section we show that transversality can be achieved for the spaces $\mathcal{C}( \Lambda^-, \Lambda^+)$ of cascades from Proposition \ref{cascades_prop} that are used to define the boundary operator. In the next section we show how similar arguments allow us to achieve transversality in Rabinowitz Floer theory.

Recall the definition of the space $\mathcal{J}$ of almost complex structures which are defined at the beginning of Section \ref{floeq} and of its subset $\mathcal{J}_{\mathrm{con}} \subset \mathcal{J}$, which is introduced below equation (\ref{ct}).

In \cite{fhs95}, Floer-Hofer-Salamon introduced the notion of a \emph{regular point} of a solution of the Floer equation. These are defined as follows. Suppose $\tilde{u}= (u ,\eta, \zeta) : \R \times \T \to \tilde{M}$ is a solution of \eqref{floer}, and suppose that 
\[
\lim_{s \to \pm \infty} \tilde{u}(s, \cdot) = \tilde{x}^{\pm} = (x^{\pm}, \tau^{\pm}, \sigma^{\pm}).
\] 
Then a point $(s,t)$ is regular if  
\[
\partial_{s}\tilde{u}(s,t)\ne0,
\qquad \tilde{u}(s,t)\ne \tilde{x}^{\pm}(t), \qquad 
\tilde{u}(s,t) \ne \tilde{u}(s' ,t), \quad \forall\, s' \in \R \setminus \{s \}.
\]
In \cite[Theorem 4.1]{fhs95} the authors proved that the set $\mathcal{R}(\tilde{u})$ of regular points is an open dense subset of $ \R \times \T$, and this proof carries over directly to our setting. We remark that in the entire paper \cite{fhs95}, there is a standing assumption that the symplectic manifold is compact and the asymptotes are non-degenerate. Nevertheless, these assumptions were not used in the proof of \cite[Theorem 4.1]{fhs95}, nor in any of the other results we quote from \cite{fhs95} below. Regular points play a crucial role in obtaining transversality for solutions of the Floer equation without needing to use $s$-dependent almost complex structures. See \cite[p269]{fhs95}. 

Since we insist that our complex structures $\tilde{J}_t$ are of the form \eqref{special form of acs}, where $\mathrm{J}$ belongs to $\mathcal{J}_{\mathrm{con}}$, the resulting deformation space used to obtain transversality is smaller than one would like. The difficulty lies not in the fact that $\mathrm{J}$ is independent of $t$ and $\tau$ and of contact type outside $M_0$, because Floer cylinders are contained in $M_0 \times T^* \R$, but rather in the rigidity of the form (\ref{special form of acs}), and in particular in the fact that no dependence on $\sigma$ is allowed. 
Because of this, we will need a slightly stronger result: namely that the set of regular points for the first two coordinates $(u , \eta)$ of a gradient flow line $\tilde{u} = (u ,\eta, \zeta)$ is open and dense in the set of points $(s,t)$ such that the $\partial_s u (s,t) \ne 0$. This should be compared to \cite[Proposition 4.3]{bo10}, where a similar enhancement of the results from \cite{fhs95} was required. Here are the precise details:

\begin{defn}
\label{regular_points}
Let $\tilde{u}= (u ,\eta, \zeta) : \R \times \T \to \tilde{M}$ denote a solution of \eqref{floer}, and suppose that 
\[
\lim_{s \to \pm \infty} \tilde{u}(s, \cdot) = \tilde{x}^{\pm} = (x^{\pm}, \tau^{\pm}, \sigma^{\pm}).
\] 
We denote by 
\[
\mathcal{R}(u, \eta):  =\left\{ (s,t)\in\mathbb{R}\times\mathbb{T}\Biggl|\begin{array}{c}
\partial_{s}(u, \eta)(s,t) \ne 0,\\
(u(s,t), \eta(s,t)) \ne (x^{\pm}(t),\tau^{\pm}) \\
(u(s,t), \eta(s,t)) \ne (u(s',t), \eta(s',t)), \quad \forall\, s' \in \R \setminus \{s \}.
\end{array}\right\} .
\]
\end{defn}

\begin{thm}
\label{regularpoints}
Assume that $\partial_s \tilde{u} $ is not identically zero. Then the set $\mathcal{R}(u, \eta)$ is an open dense subset of the non-empty open set  $\left\{ (s,t)  \mid \partial_s u(s,t) \ne 0 \right\}$.
\end{thm}

The proof of Theorem \ref{regularpoints} is given at the end of this section. For now, we will see how Theorem \ref{regularpoints} allows us to achieve transversality for the space of gradient flow lines with cascades. Standard arguments show that it suffices to show that the space of gradient flow lines with \emph{zero} cascades can generically be assumed to be a transverse problem. In order to state this result precisely, let us temporarily complicate our notation. Given 
$ \mathrm{J} \in   \mathcal{J}_{\mathrm{con}}$ and two distinct components $\Lambda^{\pm} \subset \crit \A_{\tilde{H}}$, let us write $\mathcal{M}_{\mathrm{J}} ( \Lambda^- , \Lambda^+)$ for the space of solutions of \eqref{floer} that satisfy $\lim_{s \to \pm \infty} \tilde{u}(s, \cdot) \in \Lambda^{\pm}$.

\begin{thm}
\label{transcylinders} 
There exists a residual subset $\mathcal{J}_{\mathrm{con}}^{\mathrm{reg}} \subset \mathcal{J}_{\mathrm{con}}$ such that if $\mathrm{J} \in \mathcal{J}_{\mathrm{con}}^{\mathrm{reg}}$ then for every pair $\Lambda^- \ne \Lambda^+$ of components of $\crit \A_{\tilde{H}}$, the spaces $\mathcal{M}_{\mathrm{J}}( \Lambda^- , \Lambda^+)$ are all smooth manifolds.
\end{thm}

\begin{proof}
Given $\ell \in \N$, we define $\mathcal{J}_{\mathrm{con}}^\ell$ in exactly the same way as $\mathcal{J}_{\mathrm{con}}$, only instead of requiring $\mathrm{J}= \{ J_t( \cdot , \tau)\}_{(t,\tau) \in \R \times \T} $ to be of class $C^{\infty}$, we require only that $\mathrm{J}$ is of class $C^{\ell}$, and we replace  condition \eqref{eq:acs_bounded} with
\begin{equation}
\label{eq:acs_bounded_l}
	\sup_{\tau \in \R } \| J_t(\cdot , \tau) \|_{C^{\ell}} < + \infty.
\end{equation}	
 The space $\tilde{\mathcal{J}}^\ell$ is defined similarly.
Unlike $\mathcal{J}_{\mathrm{con}}$, the space $\mathcal{J}_{\mathrm{con}}^\ell$ admits the structure of a Banach manifold for all $\ell \in \N$. We will use the notation from the previous section, apart from the fact that the section $\bar{\partial}$ from  \eqref{eq:dbar} will be now denoted by $\bar{ \partial}_{\mathrm{J}}$, since now $\mathrm{J}$ will be allowed to vary. Consider the extended section
 \[
\mathcal{F} : \mathcal{B}( \Lambda^-, \Lambda^+) \times \mathcal{J}_{\mathrm{con}}^\ell \to \mathcal{E}, \qquad \mathcal{F}(\tilde{u}, \mathrm{J} ) := \bar{ \partial}_{\mathrm{J}} (u) 
\]
Explicitly, 
\[
\bar{\partial}(\tilde{u}, \mathrm{J}):=
\left(\begin{array}{c}
	\partial_{s}u+J_t( u, \eta)\left( \partial_{t}u-\eta X_{H}(u) \right)\\
	\partial_{s}\eta+\partial_{t}\zeta-H(u)\\
	\partial_{s}\zeta-\partial_{t}\eta
\end{array}\right),
\]
where we have written $\tilde{u}=(u,\eta,\zeta)$.  As usual, the main step in the proof is to show that the
extended section $\mathcal{F}$ has the property that $\mathcal{F}^{-1}(0)$
is transverse to the zero section, and thus that the universal
moduli space $\mathcal{F}^{-1}(0)$ carries the structure of Banach
manifold. Equivalently, we must show that for every zero $(\tilde{u},\mathrm{J})$ of $\mathcal{F}$, the vertical derivative $D_{ \tilde{u}, \mathrm{J}} $ of $\mathcal{F}$, which is a map
\[
D_{\tilde{u},\mathrm{J}}:T_{\tilde{u}}\mathcal{B} (\Lambda^-, \Lambda^+) \times T_{\mathrm{J}}\mathcal{J}^\ell \to \mathcal{E}_{\tilde{u}},
\]
is a surjective operator. Recall that
\[
T_{\tilde{u}}\mathcal{B}(\Lambda^- , \Lambda^+) =  W^{1,p}_{\delta}( \R \times \T , \tilde{u}^*(T \tilde{M}))  \times \R^{\dim \, \Lambda^- + \dim \, \Lambda^+ - 1}.
\]
We identify the tangent space $T_{\mathrm{J}} \mathcal{J}_{\mathrm{con}} ^\ell$ with the space of $C^{\ell}$-maps $\tilde{Y} : \T \to  \Gamma(\mathrm{End}(T \tilde{M})) $ of the form
\[
\tilde{Y}_t( \tilde{x}) =  \begin{pmatrix}
Y_t( x, \tau ) & 0 \\
0 & 0
\end{pmatrix}, \qquad \forall\, (t, \tilde{x }= (x, \tau ,\sigma)) \in \T \times \tilde{M},
\]
where $Y : \T \times \R \to \Gamma(\mathrm{End}(TM))$ satisfies
\[
\begin{split}
 \omega(Y u ,v) + \omega(u,Yv) &=0, \quad \forall \, u,v \in T \tilde{M}, \\
 J_t(x ,\tau)Y_t(x, \tau)+ Y_t(\tau, x) J_t(\tau, x) &= 0, \quad \forall \, (t,\tau, x) \in \T \times \R \times M.
 \end{split}
\]
and finally such that $Y_t(x,\tau)= 0 $ for $x \in M \setminus M_0$, and 
\begin{equation}
\label{eq:acs_bounded_Y}
	\sup_{\tau \in \R } \| Y_t(\cdot , \tau) \|_{C^{\ell}} < + \infty.
\end{equation}	
It suffices to show that the restriction of $D_{\tilde{u}, \tilde{J}}$ to $W^{1,p}_{\delta}( \R \times \T , \tilde{u}^*(T \tilde{M})) \times T_{\mathrm{J}}\mathcal{J}_{\mathrm{con}}^\ell$ is surjective. Continuing to denote this restriction by $D_{\tilde{u},\mathrm{J}}$, it can be written as
\begin{equation}
\label{to_show_surjective}
D_{\tilde{u},\mathrm{J}}(\hat{u}, \hat{\eta},\hat{\zeta},\tilde{Y})=D_{\tilde{u}}(\hat{u}, \hat{\eta}, \hat{\zeta })+F_{\tilde{u},\mathrm{J}}(\tilde{Y}),
\end{equation}
where  $D _{\tilde{u}} = D_{\tilde{u}}|_{W^{1,p}_ \delta (\R \times \T , \tilde{u}^*(T \tilde{M}))}$ is the restriction of the operator from \eqref{eq:the_operator_we_want_to_compute}, which in this case sends a tangent vector $(\hat{u}, \hat{\eta}, \hat{\zeta}) $ to 
\[
\left(\begin{array}{c}
	\nabla_{s}\hat{u}+J_t(u, \eta)\left(\nabla_{t}\hat{u}-\eta\nabla_{\hat{u}}X_{H}(u)\right)+(\nabla_{\hat{u}}J + \hat{\eta} \partial_{ \eta} J  ) (\partial_{t}u-\eta X_{H}(u)) -\hat{\eta}J_t(u, \eta)X_{H}(u)\\
	\partial_{s}\hat{\eta}+\partial_{t}\hat{\zeta}-dH(u)[\hat{u}]\\
	\partial_{s}\hat{\zeta}-\partial_{t}\hat{\eta}
\end{array}\right),
\]
and the map $F_{\tilde{u},\tilde{J}}$ is defined by 
\[
F_{\tilde{u},\mathrm{J}}(\tilde{Y}):=\left(\begin{array}{c}
	Y_t(u, \eta)(\partial_{t}u-\eta X_{H}(u))\\
	0\\
	0
\end{array}\right) \qquad \text{for} \qquad \tilde{Y} =  \begin{pmatrix}
Y & 0 \\
0 & 0
\end{pmatrix}.
\]
We already know from Corollary \ref{virdim_cylinders} that the Floer operator $D_{\tilde{u}}$ is a Fredholm operator. Thus the range of $D_{\tilde{u}}$
is closed, and we need only prove that it is dense. Fix $q>1$ such
that $1/p+1/q=1$. The dual space $( \mathcal{E}_{\tilde{u}})^*$ is given by
\[
(w, \rho,\xi)\in (\mathcal{E}_{\tilde{u}})^* = L^q_{-\delta}(\R \times \T , \tilde{u}^*(T \tilde{M})) \Big/ \left\{ (0,0,c)  \mid c \in \R \right\}
\]
where $(0,0,c) \in L^q_{-\delta}(\R \times \T, \tilde{u}^*(T \tilde{M}))$ denotes the constant section\footnote{Note that the constants belong to $L^q_{-\delta}(\R \times \T, \tilde{u}^*(T \tilde{M}))$, thanks to the \emph{negative} weight.} and the pairing $L^p_{\delta} \times L^q_{-\delta} \to \R$ is given simply by integration (without weights). \\

To complete the proof we must show that if $(w, \rho ,\xi) \in L^q_{-\delta }(\R \times \T ,\tilde{u}^*(T \tilde{M}))$ has the property that for all $( \hat{u}, \hat{\eta}, \hat{ \zeta},\tilde{Y}) \in W^{1,p}_{\delta}( \R \times \T , \tilde{u}^*(T \tilde{M}))  \times T_{\mathrm{J}}\mathcal{J}_{\mathrm{con}}^{\ell}$
one has 
\[
\int_{ \R \times \T} \left \langle D_{ \tilde{u}, \mathrm{J}}( \hat{u}, \hat{\eta}, \hat{\zeta}, \tilde{Y}), (w, \rho , \xi) \right \rangle_{\tilde{J}_t(\tilde{u})} \,ds\,dt = 0,
\]
then $w$ and $\rho$ are identically zero and $\xi$ is constant. Firstly, taking $\tilde{Y}=0$, we see that $(w, \rho, \xi)$ is a weak solution to the equation
\begin{equation}
\label{eq:formal adjoint}
D_{\tilde{u}}^* (w, \rho, \xi)=0.
\end{equation}
Since the formal adjoint operator $D_{\tilde{u}}^* $
is again a first order elliptic operator with coefficients of class
$C^{\ell}$, we see that $(w,\rho, \xi)$ is in fact a classical solution of 
\eqref{eq:formal adjoint} and is itself of class $C^{\ell}$. 
Define 
\[
\Omega:=\left\{ (s,t)\in \mathcal{R}(u, \eta)\mid u(s,t)\notin \crit H\right\} .
\]
Since $\Lambda^- \ne \Lambda^+$, the set $\mathcal{R}(u,\eta)$ is open and dense in the open non-empty set of points $\left\{ (s,t) \mid \partial_s u (s,t) \ne 0 \right\}$ by Theorem \ref{regularpoints}. Moreover since $\Sigma = H^{-1}(0)$ is a regular energy level set and $\lim_{s \to \pm \infty}u(s, \cdot) \in \Sigma$, it follows that $\Omega$ is again an open non-empty set. Note also that since we assume $H$ is constant on $M \backslash M_0$, one has $M\backslash M_0 \subset \crit H$, and thus we are free to perturb $\mathrm{J}$ on $\Omega$ (recall outside of $M_0$ all elements of $\mathcal{J}_{\mathrm{con}}^{ \ell}$ are required to depend only on $x$ and be of contact type).  We will now prove that on $\Omega$ both $w$ and $\rho$ vanish and $\xi$ is constant. \\

First we check $w= 0$ on $\Omega$. Suppose for contradiction that there exists $(s_{0},t_{0})\in\Omega$ such that $w(s_{0},t_{0})\ne0$. Set $p := u(s_0,t_0)$ and $T := \eta(s_0,t_0)$, and choose a map $Y_p  : T_p M \to T_p M$ such that $J_{t_0}(T,p) Y_p + Y_p J_{t_0}(T,p) = 0$, and such that $\omega_p(Y_p u ,v) + \omega_p(u ,Y_pv) = 0$ for all $u,v \in T_p M$, and finally such that 
\[
\omega_p(Y_p  \partial_s u(s_0,t_0 ), w(s_0,t_0)) > 0.
\]
See for instance \cite[p1346]{sz92} for an explicit construction of such a $Y_p$. Now choose an element $\tilde{Y} \in T_J \mathcal{J}_{\mathrm{con}}^{ \ell}$ such that 
\[
\tilde{Y}(t_0, \tilde{u}(s_0,t_0)) =  \begin{pmatrix}
Y_p & 0 \\
0 & 0
\end{pmatrix}.
\]
 Since $(s_0,t_0)$ belongs to $\mathcal{R}(u, \eta)$, one can choose a smooth function $\beta : \T \times M \times \R \to [0,1]$ such that if $\tilde{Y}_1 \in T_{\mathrm{J}} \mathcal{J}_{\mathrm{con}}^{\ell}$ is defined by
 \[
 \tilde{Y}_{1,t}( \tilde{x}) := \beta(t,x, \tau) \tilde{Y}_t(\tilde{x} ), \qquad \forall (t, \tilde{x} = (x,\tau ,\sigma)) \in \T \times \tilde{M},
 \]
 then
\begin{equation}
\label{this_is_a_contradiction}
\int_{ \R \times \T} \left \langle F_{ \tilde{u}, \mathrm{J}}( \tilde{Y}_1), (w, \rho , \xi) \right \rangle_{\tilde{J}_t (\tilde{u})} \,ds\,dt  > 0
\end{equation}
(see \cite[Remark 4.4]{fhs95}). This is a contradiction, and hence $w$ vanishes on $\Omega$ as desired. Now suppose $\hat{u}\in W_{\delta}^{1,p}(\mathbb{R}\times\mathbb{T},u^{*}(T M))$
is supported in $\Omega$. Then from what we have already shown, 
\[
\int_{ \R \times \T} \left \langle D_{\tilde{u}}( \hat{u}, 0,0), (w, \rho , \xi) \right \rangle_{\tilde{J} } \,ds\,dt =-\int_{\Omega}\rho(s,t)\cdot dH(u(s,t))[\hat{u}(s,t)]\,ds\,dt.
\]
Since $u^{-1}(\crit H)\cap\Omega=\emptyset$, it follows easily
from the previous equation that if there exists $(s_{1},t_{1})\in\Omega$
such that $\rho(s_{1},t_{1})\ne0$ then for a suitable choice
of $\hat{u}$ one can make 
\[
\int_{ \R \times \T} \left \langle D_{\tilde{u}}( \hat{u}, 0,0), (w, \rho , \xi) \right \rangle_{\tilde{J}_t(\tilde{u})} \,ds\,dt>0.
\]
Thus $\rho$ also vanishes on $\Omega$. It then follows from \eqref{eq:formal adjoint} that $\xi$ is constant on $\Omega$.\\

We now complete the proof that the operator \eqref{to_show_surjective} is surjective. Suppose $\xi(s,t) = c$ for all $(s,t) \in \Omega$.  Then $(w, \rho, \xi - c)$ is another solution of \eqref{eq:formal adjoint} which vanishes on $\Omega$.  Since $(w, \rho ,\xi -c)$ has the unique continuation property \cite[Proposition 3.1]{fhs95}, it follows that $w \equiv 0$, $\rho \equiv 0$ and $\xi \equiv c$ as required. The remaining details
of the proof are absolutely standard, and use the usual Sard-Smale theorem 
 and an argument due to Taubes in order to pass from residuality in $C^{\ell}$ to residuality in $C^{\infty}$. We refer the reader to \cite[proof of Theorem 3.1.5.(ii)]{ms04}
for a detailed exposition.
\end{proof}

It remains to prove Theorem \ref{regularpoints}. Given a flow line $\tilde{u}= (u, \eta ,\zeta)$, in the forthcoming arguments it is convenient to abbreviate the first two components $(u, \eta)$ by $v$. Thus $v: \R \times \T \to M \times \R$. We will also use the notation $C(v) := \left\{ (s,t) \in \R \times \T \mid  \partial_s v(s,t) = 0\right\}$.  

\begin{rem}
Standard arguments \cite[Lemma 4.1]{fhs95} tell us that the set of points $\left\{ (s,t) \mid \partial_s \tilde{u}(s,t) = 0 \right\}$ is a discrete subset of $\R \times \T$.  In general this will not be true for the set $C(v)$. Nevertheless it is not difficult to show that the set $C(v)$ has non-empty interior in $\R \times \T$ (compare \cite[Proposition 3.3]{bo10}), and in fact one can even prove the stronger statement that the set $\left\{ (s,t) \mid \partial_s u(s,t) = 0 \right\}$ has non-empty interior in the open set $\left\{ (s,t) \in \R \times \T \mid u(s,t) \notin \crit H \right\}$. However we will not need this result in the proof of Theorem \ref{regularpoints}. 
\end{rem}	 

It will also be convenient to use the following notation. Given $s \in \R$ and $h >0 $, let $I_h(s):=[s-h,s+h]$. Similarly given $t \in \T$ we denote by $I_h(t):=[t-h,t+h]$, which for $0< h < 1/2$ defines a proper arc in $\T = \R / \Z$. Finally we write $Y_h(s,t):=I_h(s) \times I_h(t)$. We will need the following lemma in the proof of Theorem \ref{regularpoints}, which is a minor variation of \cite[Lemma 4.2]{fhs95} (compare also \cite[Lemma 4.5]{bo10}).
\begin{lem}
\label{v_0 = v_1}
Suppose $\tilde{u}_0$ and $\tilde{u}_1$ are two solutions of the Floer equation \eqref{floer}, defined on  $I_{h_0}(s_0) \times \T$ for some $h_0 > 0$. Write $\tilde{u}_j = (u_j ,\eta_j ,\zeta_j)$, and abbreviate $v_j := (u_j, \eta_j)$. Suppose there exists $t_0 \in \T$ such that
\begin{equation}
\label{v are equal}
v_0(s_0,t_0) = v_1(s_0,t_0),
\end{equation}
and
\begin{equation}
\label{assumptions on partial s}
\partial_s u_0 (s_0, t_0 ) \ne 0, \qquad \partial_s v_1(s_0,t_0)  \ne 0.
\end{equation}
Assume in addition that for all $0 < h' < h_0$ there exists $0 < h < h'$ such that for all $(s,t) \in Y_h(s_0,t_0)$ there exists $s' \in I_{h'}(s_0)$ such that
\begin{equation}
\label{agree}
v_0(s,t) = v_1(s' ,t).
\end{equation}
Then in fact $v_0 = v_1$, and hence $\tilde{u}_0$ and $\tilde{u}_1$ agree up to a constant shift in the $\R^*$ direction. 
\end{lem}

\begin{proof}
It follows from \eqref{assumptions on partial s} that there exists $h' >0 $  such that  for  all $0 <h < h'$ sufficiently small, both of the maps 
\[
I_h(s_0) \ni s \mapsto
\begin{cases}
 u_0(s,t), \\ 
 v_1(s,t),
 \end{cases}
  \qquad \forall \, t \in I_h(t_0),
\]
are embeddings.  Moreover by assumption there exists $h < h'$ such that for each $t \in I_h(t_0)$ there is a well defined smooth map 
\[
\psi_t : (v_1(\cdot, t))^{-1} \circ v_0( \cdot ,t) : I_h(s_0) \to I_{h'}(s_0),
\]
and for $h'' < h'$ small enough it holds in addition that for all $h < h''$ one has $s_0 \in \mathrm{im}(\psi_t)$ for each $t \in I_h(t_0)$. The implicit function theorem allows us to invert $\psi_t$, obtaining a map 
\[
\theta_t := (\psi_t)^{-1}: I_{h''}(s_0) \to I_h(s_0).
\]
These maps $\theta_t$ piece together to define a map
\[
\theta : Y_{h''}(s_0,t_0) \to Y_h(s_0,t_0)
\]
of the form $\theta(s,t) = (\phi (s,t), t)$. Thus for all $(s,t) \in  Y_{h''}(s_0,t_0) $, one has
\[
v_1(s,t)= v_0(\phi(s,t), t).
\]
We now apply this to the equation $\partial_s u_1 + J_t(u_1, \eta_1)( \partial_t u_1 - \eta_1 X_H(u_1)) =0$ to obtain
\[
\begin{split}
0 & =\partial_ s u_1 + J_t(u_1, \eta_1)( \partial_t u_1 - \eta_1 X_H(u_1)) \\
& =  \partial_s u_0(\theta) \partial_s \phi + J_t(u_0 (\theta) , \eta_0(\theta)) \partial_s u_0( \theta)  + \underset{=- J_t(u_0 (\theta) , \eta_0(\theta)) \partial_s u_0(\theta)}{\underbrace{J_t(u_0 (\theta) , \eta_0(\theta))( \partial_t u_0( \theta) - \eta_0(\theta) X_H(u_0(\theta)))}}   \\
& = \partial_s u_0(\theta) (\partial_s \phi - 1) + J_t(u_0 (\theta) , \eta_0(\theta)) \partial_s u_0(\theta) \partial_t \phi.
\end{split}
\]
By assumption $\partial_s u_0(\theta) \ne 0$, and hence the vectors $ \partial_s u_0(\theta)$ and $J_t(u_0 (\theta) , \eta_0(\theta))\partial_s u_0(\theta)$ are linearly independent. This implies $\partial_s \phi = 1$ and $\partial_t \phi = 0$. Since $\phi(s_0,t_0) = s_0$, we see that $\phi(s,t) = (s,t) $ for all $(s,t) \in Y_{h''}(s_0,t_0)$. Thus $v_0 = v_1$ on $Y_{h''}(s_0,t_0)$. It now follows directly from the Floer equation \eqref{floer} that there exists a constant $c \in \R^*$ such that $\zeta_1(s,t) = \zeta_0(s,t)+c$ on  $Y_{h''}(s_0,t_0)$.  Finally, define $\tilde{u}_2 : I_h(s_0) \times \T \to \tilde{M}$ by $\tilde{u}_2(s,t) = (u_1(s,t), \eta_1(s,t), \zeta_1(s,t) - c)$. Then $\tilde{u}_0$ and $\tilde{u}_2$ are both solutions of \eqref{floer} that agree on $Y_{h''}(s_0,t_0)$. By unique continuation \cite[Proposition 3.1]{fhs95} it follows that $\tilde{u}_0 = \tilde{u}_2$ on all of $I_h(s_0) \times \T$. This completes the proof.
\end{proof} 

We now prove Theorem \ref{regularpoints}. The proof is very similar to the proof of Theorem 4.3 in \cite{fhs95}, however there are some minor but important differences, and hence for the convenience of the reader we give a complete proof. 

\begin{proof}[Proof of Theorem \ref{regularpoints}]
Fix a solution $\tilde{u}= (u ,\eta, \zeta):\R \times \T \to \tilde{M}$ of \eqref{floer}, and let $\tilde{x}^{\pm}=(x^\pm, \tau^\pm, \sigma^\pm)$ denote the asymptotes of $\tilde{u}$. Abbreviate $v = (u,\eta)$. First note that since by assumption $\partial_s \tilde{u}$ is not identically zero, the set $\left\{ (s,t) \in \R \times \T \mid \partial_s u (s,t) \ne 0\right\}$ is a non-empty open set.

Let us first check that $\mathcal{R}(v)$ is open. It is clear we need only exclude the following situation: the existence of a sequence $(s_k,t_k) \in \R \times \T$ converging to some $(s_0,t_0) \in \mathcal{R}(v)$, together with another sequence $s_k ' \ne s_k$ such that $v(s_k,t_k)= v(s_k' ,t_k)$ for each $k$. Since $\partial_s v(s_0,t_0) \ne 0$, there exists $h> 0 $ such that for all $t \in I_h(t_0)$, the map $s \mapsto v(s,t)$ is an embedding for $s \in I_h(s_0)$. Thus for all $k$ sufficiently large the map $I_h(s_k) \ni s \mapsto v(s,t_k)$ is also an embedding. If (up to a subsequence) we had $s_k' \to \pm \infty$ one would obtain $v(s_0,t_0) = (x^\pm(t_0),\tau^\pm)$, contradicting the fact that $(s_0,t_0) \in \mathcal{R}(v) $, and similarly if $s_k ' \to s_0$ then one would have $s_k' \in I_h(s_k)$ for all $k$ sufficiently large, which contradicts the fact that $I_h(s_k) \ni s \mapsto v(s,t_k)$ is an embedding. Thus up to passing to a subsequence, we may assume that $s_k' \to s_0' \in \R$ for some $s_0' \ne s_0$. But then $v(s_0,t_0) = v(s_0' ,t_0)$, which contradicts the assertion that $(s_0,t_0) \in \mathcal{R}(v)$. Thus $\mathcal{R}(v)$ is open, as desired.

We now prove that $\mathcal{R}(v)$ is dense in $\left\{ (s,t) \in \R \times \T \mid \partial_s u (s,t) \ne 0\right\}$. Fix $(s_0,t_0) \in \R \times \T$ such that $\partial_s u(s_0,t_0) \ne 0$, and fix $h > 0 $ such that $\partial_s u$ is never zero on $Y_h(s_0,t_0)$, and such that for all $t \in I_h(t_0)$, the map $I_h(s_0) \ni s \mapsto u(s,t)$ is an embedding.  The proof will be accomplished via a sequence of claims.

\textbf{Claim 1:} Every such point $(s_0,t_0)$ can be approximated by a sequence $(s_k,t_k)$ such that $v(s_k,t_k) \ne (x^\pm(t_k), \tau^\pm)$.

\emph{Proof of Claim 1:} Take $s_k = s_0 + 1/k$ and $t_k = t_0$. If $v(s_k,t_k) = (x^\pm(t_k),\tau^\pm)$ for all $k$ then necessarily we have $\partial_s v(s_0,t_0 ) = 0 $, which contradicts the assertion that $\partial_s u (s_0 ,t_0 ) \ne 0$.

By the claim we may therefore assume that without loss of generality for all $(s,t) \in Y_h(s_0,t_0)$, one has $v(s,t) \ne (x^\pm(t), \tau^\pm)$. Assume now for contradiction there exists $0< \varepsilon <h $ such that $Y_{\varepsilon}(s_0,t_0) \cap \mathcal{R}(v) = \emptyset$. Thus for all $(s,t) \in Y_ \varepsilon(s_0,t_0)$ there exists $s' \ne s$ such that $v(s',t) = v(s,t)$. Arguing as in the proof above that $\mathcal{R}(v) $ is open, we see that there exists a uniform $T > 0 $ such that any such $s'$  satisfies $|s' | \le T$. 

\textbf{Claim 2:} There exists a sequence $s_k \to s_0$ such that if $s' \ne s_k$ satisfies $v(s',t_0) = v(s_k,t_0)$ then $(s',t_0) \notin C(v)$.

\emph{Proof of Claim 2 (Following \cite[p1206]{bo10}):} We may choose a chart $U \subset M \times \R$ around $v(Y_ \varepsilon (s_0,t_0))$ of the form $I_ \varepsilon(s_0) \times \R^{2n}$, such that (expressed in this chart), the map $v : V := v^{-1}(U) \to I_ \varepsilon(s_0) \times \R^{2n}$ is given by $(s,t) \mapsto (f(s,t), g(s,t))$, where the restriction of $f$ to $Y_ \varepsilon (s_0,t_0) \subset V$ is simply given by $f(s,t) = s$. Define $m(s) := f(s,t_0)$. Thus if $(s,t_0) \in V$ belongs to $C(v)$ then $m'(s) = 0$. In particular this is true for $s \in I_ \varepsilon (s_0)$. Thus
\[
 \left\{ s \in I_ \varepsilon (s_0) \mid  \exists \, (s',t_0) \in V \cap C(v) \text{ with } v(s',t_0 ) = v(s,t_0) \right\} \subset \mathrm{Spec}(m).
 \] 
 The set $\mathrm{Spec}(m)$ of critical values of $m$ is nowhere dense by Sard's Theorem. This proves the claim.

Replacing $s_0$ with $s_k$ for $k$ sufficiently large, we may therefore assume that for each $s' \in \R$ such that $v(s',t_0) = v(s_0,t_0)$, one has $(s', t_0) \notin C(v)$.

 \textbf{Claim 3:} After shrinking $\varepsilon$ if necessary, we may assume that 
 \[
 \forall\,(s,t) \in Y_ \varepsilon(s_0,t_0), \ \forall \, s' \in \R, \quad v(s,t ) = v(s',t) \ \then \ (s',t) \notin C(v).
 \]
 \emph{Proof of Claim 3:} If this failed then we could find a sequence $(s_k,t_k)$ converging to $(s_0,t_0)$ and a sequence $(s_k ')$ such that $s_k'$ is bounded away from $s_0$, and such that $(s_k ', t_k) \in C(v)$ and $v(s_k,t_k) = v(s_k' , t_k)$ for all $k \in \N$. As above, such a sequence $(s_k') $ is necessarily bounded, and hence up to passing to a subsequence we may assume $s_k' \to s_0 ' \ne s_0$. Then $(s_0', t_0) \in C(v)$ and $v(s_0',t_0) = v(s_0, t_0)$, which contradicts the line immediately before the claim.

 For any $(s,t) \in Y_ \varepsilon(s_0,t_0)$, there are certainly at most finitely many points $s' \in \R$ such that $v(s,t) = v(s',t)$. Indeed, if not then there would exists an accumulation point $s''$ of such $s'$ (since as we have already noted, all such $s'$ are uniformly bounded). Then we would have $v(s,t) = v(s'',t)$ with $(s'',t) \in C(v)$, which contradicts the last claim. Therefore let us let $s_1, \dots , s_p$ denote the finitely many points such that $v(s_j,t_0) = v(s_0,t_0)$ for $j = 1, \dots p$. 

 \textbf{Claim 4:} For any $\delta >0 $ there exists $\gamma > 0$ such that if $(s,t) \in Y_{2 \delta'}(s_0,t_0)$ then there exists $1 \le j \le p$ and $(s',t) \in Y_ \delta(s_j, t_0)$ such that $v(s,t) = v(s',t)$. 

 \emph{Proof of Claim 4:} If the claim failed we could find $\delta > 0$ and a sequence $(s_k,t_k) \to (s_0,t_0)$ such that $v(s_k,t_k) \ne v(s',t_k)$ for all $s' \in I_ \delta (s_j)$, for all $1 \le j \le p$. Since by assumption for all $k$ there exists \emph{some} $s_k' \in [-T,T]$ such that $v(s_k,t_k) = v(s_k', t_k)$, by taking a limit we find a point $s'_0$ distinct from $s_1, \dots , s_p$ such that $v(s'_0,t_0) = v(s_0,t_0)$. This is a contradiction.

 This implies that for $\gamma >0 $ sufficiently small, 
 \[
 \overline{Y}_ \gamma (s_0,t_0) = \bigcup_{j=1}^p V_j, 
 \]
 where
 \[
 V_j := \left\{ (s,t) \in \overline{Y}_ \delta (s_0,t_0) \mid \, \exists \, (s',t) \in \overline{Y}_ \delta (s_j,t_0) \text{ with } v(s',t) = v(s,t)\right\}.
 \]
 By the Baire Category theorem, at least one of the sets $V_j$ has non-empty interior. Let us suppose $V_1^\circ \ne \emptyset$. Fix $(\hat{s}, \hat{t}) \in V_1^\circ$, and let $(\hat{s}_1, \hat{t})$ denote the unique point in $Y_ \delta(s_1, t_0)$ such that $v(\hat{s},\hat{t}) = v(\hat{s}_1, \hat{t )}$. Choose $0< \delta' < \delta$ such that $Y_{ \delta'}(\hat{s}_1, \hat{t}) \subset Y_{\delta}(s_1, t_0)$ and choose $0 < \gamma' < \gamma$ such that $Y_{\gamma'}(\hat{s},\hat{t}) \subset V_1$. Then by assumption, for all $0 < h' < \delta'$, there exists $0 < h <\gamma'$ such that for all $(s,t) \in Y_h(\hat{s},\hat{t})$, there exists $(s',t) \in Y_{h'}(\hat{s}_1, \hat{t})$ such that $v(s,t) = v(s',t)$. Thus if we define $v_1(s,t) := v( s +\hat{s}_1 - \hat{s},t)$ we may apply Lemma \ref{v_0 = v_1} to obtain $v = v_1$ on all of $\R \times \T$. But then for any $(s,t) \in \R \times \T$, we have  
 \[
 v(s,t) = \lim_{s \to \pm \infty}v(s + k( \hat{s}_1 - \hat{s}),t)= (x^\pm(t), \tau^\pm),
 \]
 which contradicts the fact that $\partial_s v$ is not identically zero.
\end{proof}

\begin{rem}
\label{still_works_on_half_cylinder}
Note that this proof continues to work if our map $\tilde{u}$ is only defined on a half-cylinder $\R^\pm \times \T$. Thus for such a map $\tilde{u}$, if $\partial_s \tilde{u} $ is not identically zero then the set $\mathcal{R}(u, \eta)$ is an open dense subset of the non-empty open set  $\left\{ (s,t) \in \R^\pm \times \T  \mid \partial_s u(s,t) \ne 0 \right\}$.
\end{rem}

\subsection{Transversality in Rabinowitz Floer homology}
\label{rfh_trans}

In this section we show how transversality can be achieved in Rabinowitz Floer homology using almost complex structures $\mathrm{J} \in \mathcal{J}_{\mathrm{con}}$. The starting point is again the following analogue of the definition of a regular point. 
\begin{defn}
\label{regular_points_RF}
Let $v = (u, \eta)$ denote a solution of \eqref{eq:RFH equations}, and suppose that 
\[
\lim_{s \to \pm \infty} v(s, \cdot) =   \hat{x}^\pm( \cdot) = (x^{\pm}(\cdot), \tau^{\pm}).
\] 
We denote by 
\[
\mathcal{R}(v):  =\left\{ (s,t)\in\mathbb{R}\times\mathbb{T}\Biggl|\begin{array}{c}
\partial_{s}u(s,t) \ne 0, \quad \partial_s \eta(s) \ne 0 \\
v(s,t) \ne \hat{x}^{\pm}(t) \\
v(s,t) \ne v(s',t) \quad \forall\, s' \in \R \setminus \{s \}.
\end{array}\right\} .
\]
\end{defn}
Then the following result holds. The proof is almost word-for-word identical as the proof of Theorem \ref{regularpoints}, and so we will not give it here. It is also very similar to the proof of \cite[Proposition 4.3]{bo10}.
\begin{thm}
\label{regularpoints_RF}
Assume that $\partial_s u $ is not identically zero. Then the set $\mathcal{R}(v)$ is an open dense subset of the non-empty open set  $\left\{ (s,t)  \mid \partial_s u(s,t) \ne 0 \right\}$.
\end{thm}
As in the previous section, we will now use Theorem \ref{regularpoints_RF} to deduce the following result. Given $\mathrm{J} \in \mathcal{J}_{\mathrm{con}}$ and two components $K^- \ne K^+$ of $\crit \mathcal{A}_H$, let us denote by $\mathcal{M}_{RF, \mathrm{J}}(K^-,K^+)$ the space of solutions $v = (u,\eta)$ of \eqref{eq:RFH equations} whose asymptotic limits belong to $K^{\pm}$.
\begin{thm}
\label{transcylindersRF}
There exists a comeagre subset $\mathcal{J}_{\mathrm{con}}^{\mathrm{RF,reg}} \subset \mathcal{J}_{\mathrm{con}}$ such that if $\mathrm{J} \in \mathcal{J}_{\mathrm{con}}^{\mathrm{RF,reg}}$ then for every pair $K^- \ne K^+$ of components of $\crit \mathcal{A}_H$, the spaces $\mathcal{M}_{\mathrm{J}}(K^-, K^+)$ are all smooth manifolds.
\end{thm}
The functional setting we work in is similar to the previous section. Let $\mathcal{B}(K^-,K^+)$ denote the space of maps $v = (u ,\eta)$, where $u : \R \times \T \to M$ and $\eta : \R \to \R$, and such that $v$ is locally of class $W^{1,p}$ (for some $p >2$), $\lim_{s \to \pm \infty} v(s, \cdot) \in K^{\pm}$, and finally such that $v$  belongs to $W^{1,p}_{\delta}$ in local charts near the asymptotes (i.e. the analogue of \eqref{the_asymptotic_limits} holds). Let $\mathcal{E} \to \mathcal{B}(K^-,K^+)$ denote the Banach bundle whose fibre over $v$ is given by 
\[
 	\mathcal{E}_v = L^p_{\delta}(\R \times \T, u^*(TM)) \oplus L^p_{\delta}(\R , \R),
\] 
and consider the map
\[
	\mathcal{F}: \mathcal{B}(K^-,K^+ ) \times \mathcal{J}^{\ell}_{\mathrm{con}} \to \mathcal{E},
\]
given by
\[
	\mathcal{F}(u,\eta, \mathrm{J}) = \left(\begin{array}{c}
	\partial_{s}u+J_t( u, \eta)\left( \partial_{t}u-\eta X_{H}(u) \right)\\
	\partial_{s}\eta - \int_{\T} H(u)\,dt
\end{array}\right).
\]
As before, we need to show that for every zero $(v,\mathrm{J})$ of $\mathcal{F}$, the vertical derivative $D_{v, \mathrm{J}} $ of $\mathcal{F}$, which is a map
\[
D_{v,\mathrm{J}}:T_v\mathcal{B} (K^-, K^+) \times T_{\mathrm{J}}\mathcal{J}^\ell \to \mathcal{E}_{\tilde{u}},
\]
is a surjective operator. We can identify
\[
T_{v}\mathcal{B}(K^-, K^+) =  W^{1,p}_{\delta}( \R \times \T , u^*(T M) ) \times W^{1,p}_{\delta}(\R, \R) \times  \R^{\dim K^- + \dim K^+},
\]
and this time we identify $T_{\mathrm{J}}\mathcal{J}^{\ell}_{\mathrm{con}}$ with the space of $C^{\ell}$ maps  $Y : \T \times \R \to \Gamma(\mathrm{End}(TM))$ satisfying
\[
\begin{split}
 \omega(Y u ,v) + \omega(u,Yv) &=0, \quad \forall \, u,v \in T \tilde{M}, \\
 J_t(x ,\tau)Y_t(x, \tau)+ Y_t(\tau, x) J_t(\tau, x) &= 0, \quad \forall \, (t,\tau, x) \in \T \times \R \times M.
 \end{split}
\]
and finally such that $Y_t(x,\tau)= 0 $ for $x \in M \setminus M_0$, and 
\[
	\sup_{\tau \in \R } \| Y_t(\cdot , \tau) \|_{C^{\ell}} < + \infty.
\]
As before, it suffices to show that the restriction of $D_{v, \tilde{J}}$ to $W^{1,p}_{\delta}( \R \times \T , u^*(T M) ) \times W^{1,p}_{\delta}(\R, \R) \times T_{\mathrm{J}}\mathcal{J}_{\mathrm{con}}^\ell$ is surjective. Continuing to denote this restriction by $D_{v,\mathrm{J}}$, it can be written as
\[
D_{v,\mathrm{J}}(\hat{u}, \hat{\eta}, Y)=D_{v}(\hat{u}, \hat{\eta})+F_{v,\mathrm{J}}(\tilde{Y}),
\]
where  $D_v$ is the operator which case sends a tangent vector $(\hat{u}, \hat{\eta}) $ to 
\[
\left(\begin{array}{c}
	\nabla_{s}\hat{u}+J_t(  u,\eta)\left(\nabla_{t}\hat{u}-\eta\nabla_{\hat{u}}X_{H}(u)\right)+(\nabla_{\hat{u}}J + \hat{\eta} \partial_{ \eta} J  ) (\partial_{t}u-\eta X_{H}(u)) -\hat{\eta}J_t( u, \eta)X_{H}(u)\\
	\partial_{s}\hat{\eta} - \int_{\T} dH(u)[\hat{u}]\,dt
\end{array}\right),
\]
and the map $F_{v,\mathrm{J}}$ is defined by 
\[
F_{\tilde{u},\mathrm{J}}(Y):=\left(\begin{array}{c}
	Y_t(u, \eta)(\partial_{t}u-\eta X_{H}(u))\\
	0
\end{array}\right).
\]
The operator $D_v$ is Fredholm, and hence its range is closed. As above, it therefore suffice to show that the annhilator of its image is zero. Fix $q>1$ such
that $1/p+1/q=1$. The dual space $( \mathcal{E}_v)^*$ is given by
\[
(\mathcal{E}_v)^* = L^q_{-\delta}(\R \times \T , u^*(TM)) \times L^q_{-\delta}(\R, \R).
\]
We must show that if $(w, \rho ) \in (\mathcal{E}_v)^* $ has the property that for all $( \hat{u}, \hat{\eta},Y) \in W^{1,p}_{\delta}( \R \times \T , u^*(TM))  \times W^{1,p}_{\delta}(\R, \R) \times T_{\mathrm{J}}\mathcal{J}_{\mathrm{con}}^{\ell}$
one has 
\[
\int_{ \R \times \T} \left \langle D_{v, \mathrm{J}}( \hat{u}, \hat{\eta},Y), (w, \rho) \right \rangle_{J_t(u,\eta)} \,ds\,dt = 0,
\]
then $w$ and $\rho$ are identically zero. Taking $Y=0$, we see that $(w, \rho)$ is a weak solution to the equation $D_{v}^* (w, \rho)=0$, and hence by elliptic regularity they are of class $C^{\ell}$. 
Define 
\[
\Omega:=\left\{ (s,t)\in \mathcal{R}(v)\mid u(s,t)\notin \crit H\right\} .
\]
Since $K^- \ne K^+$, the set $\mathcal{R}(v)$ is open and dense in the open non-empty set of points $\left\{ (s,t) \mid \partial_s u (s,t) \ne 0 \right\}$ by Theorem \ref{regularpoints}. Moreover since $\Sigma = H^{-1}(0)$ is a regular energy level set and $\lim_{s \to \pm \infty}u(s, \cdot) \in \Sigma$, it follows that $\Omega$ is again an open non-empty set. Note also that since we assume $H$ is constant on $M \backslash M_0$, one has $M\backslash M_0 \subset \crit H$, and thus we are free to perturb $\mathrm{J}$ on $\Omega$ (recall outside of $M_0$ all elements of $\mathcal{J}_{\mathrm{con}}^{ \ell}$ are required to depend only on $x$ and be of contact type).  Finally, arguing exactly as in the proof of Theorem \ref{transcylinders}, we see that on $\Omega$ both $w$ and $\rho$ vanish. This completes the proof.

\subsection{The index of the hybrid problem}
\label{newindexcomp}

In this Section we compute the index of the hybrid problem $\mathcal{M}_{\Phi}(K; \Lambda)$ used to define the chain map $\Phi : RF(H ,f) \to F(\tilde{H },f)$.
\begin{thm}
\label{newindexcompthm}
The space $\mathcal{M}_{\Phi}(K ;\Lambda)$ has virtual dimension
\[
\dim \mathcal{M}_{\Phi}(K;\Lambda) =  \mu(K) + \dim K- \mu(\Lambda).
\]
\end{thm}
\begin{proof}
The proof of the above result uses the fact that the coupling condition \eqref{eq:loop coupling} and \eqref{eq:eta coupling} can be seen as a Lagrangian boundary condition for the map 
\[
\begin{split}
w & : ( - \infty, 0] \times \T   \to M \times M \times T^*\R, \\
w(s,t) & := \left( u^-(s,t), u^+( -s ,t), \eta^+(- s, t), \zeta^+(-s,t) \right).
\end{split}
\]
Indeed, the tuple $(u^-,\eta^-,u^+,\eta^+,\zeta^+)$ belongs to $\mathcal{M}_{\Phi}(K;\Lambda)$ if and only if $w$ satisfies a Floer-type equation with respect to the symplectic form $\omega \times ( - \tilde{\omega})$ coupled with an ODE for $\eta^-$, together with the boundary condition
\begin{equation}
\label{bdry}
w(0,t) \in \Delta_M \times \{\eta^-(0)\} \times \R^*,
\end{equation}
which is of Lagrangian type, and suitable asymptotic conditions. By this observation, the proof that $\mathcal{M}_{\Phi}(K ; \Lambda)$ is a Fredholm problem becomes standard and we do not present it here. Nevertheless, the index computation is complicated by the coupling of the Floer-type equation for $w$ with the ODE for $\eta^-$. Previous arguments \cite{cf09,bo13,mp11} for computing the index for Rabinowitz Floer problems have all made use of the spectral flow, and as such do not immediately apply in the situation we work with here, since our maps are defined on half-cylinders.

We will compute the index in the more difficult case where $K \ne \Sigma \times \{0 \}$ is not the set of constants. The case $K =\Sigma \times \{ 0 \}$ is easier and left to the reader. Moreover to slightly simplify the calculations that follow we will assume that the component $K$ is diffeomorphic to $\T$, and in addition that, writing $K = \left\{ (x( t + r) ,\tau) \mid  r \in \T \right\}$, $K$ is \emph{transverally non-degenerate}, which means that the algebraic multiplity of the eigenvalue $1$ of $d \phi_{X_H}^\tau(x(0))$ is exactly two.\\

As a first step, note that as far as the index computation is concerned, the precise form of the coupling condition \eqref{eq:eta coupling} is unimportant. Indeed, given $\kappa\in \R$, instead of requiring that \eqref{eq:eta coupling} holds, we could instead require that 
\[
\kappa  \eta^-(0) = \eta^+(0, t), \qquad \forall \, t \in \T.
\]
Any two such choices of $\kappa$ give rise to homotopic Fredholm problems, which therefore have the same index.  In particular, we may take $\kappa = 0$. This means that $\eta^-(0)$ is now a free boundary condition, and the loop $w(0,t)$ from \eqref{bdry} belongs to the standard Lagrangian subspace 
\begin{equation}
\label{newbdy}
w(0,t) \in \Delta_M \times \{ 0 \} \times \R^*.
\end{equation}
Let us begin by formulating an appropriate local model for the problem. Let 
\[
S \in W^{1, \infty}( (-  \infty,0] \times \T , \mathrm{L}(\R^{4n+2}))
\] 
denote a path of matrices extending to the compactification $[- \infty,0] \times \T$ in such a way that
\[
\lim_{s_0 \rightarrow- \infty}\underset{(s,t)\in(- \infty, s_0)\times\T} {\mathrm{ess\, sup}}\left(\left|\partial_{s}S(s,t)\right|+\left|\partial_{t}S(s,t)-\partial_{t}S(- \infty,t)\right|\right)=0.
\]
Assume in addition that the limit matrices $S(- \infty, t)$ can be writen as
\[
S(- \infty, t) = (S_1(t), S_2(t)) \in \mathrm{Sym}(\R^{2n}) \times \mathrm{Sym}(\R^{2n + 2}).
\] 
We denote by $W_1$ and $W_2$ the corresponding symplectic paths $W_1 : [0, 1] \to \mathrm{Symp}(\R^{2n}, \omega_n)$ and $W_2 : [0,1] \to \mathrm{Symp}(\R^{2n+2}, - \omega_n \times \omega_1)$, defined by 
\[
W_1'(t) = J_n S_1(t) W_1(t), \quad W_1(0) = I_n, \qquad W_2'(t) = ( -J_n \times J_1)S_2(t) W_2(t), \quad W_2(t) = I_{n+1}.
\]
Denote by $\mu_{\mathrm{rs}}(W_j) \in \tfrac{1}{2}\Z$ the Robbin-Salamon index of these paths. Next, let  $\beta \in W^{1, \infty}( (- \infty, 0] \times \T , \R^{2n})$ denote a vector-valued path extending to the compactification $[- \infty,0] \times \T$ in such a way that
 \[
\lim_{s_0 \rightarrow- \infty}\underset{(s,t)\in(- \infty, s_0)\times\T} {\mathrm{ess\, sup}}\left(\left|\partial_{s}\beta(s,t)\right|+\left|\partial_{t}B(s,t)-\partial_{t}\beta(- \infty,t)\right|\right)=0.
\]
Define $A_1 : = J_n \partial_t +S_1$, viewed as an unbounded linear operator
\[
A_1 : W^{1,2}(\T, \R^{2n}) \to L^2(\T, \R^{2n}),
\]
and abbreviate $\beta^-(t) := \beta(- \infty,t)$. Given a  number $c \in \R$, we says that the tuple $(A_1, \beta^-, c) $ is \emph{regular} if $\beta^- \in W^{1,2}(\T, \R^{2n}) \cap \mathrm{range}(A_1)$, and if $v \in W^{1,2}(\T, \R^{2n})$ is any vector such that $A_1v = \beta^-$, then the real number 
\begin{equation}
\label{lambda_A,beta}
\lambda_{A_1,\beta^-} := \int_{\T} \left\langle v(t), \beta^-(t) \right\rangle \,dt \ne c
\end{equation}
is not equal to $c$. This definition was introduced in \cite[Definition 2.1]{mp11} and is a minor variation on the earlier definition given by Cieliebak and Frauenfelder in \cite[Definition C.3]{cf09}.  
The number $\lambda_{A_1,\beta^-}$ is well-defined since $A_1$ is self-adjoint, see \cite[Definition 2.1]{mp11}.  We assume that the tuple $(A_1, \beta^- , 0)$ is regular. This guarantees the operator $D$ we write down below in \eqref{Dtriv} is a Fredholm problem, see \cite[Appendix C]{cf09}.

Let $\iota: \R^{2n} \hookrightarrow \R^{4n+2}$ denote the inclusion $\iota(x) = (x, 0)$, and let $\pi: \R^{4n+2} \to \R^{2n}$ denote the projection onto the first $2n$ coordinates. Fix $r > 2$ and let
\[
\begin{split}
W^{1,r}_{\delta, \Delta \times 0 \times \R}((-\infty ,0] \times \T &, \R^{4n + 2}) := \\ & \left\{ w \in W^{1,r}_{\delta} ((-\infty ,0] \times \T , \R^{4n + 2})  \mid w(0,t) \in \Delta_{\R^{2n}} \times \{0\} \times \R \right\},
\end{split}
\]
where the $\delta$ indicates that we are working with weighted Sobolev spaces (as our asymptotic operators will not be bijective).
With these preparations out the way, the approrpriate local model for our problem is the operator
\[
\begin{split}
D :W^{1,r}_{\delta, \Delta \times 0 \times \R}((-\infty ,0] \times \T , \R^{4n + 2}) &\times W^{1,r }_{\delta}(( - \infty, 0], \R) \longrightarrow \\
&L^r_{\delta}((- \infty,0] \times \T , \R^{4n +2})\times L^r_{\delta}(( - \infty, 0], \R)
\end{split}
\]
given by
\begin{equation}
\label{Dtriv}
D  
\begin{pmatrix}
w \\ \eta
\end{pmatrix}
:=
 \begin{pmatrix}
\partial_s w + J \partial_t w + A w +\eta \iota \beta\\  \eta' + \int_{\T} \left\langle \pi \circ w, \beta \right\rangle \,dt
\end{pmatrix}.
\end{equation}
More precisely, the linearisation of the problem $\mathcal{M}_{\Phi}(K ;\Lambda)$, when viewed in a suitable symplectic trivialisation, takes the form \eqref{Dtriv}. The fact that in such a trivialisation the tuple $(A_1, \beta^-, 0)$ is regular is explained\footnote{The only reason we have assumed that $K$ is transversally non-degenerate in the exposition here is that this is a standing assumption in \cite{mp11}. As far as these results are concerned this is a completely unnecessary restriction, and proofs in the general case can be found in \cite[Section 4]{cf09}, albeit with a slightly different formulation.}
 in detail in \cite[Section 2.4]{mp11}. In fact, since we are working with a hypersurface of restricted contact type, one can always choose the trivialisation so that the number $\lambda_{A_1, \beta^-}$ from \eqref{lambda_A,beta} has the same sign as $\tau$, see the discussion on \cite[p95]{mp11}, although we will not need this.

The idea now is to homotope the operator $D$ from \eqref{Dtriv} into a new operator  of product form $(w, \eta) \mapsto (D'(w), D''(\eta))$, i.e. one for which the equations are decoupled. Let us define $ S := J \partial_t + A$ and 
\[
 B : W^{1,r}_{\delta}( ( - \infty,0] \times \T, \R^{4n+2} ) \to W^{1,r}_{\delta}( (- \infty,0],\R), \qquad B(w):= \int_{\T} \left\langle \pi \circ w, \beta \right\rangle \, dt,  
 \] 
so the the operator $D$ from \eqref{Dtriv} can be concisely written as
\[
D = \partial_s +  \begin{pmatrix}
S & \iota \beta \\
B & 0
\end{pmatrix}.
\]
We will consider a homotopy $\{ D_{\theta }\}_{ \theta \in [0,1]}$ of operators of the form 
\[
D_{\theta} = \partial_s +  \begin{pmatrix}
S & (1 - \theta) \iota \beta \\
(1 - \theta) B & c( \theta)
\end{pmatrix},
\]
where $c : [0, 1] \to \R$ satisfies $c(0) = 0$.  The operators $D_{\theta}$ are all Fredholm of the same index, provided the tuple $(A_1, ( 1 - \theta) \beta^- , c(\theta))$ remains regular in the sense of \cite[Definition 2.1]{mp11}.  In other words, we must choose $c( \theta)$ so that
\[
\lambda_{A_1,(1 - \theta)\beta^-}  \ne c(\theta), \qquad \forall \, \theta \in [0,1],
\]
where the real number $\lambda_{A_1,(1 - \theta)\beta^-}$ is defined as in \eqref{lambda_A,beta}. Since we assumed our original tuple $(A_1, \beta^-, 0)$ was regular, such a choice can clearly always be made. Moreover one readily sees that
\begin{equation}
\label{signctheta}
\mathrm{sgn\,}c(1) = \mathrm{sgn\,}\lambda_{A_1, \beta^-}.
\end{equation}
The final operator $D_1$ is given by
\begin{equation}
\label{DNEWtriv}
D_1  
\begin{pmatrix}
w \\ \eta
\end{pmatrix}
:=
 \begin{pmatrix}
\partial_s w + J \partial_t w + A w \\  \eta' + c(1)\eta
\end{pmatrix},
\end{equation}
and hence
\begin{equation}
\label{indexsum}
\mathrm{ind\,}D_1 = \mathrm{ind\,}D' + \mathrm{ind\,}D'',
\end{equation}
where 
\[
D'  : W^{1,r}_{\delta, \Delta \times 0 \times \R}((-\infty ,0] \times \T , \R^{4n + 2})  \to
L^r_{\delta}((- \infty,0] \times \T , \R^{4n +2}) ,
\]
\[
D'(w)  : = \partial_s w + J \partial_t w + A w,
\]
and
\[
D'' : W^{1,r}( (- \infty, 0], \R ) \to L^r( (- \infty,0 ], \R),
\]
\[
D''(\eta) : = \eta' + c(1)\eta.
\]
The index of $D'$ is given by 
\begin{equation}
\label{indexD'}
\mathrm{ind\,}D' = \mu_{\mathrm{rs}}(W_1) - \tfrac{1}{2}\nu(W_1)+ \mu_{\mathrm{rs}}(W_2)  - \tfrac{1}{2}\nu(W_2)+ k,
\end{equation}
where the correction term $k$ is determined by the boundary condition \eqref{newbdy}, and
\[
\nu(W_1) = \dim \ker (W_1(1) - I_n), \qquad \nu(W_2) = \dim \ker (W_2(1) - I_{n+1}).
 \] 
 Note that the assumption $K$ is transversally non-degenerate translates implies that $\nu(W_1) = 1$.
 The correction term $k$ is computed in \cite[Theorem 5.24]{as10} to be
\[
k = \frac{m}{2}  - \frac{1}{2} \bigl( \dim W_0 + 2 \dim V_0 - 2 \dim W_0 \cap (V_0\times V_0) \bigr),
\]
where $2m$ is the dimension of the ambient space $M \times M \times T^* \R$, $W_0$ denotes the diagonal in $(\R^n \times \R^n \times \R)^2$ and $V_0$ is the subspace $\Delta_{\R^n} \times (0)$ of $\R^n \times \R^n \times \R$, where $\Delta_{\R^n}$ is the diagonal in $\R^n \times \R^n$. Since 
\[
m = 2n +1, \qquad \dim W_0 = 2n + 1, \qquad \dim V_0 = n, \qquad \dim W_0 \cap (V_0\times V_0) = n,
\]
the correction term $k$ vanishes. Meanwhile - provided $|c(1) |$ is sufficiently small, given that we are working with weighted Sobolev spaces - one easily sees that
\begin{equation}
\label{indexD''}
\mathrm{ind\,}D'' = \begin{cases}
0, & c(1) > 0, \\
1, & c(1) < 0.
\end{cases}
\end{equation}
By definition (recall we reversed the maps $(u^+, \eta^+,\zeta^+)$ in the definition of $w$), one has
\begin{equation}
\label{indexLambda}
- \mu(\Lambda) = \mu_{\mathrm{rs}}(W_2) + \tfrac{1}{2}\nu(W_2).
\end{equation}
Finally, in this formulation the main result of \cite[Theorem 1.10]{mp11} tells us that
\begin{equation}
\label{indexK}
\mu(K) = \begin{cases}
\mu_{\mathrm{rs}}(W_1) - \tfrac{1}{2}\nu(W_1), & \lambda_{A_1 , \beta^-} > 0, \\
\mu_{\mathrm{rs}}(W_1) - \tfrac{1}{2}\nu(W_1) + 1, &  \lambda_{A_1 , \beta^-} < 0. 
\end{cases}
\end{equation}
Putting \eqref{signctheta} and \eqref{indexsum}-\eqref{indexK} all together, we conclude that
\[
\mathrm{ind\,}D = \mu(K) - \mu(\Lambda) - \dim \,\Lambda,
\]
and hence
\[
\dim\,\mathcal{M}_{\Phi}(K ;\Lambda) = \mathrm{ind\,}D + \dim K +
\dim \Lambda =\mu(K) + \dim K- \mu(\Lambda).  
\]
This completes the proof.
\end{proof}

\subsection{Automatic transversality at stationary solutions of the hybrid problem}
\label{auto_trans_sec}
If the space $\mathcal{M}_ \Phi( \Lambda ;K)$ contains no zero-energy solutions then the proof that generically $\mathcal{M}_ \Phi(\Lambda; K)$ carries the structure of a smooth manifold contains no ideas not already present in the proof of Theorem \ref{transcylinders} and Theorem \ref{transcylindersRF}. Therefore we omit the details. However we do still need to prove the automatic transversality result used in Section \ref{theisomorphismsection}. First we will need the following simple lemma.
\begin{lem}
\label{lem:equality of Hessians}
Suppose $ \hat{x} = (x,\tau)\in\mathrm{crit}\,\mathcal{A}_{H}$.
Let $v\in C^{\infty}(\mathbb{T},x^{*}TM)$ and $\rho\in\mathbb{R}$.
Let $\tilde{x}(t):=(x(t),\tau, \sigma)$ for some $ \sigma \in \R^*$. Then for any $\xi\in C^{\infty}(\mathbb{T},\mathbb{R}^{*})$,
if we define $w(t):=(v(t),\rho,\xi(t))\in C^{\infty}(\mathbb{T},\tilde{x}^{*}T\tilde{M})$
then 
\[
d^{2}\mathcal{A}_{H}(\hat{x})[(v,\rho),(v,\rho)]=d^{2}\mathbb{A}_{\tilde{H}}(\tilde{x})[w,w].
\]
\end{lem}
\begin{proof}
One has 
\begin{align*}
d^{2}\mathbb{A}_{\tilde{H}}(\tilde{x})[w,w] & =\left\langle \left\langle \nabla^{2}\mathbb{A}_{\tilde{H}}(\tilde{x})[w],w\right\rangle \right\rangle \\
 & =\int_{\mathbb{T}}\omega(v,\nabla_{t}v + (\nabla_v J + \rho \partial_{\rho}J) x'-\tau\nabla_{v}X_{H}(x)-\rho X_{H}(x))dt\\
 & +\int_{\mathbb{T}}\rho(\xi'-dH(x)[v])dt-\int_{\mathbb{T}}\rho'\xi dt.
\end{align*}
 Now since $\rho$ is constant, clearly $\int_{\mathbb{T}}\rho'\xi dt=0$.
Since $\xi$ is a loop, $\int_{\mathbb{T}}\rho\xi'dt=0$, and thus
\begin{align*}
d^{2}\mathbb{A}_{\tilde{H}}(\tilde{x})[w,w] & =\int_{\mathbb{T}}\omega(v,\nabla_{t}v+ (\nabla_v J + \rho \partial_{\rho}J) x'-\tau\nabla_{v}X_{H}(x)-\rho X_{H}(x))dt-\rho\int_{\mathbb{T}}dH(x)[v])dt\\
 & =d^{2}\mathcal{A}_{H}( \hat{x})[(v,\rho),(v,\rho)].
\end{align*}

\end{proof}
Here is the promised automatic transversality result. 
\begin{lem}
\label{autotrans}
Suppose $ \hat{x} \in\mathrm{crit}\, f$. Let $K \subset \crit \mathcal{A}_H$ denote the connected component containing $ \hat{x}$, and let $ \Lambda = \pi_{\mathcal{K}}^{-1}( K) \subset \crit \A_{ \tilde{H}}$. Then automatic transversality
holds for the problem $\mathcal{M}_{\Phi}(\hat{x}; \hat{x})$. More
precisely $\mathcal{M}_{\Phi}(\hat{x}; \hat{x})$ contains only the
stationary solutions: 
\[
\mathcal{M}_{\Phi}(\hat{x}; \hat{x}) =\left\{ (\hat{x}, (\hat{x}, \sigma )) \mid \sigma \in \R^* )\right\} \subset\mathcal{M}_{\Phi}(K; \Lambda),
\]
and if $D_{ \sigma} $ denotes the linearisation of the problem $\mathcal{M}_{\Phi}(\hat{x};\hat{x}) $
at the stationary solution $(\hat{x}, (\hat{x}, \sigma)) $ then the operator
$D_{ \sigma} $ is surjective.\end{lem}
\begin{proof}
We already know that $D_{ \sigma}$ is Fredholm of index zero, and hence it
suffices to show that $D_{ \sigma}$ is injective. Suppose we are given 
\[
v:(-\infty,0]\times\mathbb{T}\rightarrow T_{x(t)}M,\qquad u:[0,+\infty)\times\mathbb{T}\rightarrow T_{x(t)}M,
\]
\[
\xi:(-\infty,0]\rightarrow\mathbb{R},\qquad\eta:[0,+\infty)\times\mathbb{T}\rightarrow\mathbb{R},\qquad\zeta:[0,+\infty)\times\mathbb{T}\rightarrow\mathbb{R}^{*}.
\]
Set 
\[
w=(v,\xi),\qquad z=(u,\eta,\zeta).
\]
Then the operator $D_{ \sigma}$ sends $(w,z)$ to the solution of the linear
problem:
\begin{equation}
\label{autot1}
\begin{split}
\frac{d}{ds}w+\nabla^{2}\mathcal{A}_{H}(\hat{x})[w]=0,\qquad\frac{d}{ds}z+\nabla^{2}\mathbb{A}_{\tilde{H}}( \hat{x}, \sigma)[z]=0,\\
\lim_{s\rightarrow-\infty}w(s)\in T_{\hat{x}}W^{u}_{-\nabla f}(\hat{x}),\qquad\lim_{s\rightarrow+\infty}z(s)\in T_{(\hat{x},\sigma)}W^s_{-\nabla \tilde{f}} (\pi_{\mathcal{K}}^{-1}(\hat{x}),\\
\end{split}
\end{equation}
\begin{equation}
\label{autot2}
\begin{split}
v(0,t)=u(0,t),\qquad\textrm{for all }t\in\mathbb{T},\\
\xi(0)=\eta(0,t),\qquad\textrm{for all }t\in\mathbb{T}.
\end{split}
\end{equation}
Recall that $ \tilde{f} = f \circ \pi_{\mathcal{K}}$ was defined in \eqref{tildef}. We must show that any such solution of \eqref{autot1} and \eqref{autot2} is identically zero. For this
one considers the function 
\[
\varphi(s):=\left\Vert z(s,\cdot)\right\Vert ^{2}.
\]
Then 
\[
\varphi'(s)=-2\left\langle \left\langle \nabla^{2}\mathbb{A}_{\tilde{H}}(z(s,\cdot)),z(s,\cdot)\right\rangle \right\rangle ,
\]
\[
\varphi''(s)=4\left\Vert \nabla^{2}\mathbb{A}_{\tilde{H}}(z(s,\cdot))\right\Vert ^{2}\geq0.
\]
Thus $\varphi$ is convex, and thus either $\varphi\equiv0$ or $\varphi'(0)<0$.
Thus either $z\equiv0$ or 
\[
d^{2}\mathbb{A}_{\tilde{H}}(\hat{x},\sigma)[z(0,\cdot),z(0,\cdot)]>0.
\]
Exactly the same argument shows that either $w\equiv0$ or 
\[
d^{2}\mathcal{A}_{H}(\hat{x})[w(0,\cdot),w(0,\cdot)]<0.
\]
Lemma \ref{lem:equality of Hessians} thus implies that $w\equiv z\equiv0$
as required.\end{proof}



\providecommand{\bysame}{\leavevmode\hbox to3em{\hrulefill}\thinspace}
\providecommand{\MR}{\relax\ifhmode\unskip\space\fi MR }
\providecommand{\MRhref}[2]{%
  \href{http://www.ams.org/mathscinet-getitem?mr=#1}{#2}
}
\providecommand{\href}[2]{#2}

\end{document}